\documentclass[12pt,reqno]{amsart}
\usepackage{amsfonts,amsmath,amssymb,amsthm,mathrsfs}
\usepackage{enumerate,hyperref,indentfirst,url}
\usepackage{graphicx}
\usepackage{cite}
\usepackage{xcolor}

\numberwithin{equation}{section}
\newtheorem{definition}{Definition}

\newtheorem{theorem}{Theorem}[section]
\newtheorem{lemma}[theorem]{Lemma}
\newtheorem{proposition}{Proposition}[section]

\newtheorem{remark}{Remark}

\newcommand{\td}{{\tilde \chi_{I^0}}}
\newcommand{\tdd}{{\tilde \chi_{R_{00}}}}
\newcommand{\dst}{\text{dist}}

\newcommand{\N}{\mathbb{R}^n}
\newcommand{\n}{\mathbb{R}}
\newcommand{\nn}{\mathbb{R}^2}
\newcommand{\nnn}{\mathbb{R}^3}
\newcommand{\one}{\uppercase\expandafter{\romannumeral1}}
\newcommand{\two}{\uppercase\expandafter{\romannumeral2}}
\newcommand{\three}{\uppercase\expandafter{\romannumeral3}}
\newcommand{\four}{\uppercase\expandafter{\romannumeral4}}
\newcommand{\spt}{\text{supp} \, }
\newcommand{\sch}{\mathcal{S}(\n)}
\newcommand{\z}{\mathbb{Z}}
\allowdisplaybreaks
\begin{document}

\title[Bi-parameter trilinear Fourier multipliers  with flag symbol]{Bi-parameter trilinear Fourier multipliers and pseudo-differential operators  with flag symbols }

\author{Guozhen Lu  }
\address{Department of Mathematics\\
  University of Connecticut\\
  Storrs, CT 06269}
\email[G. Lu]{guozhen.lu@uconn.edu }
\author{Jill Pipher}
 \address{Department of Mathematics\\
  Brown University\\
  Providence, RI 02912}
\email[J. Pipher]{jpipher@math.brown.edu}
\author{Lu Zhang}
\address{Department of Mathematical Sciences\\
  Binghamton University\\
  Binghamton, NY 13902}
\email[L. Zhang]{lzh@math.binghamton.edu}


\thanks{The first author's research was partly supported by  a collaboration grant from the Simons Foundation.}

\begin{abstract}
The main purpose of this paper is to study $L^r$ H\"older type estimates for a bi-parameter trilinear Fourier multiplier with flag singularity, and the analogous pseudo-differential operator, when the symbols are in a certain product form. More precisely,
 for  $f,g,h\in \mathcal{S}(\n^{2})$, the bi-parameter trilinear flag Fourier multiplier operators we consider are defined by
\begin{equation*}
T_{m_1,m_2}(f,g,h)(x):=\int_{\n^{6}}m_1(\xi,\eta,\zeta)m_2(\eta,\zeta)\hat f(\xi) \hat g(\eta)\hat h(\zeta)e^{2\pi i(\xi+\eta+\zeta)\cdot x}d\xi d\eta d\zeta,
\end{equation*}
 when $m_1,m_2$ are two bi-parameter symbols. We study H\"older type estimates:  $L^{p_1}\times L^{p_2}\times L^{p_3} \to L^r$ for $1<p_1,p_2,p_3< \infty$ with $1/p_1+1/p_2+1/p_3=1/r$, and $0<r<\infty$. We will show that our problem can be reduced  to establish the $L^r$ estimate for the special multiplier $m_1(\xi_1, \eta_1, \zeta_1) m_2(\eta_2, \zeta_2)$ (see Theorem 1.7).

We also study these $L^r$ estimates for
the corresponding bi-parameter trilinear pseudo-differential operators defined by
\begin{equation*}
T_{ab}(f,g,h)(x):=\int_{\n^6}a(x,\xi,\eta,\zeta)b(x,\eta,\zeta)\hat f(\xi)\hat g(\eta)\hat h(\zeta)e^{2\pi i x(\xi+\eta+\zeta)}d\xi d\eta d\zeta,
\end{equation*}
 where  the smooth symbols $a,b$ satisfy certain  bi-parameter H\"ormander conditions. We will also show that the $L^r$ estimate holds for $T_{ab}$ as long as the $L^r$ estimate for the flag multiplier operator holds when the multiplier has the special form $m_1(\xi_1, \eta_1, \zeta_1) m_2(\eta_2, \zeta_2)$ (see Theorem 1.10). Using our reduction of the flag multiplier, we also provide an alternative proof of some of the mixed norm estimates recently established by Muscalu and Zhai \cite{muscaluZhai} when the functions g and h are of tensor product forms (Theorem 1.8).
 Moreover, our method also allow us to establish the weighted mixed norm estimates (Theorem 1.9).

 The bi-parameter and trilinear flag Fourier multipliers considered in this paper do not satisfy the conditions of the classical bi-parameter trilinear Fourier multipliers considered by Muscalu,  Tao, Thiele and the second author \cite{muscalu2004bi, muscalu2004multi}. They may also be viewed as the
 bi-parameter trilinear variants of estimates obtained for the one-parameter flag paraproducts by  Muscalu  \cite{muscalu2007paraproducts}.

\end{abstract}
\maketitle

\section{Introduction}

For $n\geq 1$ we denote by $\mathcal{M}(\N)$ the set of all bounded symbols $m\in L^\infty(\N)$, smooth away from the origin and satisfying the classical Marcinkiewcz-Mikhlin-H\"ormander condition
\begin{equation}
\label{coifmansymbol}
|\partial^\alpha m(\xi)|\lesssim \frac{1}{|\xi|^\alpha}
\end{equation}
for every $\xi\in \N \backslash\{0\}$ and sufficiently many multi-indices $\alpha$. Denote by $T_m$  the n-linear operator
\begin{equation}
\label{coifmanop}
T_m(f_1,\dots,f_n)(x):=\int_{\n^n}m(\xi)\hat f_1(\xi_1)\cdots \hat f_n(\xi_n)e^{2\pi i(\xi_1+\cdots+\xi_n)\cdot x}d\xi,
\end{equation}
where $\xi=(\xi_1,\dots,\xi_n)\in \N$ and $f_1,\dots,f_n$ are Schwartz functions on $\n$, denoted by $\mathcal{S}(\n)$. From the classical Coifman-Meyer theorem we know $T$ extends to a bounded n-linear operator from $L^{p_1}(\n)\times \cdots \times L^{p_n}(\n)$ to $L^{r}(\n)$ for $1<p_1,\dots,p_n\leq \infty$ and $1/p_1+\cdots+1/p_n=1/r>0$. In fact this property holds in higher dimensions when $f_i\in L^{p_i}(\mathbb{R}^d),\,i=1,\dots,n$ and $m\in \mathcal{M}{(\n^{nd})}$, see \cite{coifman1975commutators,grafakos2002multilinear,kenig1999multilinear}. The case $p\geq 1$ was proved by Coifman and Meyer \cite{coifman1975commutators} and was extended to $p<1$ by Grafakos and Torres \cite{grafakos2002multilinear}, and Kenig and Stein \cite{kenig1999multilinear}.

For the corresponding pseudo-differential variant of the classical Coifman-Meyer theorem, suppose that the symbol $\sigma(x,\xi)$ belongs to the  H\"ormander symbol class $S^0_{1,0}(\mathbb{R}\times \mathbb{R}^n)$; that is, $\sigma$ satisfies the condition
\begin{equation}
\label{bscd}|
 \partial_x^l\partial_\xi^\alpha\sigma(x,\xi)|\lesssim {1\over (1+|\xi|)^{|\alpha|}}
  \end{equation}
  for any $x \in\n$, $\xi=(\xi_1,\dots,\xi_n)\in \mathbb{R}^n$ and all indices $l$, $\alpha$ . For these symbols,
  the following multi-linear, single parameter case has been studied.
\begin{theorem}[\cite{muscalu2013classical,grafakos2002multilinear}]
\label{pseth}
The operator
\begin{equation*}
T_\sigma(f_1,\dots,f_n)(x):=\int_{\N}\sigma(x,\xi)\hat f_1(\xi_1)\cdots f_n(\xi_n)e^{2\pi i(\xi_1+\cdots+\xi_n)\cdot x}d\xi
\end{equation*}
is bounded from $L^{p_1}(\n)\times \cdots \times L^{p_n}(\n)$ to $L^{r}(\n)$ for $1<p_1,\dots,p_n\leq \infty$ and $1/p_1+\cdots+1/p_n=1/r>0$, where $f_1,\dots,f_n \in \mathcal{S}(\n)$ and $\sigma$ satisfies $\eqref{bscd}.$ Again, this result still hold if the functions are defined on $\mathbb{R}^d$.
 \end{theorem}

We now consider the multi-parameter setting of the above operators, introduced and studied via time-frequency analysis in \cite{muscalu2004bi,muscalu2013classical,muscalu2004multi}. For simplicity, we just state the bi-linear, bi-parameter case when $f,g$ are defined on $\mathbb{R}^2$. The results extend to the $n$-linear, d-parameter case where $f_1,\dots, f_n$ are defined on $\mathbb{R}^d$. We denote by $m\in \mathcal{BM}(\mathbb{R}^4)$ the set of smooth bi-parameter symbols satisfying
   \begin{equation}
     \label{bim2}
 |\partial_{\xi_1,\xi_2}^{\alpha_1,\alpha_2}\partial_{\eta_1,\eta_2}^{\beta_1,\beta_2}  m(\xi,\eta)|\lesssim \prod_{i=1}^2{1\over (|\xi_i|+|\eta_i|)^{\alpha_i+\beta_i}},
  \end{equation}
   for any  $\xi=(\xi_1,\xi_2),\eta=(\eta_1,\eta_2)\in \n^{2}\setminus \{0\}$ and sufficiently many multi-indices $\alpha=(\alpha_1,\alpha_2),\beta=(\beta_1,\beta_2)$

 \begin{theorem}[\cite{muscalu2004bi,muscalu2013classical,muscalu2004multi}]
 \label{bipcoifman}
   Let $1<p,q\leq \infty$, $1/p+1/q=1/r$, $0<r<\infty$ and $m\in \mathcal{BM}(\mathbb{R}^4)$, then the operator
   \begin{equation}
   \label{bicofop}
   T_m(f,g)(x)=\int_{\n^{4}} m(\xi_1,\xi_2,\eta_1,\eta_2) e^{2\pi i x(\xi+\eta)} \hat f(\xi_1,\xi_2)\hat g(\eta_1,\eta_2)d\xi d\eta
   \end{equation}
   is bounded from $L^p(\n^{2})\times L^q(\n^{2}) \to L^r(\n^{2}).$
 \end{theorem}

 A H\"{o}rmander type multiplier theorem with limited smoothness on the multi-parameter and multilinear multipliers was obtained in \cite{chen2014hormander}.

 \begin{theorem}[\cite{chen2014hormander}]
   Let  $m\in C^{2d+1} (\n^{2d}\setminus \{0\}\times \n^{2d}\setminus \{0\})$ satisfy $\eqref{bim2}$ for all $|\alpha_1|+|\beta_1|\leq d+1$, $|\alpha_2|+|\beta_2|\leq d+1$ and $(\xi_1,\xi_2,\eta_1,\eta_2)\in (\n^{2d}\setminus\{0\}\times \n^{2d}\setminus\{0\})$. Then $T_m$ defined in $\eqref{bicofop}$    is bounded from $L^p(\n^{2d})\times L^q(\n^{2d}) \to L^r(\n^{2d})$ for $1<p,q<\infty$, $1/p+1/q=1/r$, $0<r<\infty$.
 \end{theorem}

\medskip

The corresponding bi-parameter pseudo-differential operator was studied in  \cite{dai2013p}.

 \begin{theorem}[\cite{dai2013p}]
 \label{multipseudoth}
 Define
  $$T_a (f,g)(x):=\int_{\n^{4}} a(x,\xi,\eta)\hat f(\xi_1,\xi_2) \hat g(\eta_1,\eta_2)e^{2\pi i x\cdot(\xi+\eta)}d\xi d\eta$$
  where
  $$|\partial_{x_1}^{l_1} \partial_{x_2}^{l_2}\partial_{\xi_1}^{\alpha_1} \partial_{\xi_2}^{\alpha_2}\partial_{\eta_1}^{\beta_1} \partial_{\eta_2}^{\beta_2}a(x,\xi,\eta)|\lesssim {1\over (1+|\xi_1|+|\eta_1|)^{\alpha_1+\beta_1}} {1\over (1+|\xi_2|+|\eta_2|)^{\alpha_2+\beta_2}}.$$
  Then $T_a$ is bounded on $L^{p_1}\times L^{p_2}\to L^r$provided that $1<p_1,p_2\leq \infty$ and ${1\over r}={1\over p_1}+{1\over p_2}>0$.
 \end{theorem}

In particular, in the proof of trilinear bi-parameter version of Theorem $\ref{multipseudoth}$ above, the following localized $L^r$ estimates  hold and these estimates will also play a role in our current paper.

\begin{theorem}[\cite{dai2013p}]
\label{trilocalth}
  Let $m(\xi,\eta,\zeta)$ be a smooth symbol satisfying
  \begin{equation*}
 |\partial_{\xi_1,\xi_2}^{\alpha_1,\alpha_2}\partial_{\eta_1,\eta_2}^{\beta_1,\beta_2} \partial_{\zeta_1,\zeta_2}^{\gamma_1,\gamma_2} m(\xi,\eta,\zeta)|\lesssim \prod_{i=1}^2{1\over (1+|\xi_i|+|\eta_i|+|\zeta_i|)^{\alpha_i+\beta_i+\gamma_i}}
  \end{equation*}
  for any  $\xi=(\xi_1,\xi_2),\eta=(\eta_1,\eta_2),\zeta=(\zeta_1,\zeta_2)\in \nn$ and sufficiently many multi-indices $\alpha=(\alpha_1,\alpha_2),\beta=(\beta_1,\beta_2),\gamma=(\gamma_1,\gamma_2)$.
  And define the operator
  \begin{eqnarray*}
    T_m^0(f,g,h)(x_1,x_2):=\left(\int_{\n^6} m(\xi,\eta,\zeta) e^{2\pi i x(\xi+\eta+\zeta)} \hat f(\xi)\hat g(\eta)\hat h(\zeta)d\xi d\eta d\zeta\right)\varphi_0(x_1)\varphi_0'(x_2),
  \end{eqnarray*}
  where $\varphi_0,\varphi_0'\in \mathcal{S}(\n)$ are supported on $I^0=[-1,1]$. Then for $1<p_1,p_2,p_3\leq \infty$ and $1/p_1+1/p_2+1/p_3=1/r>0$  there holds
 \begin{eqnarray*}
   \|T_m^0(f,g,h)(x_1,x_2)\|_{L^r}\lesssim \|f\tilde \chi_0\otimes \tilde \chi_0 \|_{p_1} \|g \tilde \chi_0 \otimes\tilde \chi_0\|_{p_2} \|h \tilde \chi_0 \otimes \tilde \chi_0\|_{p_3},
 \end{eqnarray*}
  where $\tilde \chi_0:=\tilde \chi_{I^0}$ is defined as $\eqref{apc}$.
\end{theorem}

\medskip

  We now return to the discussion of the classical single-parameter Coifman-Meyer type operator $\eqref{coifmanop}$ under the condition
   $\eqref{coifmansymbol}$. Note that in $\eqref{coifmansymbol}$ the only singularity for the symbol $m$ is at the origin.  In  \cite{muscalu2007paraproducts}, Muscalu considered some types of symbols having flag singularities. More precisely, in the trilinear case, the symbol $m(\xi,\eta,\zeta)$ is a product of two symbols in $\mathcal{M}(\n^3)$ and $\mathcal{M}(\n^2)$ respectively, i.e, $m(\xi,\eta,\zeta)=m_1(\xi,\eta,\zeta)m_2(\eta,\zeta)$ for $m_1\in \mathcal{M}(\nnn), m_2\in \mathcal{M}(\nn)$ satisfying
\begin{gather*}
 \nonumber |\partial_\xi^\alpha \partial_\eta^\beta \partial_\zeta^\gamma m_1(\xi,\eta,\zeta)|\lesssim {1\over (|\xi|+|\eta|+|\zeta|)^{\alpha+\beta+\gamma}}\\
 |\partial_\eta^\beta \partial_\zeta^\gamma m_2(\eta,\zeta)|\lesssim {1\over (|\eta|+|\zeta|)^{\beta+\gamma}}
\end{gather*}
for every $\xi,\eta,\zeta\in \n$ and sufficiently many indices $\alpha$, $\beta$ and $\gamma$. Define
\begin{equation}
\label{flagop}
T_{m_1,m_2}(f,g,h)(x):=\int_{\nnn}m_1(\xi,\eta,\zeta)m_2(\eta,\zeta)\hat f(\xi) \hat g(\eta)\hat h(\zeta)e^{2\pi i(\xi+\eta+\zeta)\cdot x}d\xi d\eta d\zeta,
\end{equation}
where $f,g,h\in \mathcal{S}(\n)$. Then there holds
\begin{theorem}
\emph{(\cite{muscalu2007paraproducts})}
\label{Muscalu}
The operator defined in $\eqref{flagop}$ maps $L^{p_1}\times L^{p_2}\times L^{p_3}\to L^r$ for $1<p_1,p_2,p_3\leq \infty$ with $1/p_1+1/p_2+1/p_3=1/r$, and $0<r<\infty$. In particular, the boundedness $L^{\infty}\times L^{\infty}\times L^{p_3}\to L^{p_3}$, $L^{\infty}\times L^{p_2}\times L^{\infty}\to L^{p_2}$, $L^{p_1}\times L^{\infty}\times L^{p_3}\to L^r$ and $L^{p_1}\times L^{p_2}\times L^{\infty}\to L^r$ are true.
\end{theorem}

Moreover, for the above theorem, the estimates like $L^t\times L^{\infty}\times L^{\infty}  \to L^t$ or $L^{\infty}\times L^{\infty} \times L^{\infty}\to L^{\infty}$ are false, and these can be checked  by setting one of the symbols to be identically 1. Moreover, the method in \cite{muscalu2007paraproducts} can be applied when studying the adjoints of those operators. Also, when $0<p_1,p_2,p_3\leq \infty$, Miyachi and Tomita in \cite{MT} proved the boundedness of \eqref{flagop} on Hardy and BMO spaces.

\medskip

 The main purpose of this paper is  to study  the  $L^r$ estimates for the bi-parameter trilinear Fourier multipliers  with flag singularity as defined in  $\eqref{Muscalu}$, as well as the corresponding bi-parameter trilinear  pseudo-differential variants. We consider the multipliers $m_1\in \mathcal{BM}(\mathbb{R}^6)$ and $m_2\in \mathcal{BM}(\mathbb{R}^4)$ satisfying the following conditions:
\begin{gather*}
 \nonumber |\partial_{\xi_1}^{\alpha_1} \partial_{\xi_2}^{\alpha_2}\partial_{\eta_1}^{\beta_1} \partial_{\eta_2}^{\beta_2} \partial_{\zeta_1}^{\gamma_1} \partial_{\zeta_2}^{\gamma_2} m_1(\xi,\eta,\zeta)| \nonumber \\
 \lesssim {1\over (|\xi_1|+|\eta_1|+|\zeta_1|)^{\alpha_1+\beta_1+\gamma_1}} {1\over (|\xi_2|+|\eta_2|+|\zeta_2|)^{\alpha_2+\beta_2+\gamma_2}},\\
 \label{bimultiplier}
 |\partial_{\eta_1}^{\beta_1} \partial_{\eta_2}^{\beta_2} \partial_{\zeta_1}^{\gamma_1} \partial_{\zeta_2}^{\gamma_2} m_2(\eta,\zeta)|\lesssim {1\over (|\eta_1|+|\zeta_1|)^{\beta_1+\gamma_1}} {1\over (|\eta_2|+|\zeta_2|)^{\beta_2+\gamma_2}},
\end{gather*}
for every $\xi=(\xi_1,\xi_2),\eta=(\eta_1,\eta_2),\zeta=(\zeta_1,\zeta_2)\in \n \times \n$ and  all multi-indices $\alpha=(\alpha_1,\alpha_2)$, $\beta=(\beta_1,\beta_2)$ and $\gamma=(\gamma_1,\gamma_2)$.

Our main theorems are as follows.

\begin{theorem}
\label{bithm}
For  $f,g,h\in \mathcal{S}(\n^{2})$, the bi-parameter operators
\begin{equation}
\label{biop34}
T_{m_1,m_2}(f,g,h)(x):=\int_{\n^{6}}m_1(\xi,\eta,\zeta)m_2(\eta,\zeta)\hat f(\xi) \hat g(\eta)\hat h(\zeta)e^{2\pi i(\xi+\eta+\zeta)\cdot x}d\xi d\eta d\zeta
\end{equation}
would map $L^{p_1}\times L^{p_2}\times L^{p_3}\to L^r$ for $1<p_1,p_2,p_3 < \infty$ with $1/p_1+1/p_2+1/p_3=1/r$ and $0<r<\infty$, if one assumes that the operators
\begin{equation}
\label{assumption}
\int_{\n^{6}}m'(\xi_1,\eta_1,\zeta_1)m''(\eta_2,\zeta_2)\hat f(\xi) \hat g(\eta)\hat h(\zeta)e^{2\pi i(\xi+\eta+\zeta)\cdot x}d\xi d\eta d\zeta,
\end{equation}
satisfy the same H\"older type estimates, where $m' \in \mathcal{M}(\mathbb{R}^3)$ and $m'' \in \mathcal{M}(\mathbb{R}^2)$ are two one-parameter symbols.
\end{theorem}
In fact, we will reduce \eqref{biop34} to a sum of Fourier multipliers, each with different types of symbols; for some of these operators we prove H\"older type estimates, and the remaining operators are treated under an additional assumption about symbols of form \eqref{assumption}.
While the helicoidal method of \cite{BMu2} can be used to treat certain of the operators in our reduction, it does not give the
boundedness of
\eqref{assumption}. We conjecture that H\"older-type estimates for these operators are true.
More details are in Section \ref{bithmpf} and Remark \ref{rmk1}.

The above theorem indicates that providing estimates for the operator
  \eqref{biop34} are the fundamentally new obstacles in obtaining estimates for the operator  \eqref{assumption}.
In fact, the proof of this theorem shows that the study of the bi-parameter flag multiplier \eqref{biop34} can be essentially
reduced to the study of classical bi-parameter multilinear Fourier multipliers like \eqref{bicofop} as well as the multipliers
with tensor product symbols like \eqref{assumption}.

As it turns out, estimates for the operator  \eqref{biop34} are of interest even in restricted
function spaces. Recently, Muscalu and Zhai (\cite{muscaluZhai}), and see also \cite{Zhai})  proved the following estimates for \eqref{biop34} under a certain tensor product assumption.
\begin{theorem}[\cite{muscaluZhai}]\label{muscaluZhai}
  Let $g(x)=g_1(x_1)\otimes g_2(x_2) $, $h(x)=h_1(x_1)\otimes h_2(x_2) $,  ${1\over p}+ {1\over p_2}+{1\over p_3}={1\over p}+{1\over q_2}+{1\over q_3}={1\over r}$, and $0<r<\infty$. Then \eqref{biop34} maps
  \begin{itemize}
    \item[(a)] $L^p \times L^{p_2}_{x_1}(L^{q_2}_{x_2})\times L^{p_3}_{x_1}(L^{q_3}_{x_2}) \to L^r$, for $1<p,p_2,p_3,q_2,q_3\leq\infty $ and $(p_2,p_3)\neq (\infty,\infty)$ and $(q_2,q_3) \neq (\infty,\infty)$.
    \item[(b)] $L^p \times L^\infty \times L^{p_3} \to L^r$  and $L^p  \times L^{p_2} \times L^\infty\to L^r$, for $1<p\leq \infty$ and $1<p_2,p_3<\infty$.
  \end{itemize}.
\end{theorem}
The methods of \cite{muscaluZhai} used flag paraproducts and some novel and careful stopping-time arguments. Note that when $1<p,p_2,q_2,p_3,q_3<\infty$, our reduction of \eqref{biop34} essentially into the sum of bi-parameter trilinear multipliers (namely, trilinear version of \eqref{bicofop}), and our reduced multipliers $m'(\xi_1,\eta_1,\zeta_1)m''(\eta_2,\zeta_2)$ as in \eqref{assumption} can be used to offer an alternative proof for some of these mixed norm estimates as above under the tensor product setting.   We include some details in Appendix \ref{ap2}. To obtain the mixed norm estimates involving $L^\infty$, it seems that further ideas are required and that the arguments of \cite{muscaluZhai} are essential.

\medskip

In fact, our argument, which gives an alternative  proof of the above Theorem \ref{muscaluZhai}, also allows us to establish the following  weighted mixed norm estimates whose proof will be given in Appendix \ref{ap3}. We will denote $A_p$ as the class of Muckenhoupt weights below and we refer the reader to Appendix C for definition. 

 \begin{theorem}
  \label{goalweighted1}
    Let $g(x)=g_1(x_1)\otimes g_2(x_2) $, $h(x)=h_1(x_1)\otimes h_2(x_2) $. Assume that
    \begin{gather*}
      w_1^1(x_1), w_2^1(x_2) \in A_{p}, \ w_1^2(x_1) \in A_{p_2},\  w_1^3(x_1) \in A_{p_3},\  w_2^2(x_2) \in A_{q_2}, \ w_2^3(x_2) \in A_{q_3},
   \end{gather*}
      then \eqref{biop34} maps $L^p(w_1^1\otimes w_2^1) \times L^{p_2}_{x_1}(w_1^2)(L^{q_2}_{x_2}(w_2^2))\times L^{p_3}_{x_1}(w_1^3)(L^{q_3}_{x_2}(w_2^3)) \to L^r(w_1\otimes w_2)$ for $1<p,p_2,p_3,q_2,q_3<\infty$, $0<r<\infty$  and ${1\over p}+ {1\over p_2}+{1\over p_3}={1\over p}+{1\over q_2}+{1\over q_3}={1\over r}$, where
     \begin{gather*}
        w_1(x_1)=(w_1^1)^{r/p} \cdot (w_1^2)^{r/p_2} \cdot (w_1^3)^{r/p_3},\\
        w_2(x_2)=(w_2^1)^{r/p} \cdot (w_2^2)^{r/q_2} \cdot (w_2^3)^{r/q_3}.
     \end{gather*}

    In particular, by taking $w_1^1=w_1^2=w_1^3=w_1\in A_{\min{(p,p_2,p_3)}}$, $w_2^1=w_2^2=w_2^3=w_2\in A_{\min{(p,q_2,q_3)}}$, \eqref{biop34} maps $L^p(w_1\otimes w_2) \times L^{p_2}_{x_1}(w_1)(L^{q_2}_{x_2}(w_2))\times L^{p_3}_{x_1}(w_1)(L^{q_3}_{x_2}(w_2)) \to L^r(w_1 \otimes w_2)$.
  \end{theorem}

Now we state our result for the corresponding bi-parameter trilinear pseudo-differential operators.  The one-parameter case was studied in
\cite{LZH}. Let
\begin{equation}
\label{bipdo}
T_{ab}(f,g,h)(x):=\int_{\n^6}a(x,\xi,\eta,\zeta)b(x,\eta,\zeta)\hat f(\xi)\hat g(\eta)\hat h(\zeta)e^{2\pi i x(\xi+\eta+\zeta)}d\xi d\eta d\zeta,
\end{equation}
where $f,g,h \in \mathcal{S}(\nn)$, and the bi-parameter  smooth symbols $a,b\in BS^0_{1,0}$ satisfy the following conditions
\begin{gather}
 \nonumber |\partial_{x_1}^{l_1} \partial_{x_2}^{l_2}\partial_{\xi_1}^{\alpha_1} \partial_{\xi_2}^{\alpha_2}\partial_{\eta_1}^{\beta_1} \partial_{\eta_2}^{\beta_2}\partial_{\zeta_1}^{\gamma_1} \partial_{\zeta_2}^{\gamma_2} a(x,\xi,\eta,\zeta)| \\
 \nonumber \lesssim {1\over (1+|\xi_1|+|\eta_1|+|\zeta_1|)^{\alpha_1+\beta_1+\gamma_1}} {1\over (1+|\xi_2|+|\eta_2|+|\zeta_2|)^{\alpha_2+\beta_2+\gamma_2}},\\
 \label{bisymbol}
 |\partial_{x_1}^{l_1} \partial_{x_2}^{l_2}\partial_{\eta_1}^{\beta_1} \partial_{\eta_2}^{\beta_2}\partial_{\zeta_1}^{\gamma_1} \partial_{\zeta_2}^{\gamma_2} b(x,\eta,\zeta)|\lesssim {1\over (1+|\eta_1|+|\zeta_1|)^{\beta_1+\gamma_1}} {1\over (1+|\eta_2|+|\zeta_2|)^{\beta_2+\gamma_2}}.
\end{gather}
for every $x=(x_1,x_2), \xi=(\xi_1,\xi_2),\eta=(\eta_1,\eta_2),\zeta=(\zeta_1,\zeta_2)\in \n \times \n$ and all multi-indices $\alpha=(\alpha_1,\alpha_2)$, $\beta=(\beta_1,\beta_2)$ and $\gamma=(\gamma_1,\gamma_2)$. We will prove the following estimate when assuming the estimates in Theorem \ref{bithm}
\begin{theorem}
\label{bithp}
The operators $T_{ab}$ defined as $\eqref{bipdo}$ map $L^{p_1}\times L^{p_2}\times L^{p_3}\to L^r$ for $1<p_1,p_2,p_3<\infty$ with $1/p_1+1/p_2+1/p_3=1/r$  and $0<r<\infty$, provided that the multiplier operator defined in (\ref{assumption})  satisfies the same $L^r$ estimate.
\end{theorem}

The proof of the $L^r$ estimates for the bi-parameter trilinear flag Fourier multipliers of Theorem \ref{bithm} proceeds by reducing to a decomposition into multipliers based on the  support on the frequency variables, then studying the H\"older type $L^r$ estimates for each. Such a reduction is partly inspired by earlier work in both the single-parameter and bi-parameter settings, for instance, \cite{MT, muscalu2004bi, muscalu2004multi, muscaluflag}. To prove Theorem \ref{bithp}, we reduce the bi-parameter trilinear pseudo-differential operator to a localized version. Then by taking advantage of the paraproducts studied in \cite{muscalu2013classical,muscalu2007paraproducts}, but with all dyadic intervals having lengths at most $1$,  and Theorem $\ref{bithm}$, we prove the H\"older  estimates for the localized operator. This is Theorem $\ref{localth}$.

The rest of the paper is organized as follows.
In Section $\ref{nt}$, we collect some notation and definitions used in the paper.
Section $\ref{bithmpf}$ contains the proof of Theorem $\ref{bithm}$.
In Section $\ref{rtl}$,  we show that the main Theorem  $\ref{bithp}$ can be reduced to an estimate for a localized operator (Theorem $\ref{localth}$).
In Section $\ref{rtp}$, we give the proof of Theorem $\ref{localth}$. In fact, the localized operator will be written as certain bi-parameter paraproducts, where all the involved dyadic intervals have lengths at most 1. This allows us to avoid the more complicated ``size" and ``energy" estimates used in \cite{muscalu2004bi, muscalu2004multi, muscaluflag} to deal with paraproducts.

{\bf Acknowledgement.} The authors are grateful to Camil Muscalu for pointing out an error in our first version posted in the arxiv.org and for
many useful comments as we were revising the paper.
To be precise, the derivation of (4.5) in the original version of this paper was incorrect. Indeed, we do not claim to have a proof of the
earlier version of Theorem 1.7. Instead
we prove here that the $L^r$ estimates for the bi-parameter trilinear flag multiplier can be reduced to the $L^r$ estimate for multipliers of the form $m_1(\xi_1, \eta_1, \zeta_1)m_2(\eta_2, \zeta_2)$. The main revision is in subsection 4.2.2 where we have adapted a new method of reduction of the general bi-parameter trilinear multiplier to the special one of the form $m_1(\xi_1, \eta_1, \zeta_1)m_2(\eta_2, \zeta_2)$. We also thank Camil Muscalu
for communicating the results of Zhai's thesis,  which alerted
us to the interest in addressing the tensor product case  (Theorem \ref{muscaluZhai}) (See Appendix B).

\section{notations and preliminaries}
\label{nt}
 Let $\mathcal{S}(\n^n)$ denote the Schwartz space of rapidly decreasing, $C^{\infty}$ functions in $\n^n$. Define the Fourier transform of a function $f$ in $\mathcal{S}(\n^n)$ as $$F(f)(\xi)=\hat f(\xi)=\int_{\n^n}f(x)e^{-2\pi ix\cdot\xi}dx$$ extended in the usual way to the space of tempered distribution $\mathcal{S}'(\n^n)$, which is the dual space of $\mathcal{S}(\n^n)$. The use $\mathcal{F}^{-1}(f)$ to denote the inverse Fourier transform of $f$.

\medskip

 Throughout the paper, expressions of the form $A\lesssim B$ are used to mean that there exists a universal constant $C>1$ so that $A\leq CB$, and the notation $A\sim B$ denotes that both $A\lesssim B$ and $B\lesssim A$.
\medskip

Intervals in the form of $[2^kn, 2^k(n+1)]$ in $\n$, where $k,n\in \z$, are called dyadic intervals; and $\mathbb{D}$ is the set of all such dyadic intervals. Moreover, the occurrence of any of these expressions in this paper means the following: $I^0:=[-1,1]$, $I_n=J_n=I'_n=J'_n:=[n,n+1]$ for $n\in \z$.

  \begin{definition}
 For $I\in \mathbb{D}$, we define the approximate cutoff function as
 \begin{equation}
 \label{apc}
 \tilde \chi_I(x):=(1+{\dst(x,I)\over |I|})^{-100}
 \end{equation}
 \end{definition}

 \begin{definition}
 Let $I\subseteq \n$ be an arbitrary interval. A smooth function $\varphi$ is said to be a bump adapted to $I$ if and only if one has $$|\varphi^{(l)}|\leq C_lC_M {1\over |I|^l}{1\over (1+|x-x_I|/|I|)^M}$$ for every integer $M\in \mathbb{N}$ and sufficiently many derivatives $l\in \mathbb{N}$, where $x_I$ denotes the center of $I$ and $|I|$ is the length of $I$.

 If $\varphi_I$ is a bump adapted to $I$, we say that $|I|^{-1/p}\varphi_I$ is an $L^p$-normalized bump adapted to $I$, for $1\leq p\leq \infty.$
 \end{definition}

 \begin{definition}
 A sequence of $L^2$-normalized bumps $(\Phi_I)_{I\in\mathbb{D}}$ adapted to dyadic intervals $I \in \mathbb{D}$ is called a non-lacunary sequence if and only if for each $I \in \mathbb{D}$ there exists an interval $\omega_I=\omega_{|I|}$ symmetric with respect to the origin so that $\spt \widehat {\Phi_I}\subseteq \omega_I$ and $|\omega_I|\sim |I|^{-1}$.
 \end{definition}

 \begin{definition}
 A sequence of $L^2$-normalized bumps $(\Phi_I)_{I\in\mathbb{D}}$ adapted to dyadic intervals $I \in \mathbb{D}$ is called a lacunary sequence if and only if for each $I \in \mathbb{D}$ there exists an interval $\omega_I=\omega_{|I|}$ so that $\spt \widehat {\Phi_I}\subseteq \omega_I$, $|\omega_I|\sim |I|^{-1}\sim dist(0,\omega_I)$ and $0\notin 5\omega_I$.
 \end{definition}

\begin{definition}
\label{defop}
Let $\mathcal{I},\mathcal{J}\subseteq \mathbb{D}$ be two families of dyadic intervals that have lengths at most $1$. Suppose that $(\phi_I^j)_{I\in \mathcal{I}}$ for $j=1,2,3$ are three families of $L^2$-normalized bump functions such that the family  $(\phi_I^2)_{I\in \mathcal{I}}$ is non-lacunary while the families $(\phi_I^j)_{I\in \mathcal{I}}$ for $j\neq 2$ are both lacunary, and $(\phi_J^j)_{J\in \mathcal{J}}$ for $j=1,2,3$ are three families of $L^2$-normalized bump functions, where at least two of the three are lacunary.

\medskip

We define as in \cite{muscalu2007paraproducts} the discrete model operators $T_1$ and $T_{1,k_0}$ for a positive integer $k_0$ by
\begin{gather}
\label{t1}
T_1(f,g,h)=\sum_{I\in \mathcal{I}}{1\over |I|^{1\over 2}}\langle f,\phi_I^1\rangle \langle B^1_I(g,h),\phi_I^2\rangle \phi_I^3 \\
\label{b1}
\text{where} \qquad B^1_I(g,h)=\sum_{\substack{J\in \mathcal{J}\\ |\omega_J^3|\leq |\omega^2_I|}}{1\over |J|^{1\over 2}} \langle g,\phi_J^1\rangle \langle h,\phi_J^2\rangle \phi_J^3 \\
\label{tk}
T_{1,k_0}(f,g,h)=\sum_{I\in \mathcal{I}}{1\over |I|^{1\over 2}}\langle f,\phi_I^1\rangle \langle B^1_{I,k_0}(g,h),\phi_I^2\rangle \phi_I^3 \\
\label{bk}
\text{where} \qquad B^1_{I,k_0}(g,h)=\sum_{\substack{J\in \mathcal{J}\\ 2^{k_0}|\omega_J^3|\simeq |\omega^2_I|}}{1\over |J|^{1\over 2}} \langle g,\phi_J^1\rangle \langle h,\phi_J^2\rangle \phi_J^3
\end{gather}



\end{definition}

\section{A Leibniz rule}
\label{sectionlb}
Before giving the proof of Theorem  \ref{bithm}, we give an example where the operator we consider plays a role. This is one of the possible
motivations for the study of such operators. The details are included in the Appendix \ref{ap1}.

Let $f,g,h \in \mathcal{S}(\nn)$, and for  $\alpha_1,\alpha_2>0$ define
$$\widehat {D^{\alpha_1}_1 h}(u)=(2\pi |u_1|)^{\alpha_1}\hat h(u), \quad \widehat {D^{\alpha_2}_2 h}(u)=(2\pi |u_2|)^{\alpha_2} \hat h(u),\quad u=(u_1,u_2)\in \n^2.$$
In \cite{muscalu2004bi,muscalu2013classical,muscalu2004multi}, it was proved that the boundedness of the bi-parameter bilinear Fourier multiplier in Theorem \ref{bipcoifman} implies the following Leibniz rule.
\begin{eqnarray}
  \|D^{\alpha_1}_1 D^{\alpha_2}_2 (f g)\|_{L^r} &\lesssim&  \|D^{\alpha_1}_1 D^{\alpha_2}_2 f\|_{L^{p_1}} \|g\|_{L^{q_1}} +  \|D^{\alpha_1}_1  f\|_{L^{p_2}} \|D^{\alpha_2}_2 g\|_{L^{q_2}} \nonumber\\
  & & +  \|D^{\alpha_2}_2 f\|_{L^{p_3}} \|D^{\alpha_1}_1  g\|_{L^{q_3}} +  \| f\|_{L^{p_4}} \| D^{\alpha_1}_1 D^{\alpha_2}_2 g\|_{L^{q_4}},\label{leib1}
\end{eqnarray}
where $1/p_i+1/q_i=1/r$, $1<p_i,q_i\leq \infty$ for $i=1,2,3,4$ and $\max{\big({1\over 1+\alpha_1}, {1\over 1+\alpha_2}\big)}<r<\infty$ (one can refer to \cite{muscalu2004bi,muscalu2013classical,muscalu2004multi} to see how such restrictions appear).

Then it's very natural and interesting to ask if such bi-parameter Leibniz rule holds when there is higher complexity of the differentiation. In particular,
a Leibniz estimate for an expression like the following relies on our theorem:
 $$\|D^{\alpha_1}_1 D^{\alpha_2}_2\big(f\cdot D^{\beta_1}_1 D^{\beta_2}_2(g h)\big)\|_{L^r},$$
 where $ \max{\big({1\over 1+\alpha_1}, {1\over 1+\alpha_2}, {1\over 1+\beta_1}, {1\over 1+\beta_2}\big)}<r<\infty$.
 First note that an iteration of \eqref{leib1} results in the inequality
 \begin{eqnarray}
 & &\|D^{\alpha_1}_1 D^{\alpha_2}_2\big(f  \cdot D^{\beta_1}_1 D^{\beta_2}_2(g h)\big)\|_{L^r} \nonumber \\
  &\lesssim&  \|D^{\alpha_1}_1 D^{\alpha_2}_2 f\|_{L^{p_1}} \| D^{\beta_1}_1 D^{\beta_2}_2(g h)\|_{L^{t_1}} +  \|D^{\alpha_1}_1  f\|_{L^{p_2}} \| D^{\beta_1}_1 D^{\alpha_2+\beta_2}_2(g h)\|_{L^{t_2}} \nonumber \nonumber \\
  & & +  \|D^{\alpha_2}_2 f\|_{L^{p_3}} \| D^{\alpha_1+\beta_1}_1 D^{\beta_2}_2(g h)\|_{L^{t_3}} +  \| f\|_{L^{p_4}} \|  D^{\alpha_1+\beta_1}_1 D^{\alpha_2+\beta_2}_2(g h)\|_{L^{t_4}}, \label{lbrule}
\end{eqnarray}
where $1/r=1/p_i+1/t_i$ for $i=1,2,3,4$. However, this argument using \eqref{leib1} requires $t_i> 1$, while the ideal restriction is $t_i> {1\over 2}$, since we expect to further expand the differentiation on $gh$. Thus, it's a non-trivial question to get a general Leibniz rule for \eqref{lbrule}, that is
\begin{eqnarray}
 & &\|D^{\alpha_1}_1 D^{\alpha_2}_2\big(f \cdot D^{\beta_1}_1 D^{\beta_2}_2(gh)\big)\|_{L^r}  \nonumber \\
  &\lesssim & \|D^{\alpha_1}_1 D^{\alpha_2}_2 f\|_{L^{p_1}} \cdot \|D^{\beta_1}_1 D^{\beta_2}_2 g\|_{L^{q_1}}\cdot  \| h \|_{L^{s_1}}+  \|D^{\alpha_1}_1 D^{\alpha_2}_2 f\|_{L^{p_2}} \cdot \|D^{\beta_1}_1 g\|_{L^{q_2}} \cdot   \|  D^{\beta_2}_2 h \|_{L^{s_2}} \nonumber\\
    & &+ \|D^{\alpha_1}_1 D^{\alpha_2}_2 f\|_{L^{p_3}}\cdot\|D^{\beta_2}_2 g\|_{L^{q_3}}  \cdot\|  D^{\beta_1}_1 h \|_{L^{s_3}}+ \|D^{\alpha_1}_1 D^{\alpha_2}_2 f\|_{L^{p_4}} \cdot \| g \|_{L^{q_4}}\cdot \|D^{\beta_1}_1 D^{\beta_2}_2 h\|_{L^{s_4}} \nonumber \\
  & & +\|D^{\alpha_1}_1 f\|_{L^{p_5}}\cdot \|D^{\beta_1}_1 D^{\alpha_2+\beta_2}_2 g\|_{L^{q_5}} \cdot \| h \|_{L^{s_5}}  +  \|D^{\alpha_1}_1  f\|_{L^{p_6}} \cdot \|D^{\beta_1}_1 g\|_{L^{q_6}} \cdot \|  D^{\alpha_2+\beta_2}_2 h \|_{L^{s_6}} \nonumber\\
    & &+ \|D^{\alpha_1}_1 f\|_{L^{p_7}}\cdot \|D^{\alpha_2+\beta_2}_2 g\|_{L^{q_7}} \cdot \|  D^{\beta_1}_1 h \|_{L^{s_7}}+\|D^{\alpha_1}_1  f\|_{L^{p_8}} \cdot\| g \|_{L^{q_8}} \cdot \|D^{\beta_1}_1 D^{\alpha_2+\beta_2}_2 h\|_{L^{s_8}} \nonumber\\
  & &+ \|D^{\alpha_2}_2 f\|_{L^{p_9}}\cdot\|D^{\alpha_1+\beta_1}_1 D^{\beta_2}_2 g\|_{L^{q_9}} \cdot \| h \|_{L^{s_9}} +  \| D^{\alpha_2}_2 f\|_{L^{p_{10}}} \cdot \|D^{\alpha_1+\beta_1}_1 g\|_{L^{q_{10}}} \cdot \|  D^{\beta_2}_2 h \|_{L^{s_{10}}} \nonumber \\
   & &+ \| D^{\alpha_2}_2 f\|_{L^{p_{11}}}\cdot \|D^{\beta_2}_2 g\|_{L^{q_{11}}}\cdot  \|  D^{\alpha_1+\beta_1}_1 h \|_{L^{s_{11}}}+ \| D^{\alpha_2}_2 f\|_{L^{p_{12}}}\cdot \| g \|_{L^{q_{12}}}\cdot \|D^{\alpha_1+\beta_1}_1 D^{\beta_2}_2 h\|_{L^{s_{12}}} \nonumber\\
  & & + \| f\|_{L^{p_{13}}}\cdot \|D^{\alpha_1+\beta_1}_1 D^{\alpha_2+\beta_2}_2 g\|_{L^{q_{13}}}  \cdot\| h \|_{L^{s_{13}}}  +  \|f\|_{L^{p_{14}}} \cdot \|D^{\alpha_1+\beta_1}_1 g\|_{L^{q_{14}}} \cdot \|  D^{\alpha_2+\beta_2}_2 h \|_{L^{s_{14}}} \nonumber \\
    & &+ \|f\|_{L^{p_{15}}} \cdot\|D^{\alpha_2+\beta_2}_2 g\|_{L^{q_{15}}} \cdot \|  D^{\alpha_1+\beta_1}_1 h \|_{L^{s_{15}}} + \| f\|_{L^{p_{16}}}\cdot \| g     \|_{L^{q_{16}}} \|D^{\alpha_1+\beta_1}_1 D^{\alpha_2+\beta_2}_2 h\|_{L^{s_{16}}}, \nonumber \\
   & & \label{lb16}
\end{eqnarray}
where $1/p_i+1/q_i+1/s_i=1/r$, $1<p_i,q_i,s_i<\infty$ for $i=1,\dots,16$ and $\max{\big({1\over 1+\alpha_1}, {1\over 1+\alpha_2}, {1\over 1+\beta_1}, {1\over 1+\beta_2}\big)}<r<\infty$.

It turns out, we can write $D^{\alpha_1}_1 D^{\alpha_2}_2\big(f  \cdot D^{\beta_1}_1 D^{\beta_2}_2(g h)\big)$ as a sum of essentially two types of Fourier multipliers. More precisely, we can write
$$D^{\alpha_1}_1 D^{\alpha_2}_2\big(f  \cdot D^{\beta_1}_1 D^{\beta_2}_2(g h)\big)=T_{m_1,m_2}(f,g,h)+T_{m_3,m_4}(f,g,h),$$
 where $T_{m_1,m_2}$ is the operator \eqref{biop34} in Theorem \ref{bithm}, and $$ T_{m_3,m_4}(f,g,h)=\int_{\n^6} m_3(\xi,\eta)m_4(\eta,\zeta) \hat f(\xi) \hat g(\eta) \hat h(\zeta)  d\xi d\eta d\zeta,$$
 with $m_3, m_4 \in \mathcal{BM}(\n^4)$.  Thus, in order to get the Leibniz estimate, we just need to show the H\"older $L^r$ estimate for each of the above two operators implies those pieces in \eqref{lb16}.

 Now let's take a quick look at how the estimate in Theorem \ref{bithm} is associated with the 16 terms appearing in \eqref{lb16}. We indicate some key steps here, and more details can be found in Appendix \ref{ap1}. Let $\psi\in \sch $ be a Schwartz function satisfying $\spt \hat \psi \subseteq\{1/2\leq |u| \leq 2\}$ and  $$1=\sum_{k\in \z} \widehat {\psi}_k(u), \qquad u \neq 0.$$

Now  $f\cdot g \cdot h$ can be rewritten by using
\begin{gather}
\nonumber
\hat f(\xi) \hat g(\eta)  \hat h(\zeta)= \\
\left(\sum_{j_1,j_2}\widehat \psi_{j_1}(\xi_1)\widehat\psi_{j_2}(\xi_2)\hat f(\xi)\right)\left(\sum_{k_1,k_2}\widehat\psi_{k_1}(\eta_1)\widehat \psi_{k_2}(\eta_2)\hat g(\eta)\right)\left(\sum_{l_1,l_2}\widehat\psi_{l_1}(\zeta_1)\widehat \psi_{l_2}(\zeta_2)\hat h(\zeta)\right) \label{rdorg}.
\end{gather}


Then by using a sequence of appropriate reductions, it turns out that our goal $D^{\alpha_1}_1 D^{\alpha_2}_2 \left(f \cdot D_1^{\beta_1} D_2^{\beta_2}(gh)\right)$ can be written as a summation of terms that includes, for example,
\begin{eqnarray}
\nonumber
& &D^{\alpha_1}_1 D^{\alpha_2}_2 \left( \sum _ { j _ { 1 }, j_2 } \big( \left( f * \left(\psi _ { j _ { 1 } } \otimes  \psi _ { j _ { 2 } }\right)\right) \left( D_1^{\beta_1} D_2^{\beta_2} \Pi_1(g,h) * \left(\varphi _ { j _ { 1 } }\otimes \varphi _ { j _ { 2 } } \right) \right)\big ) *  \left(\psi _ { j _ { 1 } }\otimes \psi _ { j _ { 2 } }\right)\right) \\
&=& \sum _ { j _ { 1 }, j_2 } \big( \left(  D^{\alpha_1}_1 D^{\alpha_2}_2f * \left(\psi'' _ { j _ { 1 } } \otimes  \psi'' _ { j _ { 2 } }\right)\right) \left( D_1^{\beta_1} D_2^{\beta_2} \Pi_1(g,h) * \left(\varphi _ { j _ { 1 } }\otimes \varphi _ { j _ { 2 } } \right) \right)\big ) *  \left(\psi' _ { j _ { 1 } }\otimes \psi' _ { j _ { 2 } }\right). \nonumber \\
& & \label{rd1}
\end{eqnarray}
where $ D_1^{\beta_1}D_2^{\beta_2}\Pi_1(g,h)$ has the form
$$ \sum _ { k _ { 1 } \ll j _ { 1 } } \sum _ { k _ { 2 } \ll j _ { 2 } }  \left( \left(D^{\beta_1}_1 D^{\beta_2}_2 g * \left( \psi'' _ { k_1 } \otimes   \psi'' _ { k_2 } \right) \right) \cdot \left(  h * \left(  \varphi _ { k_1 } \otimes \varphi _ { k_2 }\right) \right) \right)*\left({ \psi' } _ { k_1 } \otimes { \psi' } _ { k_2 }\right).$$
 Here for $i=1,2$, $k_i\ll j_i$ means $k_i<j_i-100$, $\varphi_{k_i}= \sum_{l_i\ll k_i}\psi_{l_i}$, and $\psi_{k_i}$ may actually represent $\sum_{k_i-100\leq l_i \leq k_i+100}\psi_{l_i}$. Moreover,  $\widehat \psi'_{k_i} (u)=\widehat \psi_{k_i}(u)|{u\over 2^{k_i}}|^{\beta_i}$, $\widehat \psi''_{k_i} (u)=\widehat \psi_{k_i}(u)|{ 2^{k_i}\over u}|^{\beta_i}$, and $\otimes$ represents the tensor product.

In fact, the expression \eqref{rd1} is a Fourier multiplier with symbol having the form $m_1(\xi,\eta,\zeta)m_2(\eta,\zeta)$, where $m_1\in \mathcal{BM}(\n^6)$ and $m_2\in \mathcal{BM}(\n^4)$ respectively. Then one can see the boundedness of \eqref{biop34} in Theorem \eqref{bithm} implies the bound
$$\|D^{\alpha_1}_1 D^{\alpha_2}_2 f\|_{L^{p_1}} \cdot \|D^{\beta_1}_1 D^{\beta_2}_2 g\|_{L^{q_1}}\cdot  \| h \|_{L^{s_1}},$$
which appears on the right hand side  of \eqref{lb16}. To see how to get the other terms in \eqref{lb16} from the boundedness of \eqref{biop34}, we  need to look at the terms that are similar to \eqref{rd1}, which appear in the process of reduction. Here we just give one more example,

\begin{eqnarray}
\nonumber
& &D^{\alpha_1}_1 D^{\alpha_2}_2 \left( \sum _ { j _ { 1 }, j_2 } \big( \left( f * \left(\psi _ { j _ { 1 } } \otimes  \varphi _ { j _ { 2 } }\right)\right) \left( D_1^{\beta_1} D_2^{\beta_2} \Pi_2(g,h) * \left(\varphi _ { j _ { 1 } }\otimes \psi _ { j _ { 2 } } \right) \right)\big ) *  \left(\psi _ { j _ { 1 } }\otimes \psi _ { j _ { 2 } }\right)\right) \\
&=& \sum _ { j _ { 1 }, j_2 } \big( \left(  D^{\alpha_1}_1 f * \left(\psi'' _ { j _ { 1 } } \otimes  \varphi _ { j _ { 2 } }\right)\right) \left( D_1^{\beta_1} D_2^{\alpha_2+\beta_2} \Pi_2(g,h) * \left(\varphi _ { j _ { 1 } }\otimes \psi'' _ { j _ { 2 } } \right) \right)\big ) *  \left(\psi' _ { j _ { 1 } }\otimes \psi' _ { j _ { 2 } }\right). \nonumber \\
& & \label{rd2}
\end{eqnarray}
where $ D_1^{\beta_1}D_2^{\alpha_2+\beta_2}\Pi_2(g,h)$ has the form
$$ \sum _ { k _ { 1 } \ll j _ { 1 } } \sum _ { k _ { 2 } \ll j _ { 2 } }  \left( \left( D^{\beta_2+\alpha_2}_2 g * \left( \varphi _ { k_1 } \otimes   \psi'' _ { k_2 } \right) \right) \cdot \left( D^{\beta_1}_1 h * \left( \varphi _ { k_1 } \otimes \psi'' _ { k_2 }\right) \right) \right)*\left({ \psi' } _ { k_1 } \otimes { \psi' } _ { k_2 }\right).$$
As before, \eqref{rd2} corresponds to an operator with the symbol $m_1(\xi,\eta,\zeta)m_2(\eta,\zeta)$, and its boundedness gives
$$\|D^{\alpha_1}_1  f\|_{L^{p_7}} \cdot \| D^{\alpha_2+\beta_2}_2 g\|_{L^{q_7}}\cdot  \|D^{\beta_1}_1 h \|_{L^{s_7}}.$$

However, there are still terms in the reduction that can not be covered by the operators $\eqref{biop34}$. Such terms appear when, for example, $j_1\gg k_1\gg l_1$, $j_2\ll k_2\ll l_2$ in \eqref{rdorg}, and they actually correspond to the operator $T_{m_3,m_4}$. Thus, in order to obtain the final goal \eqref{lb16}, one also needs  the H\"older type $L^r$ estimate for $T_{m_3,m_4}$. Note that the bi-parameter symbol $m_3(\xi,\eta) m_4(\eta,\zeta)$  is more singular than the one in $T_{m_1,m_2}$, and thus it is a more challenging task to obtain H\"older type estimates for the associated operator. Overcoming this obstacle is another issue for future research in this subject.

\section{Proof of Theorem \ref{bithm}}
\label{bithmpf}
\subsection{Reduction of the symbols}
To prove this theorem, we start with the decomposition and reduction of the symbols
 $$m(\xi,\eta,\zeta):=m_1(\xi,\eta,\zeta) m_2(\eta,\zeta).$$

 We take smooth homogeneous functions $\phi_0$ and $\phi_1$ on $\nnn\setminus \{0\}$ such that $\phi_0(u,v,w)+\phi_1(u,v,w)=1$  and
  $$\spt \phi_0 \subset \{(u,v,w)| |v|+|w|\leq \epsilon |u|\},\qquad \spt \phi_1 \subset \{(u,v,w)| |v|+|w|\geq {1\over 2}\epsilon |u|\}$$
  for some small $\epsilon$. Then we decompose $m_1$ as
  \begin{eqnarray*}
  & &m_1=m_1(\phi_0(\xi_1,\eta_1,\zeta_1) + \phi_1(\xi_1,\eta_1,\zeta_1))(\phi_0(\xi_2,\eta_2,\zeta_2) + \phi_1(\xi_2,\eta_2,\zeta_2))  \nonumber \\
  &=&m_1 \phi_0(\xi_1,\eta_1,\zeta_1)\phi_0(\xi_2,\eta_2,\zeta_2) + m_1 \phi_0(\xi_1,\eta_1,\zeta_1)\phi_1(\xi_2,\eta_2,\zeta_2) \nonumber \\
  & &\ + m_1\phi_1(\xi_1,\eta_1,\zeta_1)\phi_0(\xi_2,\eta_2,\zeta_2)+  m_1\phi_1(\xi_1,\eta_1,\zeta_1)\phi_1(\xi_2,\eta_2,\zeta_2) \nonumber\\
  &:=&m_{0,0}+m_{0,1}+m_{1,0}+m_{1,1}
  \end{eqnarray*}

\vskip0.3cm
Obviously $m_{1,1}(\xi,\eta,\zeta)m_2(\eta,\zeta)$ satisfies the condition
\begin{gather*}
  |\partial_{\xi_1}^{\alpha_1} \partial_{\xi_2}^{\alpha_2} \partial_{\eta_1}^{\beta_1}  \partial_{\eta_2}^{\beta_2} \partial_{\zeta_1}^{\gamma_1}  \partial_{\zeta_2}^{\gamma_2} \left(m_{1,1}(\xi,\eta,\zeta) m_{2}(\eta,\zeta)\right)| \\
  \lesssim {1\over (|\xi_1|+|\eta_1|+|\zeta_1|)^{\alpha_1+\beta_1+\gamma_1}} {1\over (|\xi_2|+|\eta_2|+|\zeta_2|)^{\alpha_2+\beta_2+\gamma_2}},
   \end{gather*}

and the desired estimate follows from the multilinear version of Theorem $\ref{bipcoifman}$.

\medskip
 $m_{0,1}$ and $m_{1,0}$ satisfy the similar conditions. For example,
\begin{gather}
  |\partial_{\xi_1}^{\alpha_1} \partial_{\xi_2}^{\alpha_2} \partial_{\eta_1}^{\beta_1}  \partial_{\eta_2}^{\beta_2} \partial_{\zeta_1}^{\gamma_1}  \partial_{\zeta_2}^{\gamma_2} \left(m_{0,1}(\xi,\eta,\zeta)  m_{2}(\eta,\zeta)\right)| \label{mid1} \\ \nonumber
  \lesssim  \sum_{\substack{\beta'_1+\beta''_1=\beta_1\\\gamma_1'+\gamma_1''=\gamma_1}}{1\over (|\xi_1|+|\eta_1|+|\zeta_1|)^{\alpha_1+\beta'_1+\gamma'_1}}{1\over (|\eta_1|+|\zeta_1|)^{\beta''_1+\gamma''_1}}  {1\over (|\xi_2|+|\eta_2|+|\zeta_2|)^{\alpha_2+\beta_2+\gamma_2}},
   \end{gather}
Note that such a condition is stronger than   what $m_{0,0}(\xi,\eta,\zeta)m_{2}(\eta,\zeta)$ satisfies, since the estimate for the second parameter is in a classical type. More precisely,

\begin{gather*}
  |\partial_{\xi_1}^{\alpha_1} \partial_{\xi_2}^{\alpha_2} \partial_{\eta_1}^{\beta_1}  \partial_{\eta_2}^{\beta_2} \partial_{\zeta_1}^{\gamma_1}  \partial_{\zeta_2}^{\gamma_2} \left(m_{0,0}(\xi,\eta,\zeta)m_{2}(\eta,\zeta)\right)| \\
  \lesssim  \sum_{\substack{\beta'_1+\beta''_1=\beta_1\\\gamma_1'+\gamma_1''=\gamma_1}}\sum_{\substack{\beta'_2+\beta''_2=\beta_2\\ \gamma_2'+\gamma_2''=\gamma_2}}{1\over (|\xi_1|+|\eta_1|+|\zeta_1|)^{\alpha_1+\beta'_1+\gamma'_1}}{1\over (|\eta_1|+|\zeta_1|)^{\beta''_1+\gamma''_1}} \\
  \cdot
  {1\over (|\xi_2|+|\eta_2|+|\zeta_2|)^{\alpha_2+\beta'_2+\gamma'_2}}{1\over (|\eta_2|+|\zeta_2|)^{\beta''_2+\gamma''_2}}.
   \end{gather*}
 Thus, if suffices to study $m_{0,0}(\xi,\eta,\zeta)m_{2}(\eta,\zeta)$.

We choose $\widehat \psi(u) \in \mathcal{S}(\n)$ such that $\text{supp}\,\widehat \psi \subseteq \{u\in \n: 1/2\leq |u|\leq 2\}$, and
\begin{equation}
 \label{psidef} \sum_{j\in \z}\widehat \psi_j(u):=\sum_{j\in \z}\widehat \psi({u\over 2^j})=1\quad \text{for}\ u \neq 0.
\end{equation}
We  define $\varphi$ by
\begin{equation}
\label{varphidef}
\widehat \varphi(u)=\sum_{k=-\infty}^{-3} \widehat \psi(2^{-k} u)\ \  \text{for} \ \  u\neq 0,\quad  \widehat \varphi(0)=1,
\end{equation}
which implies
$$\spt \widehat \varphi \subset \{u\in \n :  |u|\leq 2^{-2}\},\quad \text{and} \ \ \widehat \varphi(u)=1 \ \  \text{for}\ \  |u|\leq 2^{-3}.$$
We also set
$$\chi(u,v,w)=\sum_{j\in \z} \widehat \psi(2^{-j}u)\widehat \varphi(2^{-j+10}v)\widehat \varphi(2^{-j+10}w).$$
Note that $\chi \in \mathcal{M}(\nnn)$ and $\chi(\xi_1,\eta_1,\zeta_1)=1$ , $\chi(\xi_2,\eta_2,\zeta_2)=1$ on $\spt m_{0,0} $.
\vskip0.3cm

Since it is sufficient to consider $m_{0,0}(\xi,\eta,\zeta)m_2(\eta,\zeta)$, we now use $m_1$ to represent $m_{0,0}$. From Taylor's theorem, we have:
\begin{eqnarray*}
m_1(\xi,\eta,\zeta)&=&\sum_{\beta'_2+\gamma'_2<N} \frac{\eta_2^{\beta'_2}\zeta_2^{\gamma'_2}}{\beta'_2!\gamma'_2!}\partial^{\beta'_2}_{\eta_2}\partial^{\gamma'_2}_{\zeta_2}  m_1(\xi,\eta_1,0,\zeta_1,0) \\
& & \ + N \sum_{\beta'_2+\gamma'_2=N} \frac{\eta_2^{\beta'_2}\zeta_2^{\gamma'_2}}{\beta'_2!\gamma'_2!} \int_0^1(1-t)^{N-1}\partial^{\beta'_2}_{\eta_2}\partial^{\gamma'_2}_{\zeta_2}m_1(\xi,\eta_1,t\eta_2,\zeta_1,t\zeta_2)dt,
 \end{eqnarray*}
and
\begin{eqnarray*}
m_1(\xi,\eta,\zeta)&=&\sum_{\beta'_1+\gamma'_1<N} \frac{\eta_1^{\beta'_1}\zeta_1^{\gamma'_1}}{\beta'_1!\gamma'_1!}\partial^{\beta'_1}_{\eta_1}\partial^{\gamma'_1}_{\zeta_1}  m_1(\xi,0,\eta_2,0,\zeta_2) \\
& & \ + N \sum_{\beta'_1+\gamma'_1=N} \frac{\eta_1^{\beta'_1}\zeta_1^{\gamma'_1}}{\beta'_1!\gamma'_1!} \int_0^1(1-s)^{N-1}\partial^{\beta'_1}_{\eta_1}\partial^{\gamma'_1}_{\zeta_1}m_1(\xi,s\eta_1,\eta_2,s\zeta_1,\zeta_2)ds.
 \end{eqnarray*}

 This gives the expression
 \begin{eqnarray}
 & &  m_1(\xi,\eta,\zeta) \nonumber \\
  &=&\sum_{\beta'_1+\gamma'_1<N}\sum_{\beta'_2+\gamma'_2<N} \frac{\eta_1^{\beta'_1}\zeta_1^{\gamma'_1}}{\beta'_1!\gamma'_1!} \frac{\eta_2^{\beta'_2}\zeta_2^{\gamma'_2}}{\beta'_2!\gamma'_2!} \partial^{\beta'_1}_{\eta_1} \partial^{\beta'_2}_{\eta_2}\partial^{\gamma'_1}_{\zeta_1} \partial^{\gamma'_2}_{\zeta_2} m_1(\xi,0,0)\nonumber\\
   & & + N \sum_{\beta'_1+\gamma'_1<N}\sum_{\beta'_2+\gamma'_2=N} \frac{\eta_1^{\beta'_1}\zeta_1^{\gamma'_1}}{\beta'_1!\gamma'_1!} \frac{\eta_2^{\beta'_2}\zeta_2^{\gamma'_2}}{\beta'_2!\gamma'_2!} \int_0^1(1-t)^{N-1}\partial^{\beta'_1}_{\eta_1} \partial^{\beta'_2}_{\eta_2}\partial^{\gamma'_1}_{\zeta_1} \partial^{\gamma'_2}_{\zeta_2}m_1(\xi,0,t\eta_2,0,t\zeta_2)dt \nonumber \\
    & & + N \sum_{\beta'_1+\gamma'_1=N}\sum_{\beta'_2+\gamma'_2<N} \frac{\eta_1^{\beta'_1}\zeta_1^{\gamma'_1}}{\beta'_1!\gamma'_1!} \frac{\eta_2^{\beta'_2}\zeta_2^{\gamma'_2}}{\beta'_2!\gamma'_2!} \int_0^1(1-s)^{N-1}\partial^{\beta'_1}_{\eta_1} \partial^{\beta'_2}_{\eta_2}\partial^{\gamma'_1}_{\zeta_1} \partial^{\gamma'_2}_{\zeta_2}m_1(\xi,s\eta_1,0,s\zeta_1,0)dt  \nonumber\\
    & &+N^2\sum_{\beta'_1+\gamma'_1=N}\sum_{\beta'_2+\gamma'_2=N} \frac{\eta_1^{\beta'_1}\zeta_1^{\gamma'_1}}{\beta'_1!\gamma'_1!} \frac{\eta_2^{\beta'_2}\zeta_2^{\gamma'_2}}{\beta'_2!\gamma'_2!} \int_0^1\int_0^1 (1-s)^{N-1}(1-t)^{N-1}\nonumber \\
    & &\qquad \qquad \qquad \qquad \qquad \qquad \qquad \qquad \partial^{\beta'_1}_{\eta_1} \partial^{\beta'_2}_{\eta_2}\partial^{\gamma'_1}_{\zeta_1} \partial^{\gamma'_2}_{\zeta_2} m_1(\xi,s\eta_1,t\eta_2,s\zeta_1,t\zeta_2)ds dt \nonumber \\
    &:=& m_1^{1,1}(\xi,\eta,\zeta)+m_1^{1,2}(\xi,\eta,\zeta)+m_1^{2,1}(\xi,\eta,\zeta)+m_1^{2,2}(\xi,\eta,\zeta)
    \label{dcpm3}
 \end{eqnarray}

These computations imply that our original symbol $m_1(\xi,\eta,\zeta)m_2(\eta,\zeta)$ can actually be reduced to
$$\big(m_1^{1,1}(\xi,\eta,\zeta)+m_1^{1,2}(\xi,\eta,\zeta)+m_1^{2,1}(\xi,\eta,\zeta)+m_1^{2,2}(\xi,\eta,\zeta)\big) m_2(\eta,\zeta).$$
In the following subsections, we deal with each of these four types of symbols.\\
\vskip1cm

 \subsubsection{The symbol $m_1^{2,2}(\xi,\eta,\zeta)m_2(\eta,\zeta)$}
 ~\\

 A straightforward calculation shows that
 $$|\partial_{\xi_1}^{\alpha_1}\partial_{\xi_2}^{\alpha_2}\partial_{\eta_1}^{\beta_1}\partial_{\eta_2}^{\beta_2}
 \partial_{\zeta_1}^{\gamma_1}\partial_{\zeta_2}^{\gamma_2}\left(m_1^{2,2}(\xi,\eta,\zeta) m_2(\eta,\zeta)\right)|\lesssim  \frac{(|\eta_1|+|\zeta_1|)^{N-\beta_1-\gamma_1}}{|\xi_1|^{N+\alpha_1}} \frac{(|\eta_2|+|\zeta_2|)^{N-\beta_2-\gamma_2}}{|\xi_2|^{N+\alpha_2}},$$
 which means that
  \begin{gather*}
  |\partial_{\xi_1}^{\alpha_1}\partial_{\xi_2}^{\alpha_2}\partial_{\eta_1}^{\beta_1}\partial_{\eta_2}^{\beta_2}
 \partial_{\zeta_1}^{\gamma_1}\partial_{\zeta_2}^{\gamma_2}\left(m_1^{2,2}(\xi,\eta,\zeta) m_2(\eta,\zeta)\chi(\xi_1,\eta_1,\zeta_1)\chi(\xi_2,\eta_2,\zeta_2)\right)| \\
 \lesssim  \frac{1}{(|\xi_1|+|\eta_1|+|\zeta_1|)^{\alpha_1+\beta_1+\gamma_1}} \frac{1}{(|\xi_2|+|\eta_2|+|\zeta_2|)^{\alpha_2+\beta_2+\gamma_2}}
 \end{gather*}
 for $\beta_1+\gamma_1\leq N$ and $\beta_2+\gamma_2\leq N$ for $N$ sufficiently large. Therefore this symbol also falls within
 the scope of Theorem \ref{bipcoifman}.
 ~\\

\vskip1cm

\subsubsection{The symbol $m_1^{1,2}(\xi,\eta,\zeta)m_2(\eta,\zeta),m_1^{2,1}(\xi,\eta,\zeta)m_2(\eta,\zeta)$}
~\\

As in the argument for $m_1^{1,1}(\xi,\eta,\zeta)$, by a direct calculation one can see that $m_1^{1,2}(\xi,\eta,\zeta)m_2(\eta,\zeta)$ and $m_1^{2,1}(\xi,\eta,\zeta)m_2(\eta,\zeta)$
satisfy the classical restriction on the second and first parameter respectively. More precisely, for example,
 \begin{gather}
 \label{mid2}
  |\partial_{\xi_1}^{\alpha_1} \partial_{\xi_2}^{\alpha_2} \partial_{\eta_1}^{\beta_1}  \partial_{\eta_2}^{\beta_2} \partial_{\zeta_1}^{\gamma_1}  \partial_{\zeta_2}^{\gamma_2} \left(m_{1}^{1,2}(\xi,\eta,\zeta)m_2(\eta,\zeta)\right)| \\ \nonumber
  \lesssim  \sum_{\substack{\beta'_1+\beta''_1=\beta_1\\\gamma_1'+\gamma_1''=\gamma_1}}{1\over (|\xi_1|+|\eta_1|+|\zeta_1|)^{\alpha_1+\beta'_1+\gamma'_1}}{1\over (|\eta_1|+|\zeta_1|)^{\beta''_1+\gamma''_1}}  {1\over (|\xi_2|+|\eta_2|+|\zeta_2|)^{\alpha_2+\beta_2+\gamma_2}},
   \end{gather}
    for $\beta_2+\gamma_2\leq N$, where $N$ sufficiently large. Note that this is stronger than the condition that $m_1^{1,1}(\xi,\eta,\zeta)m_2(\eta,\zeta)$ satisfies, i.e.,

 \begin{gather*}
  |\partial_{\xi_1}^{\alpha_1} \partial_{\xi_2}^{\alpha_2} \partial_{\eta_1}^{\beta_1}  \partial_{\eta_2}^{\beta_2} \partial_{\zeta_1}^{\gamma_1}  \partial_{\zeta_2}^{\gamma_2} \left(m_{1}^{1,1}(\xi,\eta,\zeta)m_2(\eta,\zeta)\right)| \\
  \lesssim  \sum_{\substack{\beta'_1+\beta''_1=\beta_1\\\gamma_1'+\gamma_1''=\gamma_1}}\sum_{\substack{\beta'_2+\beta''_2=\beta_2\\ \gamma_2'+\gamma_2''=\gamma_2}}{1\over (|\xi_1|+|\eta_1|+|\zeta_1|)^{\alpha_1+\beta'_1+\gamma'_1}}{1\over (|\eta_1|+|\zeta_1|)^{\beta''_1+\gamma''_1}} \\
  \cdot
  {1\over (|\xi_2|+|\eta_2|+|\zeta_2|)^{\alpha_2+\beta'_2+\gamma'_2}}{1\over (|\eta_2|+|\zeta_2|)^{\beta''_2+\gamma''_2}}.
   \end{gather*}

   Thus, we only need to consider the symbol $m_1^{1,1}(\xi,\eta,\zeta)m_2(\eta,\zeta)$.
~\\

\subsubsection{The symbol $m_1^{1,1}(\xi,\eta,\zeta)m_2(\eta,\zeta)$}
~\\

To handle the symbol $m_1^{1,1}(\xi,\eta,\zeta)m_2(\eta,\zeta)$, we use the standard decomposition for $m_2(\eta,\zeta)$. We denote by
 ${\widehat \psi}'(u):=\sum_{k=-2}^2 \widehat \psi(2^{-k}u)$, so that $\widehat \psi'\in \mathcal{S}({\n})$ and $\spt \widehat \psi' \subset \{u \in \n: 2^{-3}\leq |u|\leq 2^3\}$. Then we can write

 \begin{eqnarray*}
   1&=&\sum_{{k_1},{k_1}'\in \z}\sum_{{k_2},{k_2}'\in \z}{\widehat \psi}(2^{-{k_1}}\eta_1){\widehat \psi}(2^{-{k'_1}}\zeta_1){\widehat \psi}(2^{-{k_2}}\eta_2){\widehat \psi}(2^{-{k_2}'}\zeta_2)\\
   & =& \big(\sum_{|{k_1}-{k_1}'|\leq2}+ \sum_{{k_1}-{k_1}'>2}+\sum_{{k_1}-{k_1}'<-2}\big)\big(\sum_{|{k_2}-{k_2}'|\leq2}+ \sum_{{k_2}-{k_2}'>2}+\sum_{{k_2}-{k_2}'<-2}\big) \\
   & & \qquad {\widehat \psi}(2^{-{k_1}}\eta_1){\widehat \psi}(2^{-{k'_1}}\zeta_1){\widehat \psi}(2^{-{k_2}}\eta_2){\widehat \psi}(2^{-{k_2}'}\zeta_2)\\
   &=&\sum_{k_1,k_2 \in \z}\big({\widehat \psi}(2^{-{k_1}}\eta_1){\widehat \psi}'(2^{-{k_1}}\zeta_1)+{\widehat \psi}(2^{-{k_1}}\eta_1){\widehat \varphi}(2^{-{k_1}}\zeta_1)+ {\widehat \varphi}(2^{-{k_1}}\eta_1){\widehat \psi}(2^{-{k_1}}\zeta_1)\big)\\
   & & \ \qquad \cdot \big({\widehat \psi}(2^{-{k_2}}\eta_2){\widehat \psi}'(2^{-{k_2}}\zeta_2)+{\widehat \psi}(2^{-{k_2}}\eta_2){\widehat \varphi}(2^{-{k_2}}\zeta_2)+ {\widehat \varphi}(2^{-{k_2}}\eta_2){\widehat \psi}(2^{-{k_2}}\zeta_2)\big)
 \end{eqnarray*}
Applying the above decomposition to $m_2(\eta,\zeta)$, by symmetry it suffices to consider the following cases:
\begin{eqnarray*}
   m_2^{1,1} (\eta,\zeta)&=& \sum_{{k_1},{k_2}\in \z} m_2(\eta,\zeta){\widehat \psi}(2^{-{k_1}}\eta_1){\widehat \psi}'(2^{-{k_1}}\zeta_1){\widehat \psi}(2^{-{k_2}}\eta_2){\widehat \psi}'(2^{-{k_2}}\zeta_2) \\
   m_2^{1,2} (\eta,\zeta)&=& \sum_{{k_1},{k_2}\in \z} m_2(\eta,\zeta){\widehat \psi}(2^{-{k_1}}\eta_1){\widehat \psi}'(2^{-{k_1}}\zeta_1){\widehat \psi}(2^{-{k_2}}\eta_2){\widehat \varphi}(2^{-{k_2}}\zeta_2) \\
    m_2^{2,2} (\eta,\zeta)&=& \sum_{{k_1},{k_2}\in \z} m_2(\eta,\zeta){\widehat \psi}(2^{-{k_1}}\eta_1){\widehat \varphi}(2^{-{k_1}}\zeta_1){\widehat \psi}(2^{-{k_2}}\eta_2){\widehat \varphi}(2^{-{k_2}}\zeta_2) \\
    m_2^{2,3} (\eta,\zeta)&=& \sum_{{k_1},{k_2}\in \z} m_2(\eta,\zeta){\widehat \psi}(2^{-{k_1}}\eta_1){\widehat \varphi}(2^{-{k_1}}\zeta_1){\widehat \varphi}(2^{-{k_2}}\eta_2){\widehat \psi}(2^{-{k_2}}\zeta_2).
\end{eqnarray*}
We now rewrite these using their Fourier expansions. For example,
 \begin{eqnarray*}
 m_2^{1,1} (\eta,\zeta)&=&\sum_{{k_1},{k_2}\in \z} \sum_{n_1,n_2\in\z}\sum_{n'_1,n'_2\in\z} c^{1,1}_{{k_1},{k_2},n_1,n'_1,n_2,n'_2} e^{i n_1 \frac{\eta_1}{2^{k_1}}} e^{i \frac{n_2}{4} \frac{\zeta_1}{2^{k_1}}}
e^{i n'_1 \frac{\eta_2}{2^{k_2}}} e^{i \frac{n'_2}{4} \frac{\zeta_2}{2^{k_2}}} \\
& &\qquad \qquad \qquad \qquad \qquad \cdot {\widehat \psi}_{k_1}(\eta_1) {\widehat \psi}'_{k_1}(\zeta_1){\widehat \psi}_{k_2}(\eta_2){\widehat \psi}'_{k_2}(\zeta_2),
\end{eqnarray*}
where $$\sup_{k_1,k_2 \in \z}|c^{1,1}_{k_1,k_2,n_1,n'_1,n_2,n'_2}|\lesssim (1+|n_1|+|n_2|)^{-M}(1+|n'_1|+|n'_2|)^{-M}$$ for sufficiently large $M>0$.

Moreover,for any index $\beta\in \mathbb{N}$ and $n\in \z$ we denote by
$${\widehat \psi}_{\beta,n}(u)=u^{\beta}e^{i n u}{\widehat \psi}(u),\  {\widehat \psi}'_{\beta,n}(u)=u^{\beta}e^{i \frac{n}{4} u}{\widehat \psi}'(u), \ {\widehat \varphi}_{\beta,n}(u)=u^{\beta}e^{i n u}{\widehat \varphi}(u). $$
Then we have
\begin{eqnarray*}
  \eta_1^{\beta'_1}\eta_2^{\beta'_2} \zeta_1^{\gamma'_1}\zeta_2^{\gamma'_2}  m_2^{1,1} (\eta,\zeta)&=& \sum_{{k_1},{k_2}\in \z} \sum_{n_1,n_2\in\z}\sum_{n'_1,n'_2\in\z} c^{1,1}_{{k_1},{k_2},n_1,n'_1,n_2,n'_2} 2^{{k_1}(\beta'_1+\gamma'_1)} 2^{{k_2}(\beta'_2+\gamma'_2)}  \\ \ & & \cdot {\widehat \psi}_{\beta'_1,n_1}(2^{-{k_1}}\eta_1){\widehat \psi}'_{\gamma'_1,n_2}(2^{-{k_1}}\zeta_1){\widehat \psi}_{\beta'_2,n'_1}(2^{-{k_2}}\eta_2){\widehat \psi}'_{\gamma'_2,n'_2}(2^{-{k_2}}\zeta_2) \\
    \eta_1^{\beta'_1}\eta_2^{\beta'_2} \zeta_1^{\gamma'_1}\zeta_2^{\gamma'_2} m_2^{1,2} (\eta,\zeta)&=& \sum_{{k_1},{k_2}\in \z} \sum_{n_1,n_2\in\z}\sum_{n'_1,n'_2\in\z} c^{1,2}_{{k_1},{k_2},n_1,n'_1,n_2,n'_2} 2^{{k_1}(\beta'_1+\gamma'_1)} 2^{{k_2}(\beta'_2+\gamma'_2)} \\ \ & & \cdot {\widehat \psi}_{\beta'_1,n_1}(2^{-{k_1}}\eta_1){\widehat \psi}'_{\gamma'_1,n_2}(2^{-{k_1}}\zeta_1){\widehat \psi}_{\beta'_2,n'_1}(2^{-{k_2}}\eta_2){\widehat \varphi}_{\gamma'_2,n'_2}(2^{-{k_2}}\zeta_2) \\
    \eta_1^{\beta'_1}\eta_2^{\beta'_2} \zeta_1^{\gamma'_1}\zeta_2^{\gamma'_2} m_2^{2,2} (\eta,\zeta)&=& \sum_{{k_1},{k_2}\in \z} \sum_{n_1,n_2\in\z}\sum_{n'_1,n'_2\in\z} c^{2,2}_{{k_1},{k_2},n_1,n'_1,n_2,n'_2} 2^{{k_1}(\beta'_1+\gamma'_1)} 2^{{k_2}(\beta'_2+\gamma'_2)} \\ \ & & \cdot {\widehat \psi}_{\beta'_1,n_1}(2^{-{k_1}}\eta_1){\widehat \varphi}_{\gamma'_1,n_2}(2^{-{k_1}}\zeta_1){\widehat \psi}_{\beta'_2,n'_1}(2^{-{k_2}}\eta_2){\widehat \varphi}_{\gamma'_2,n'_2}(2^{-{k_2}}\zeta_2) \\
     \eta_1^{\beta'_1}\eta_2^{\beta'_2} \zeta_1^{\gamma'_1}\zeta_2^{\gamma'_2}m_2^{2,3} (\eta,\zeta) &=& \sum_{{k_1},{k_2}\in \z} \sum_{n_1,n_2\in\z}\sum_{n'_1,n'_2\in\z} c^{2,3}_{{k_1},{k_2},n_1,n'_1,n_2,n'_2} 2^{{k_1}(\beta'_1+\gamma'_1)} 2^{{k_2}(\beta'_2+\gamma'_2)} \\ \ & & \cdot {\widehat \psi}_{\beta'_1,n_1}(2^{-{k_1}}\eta_1){\widehat \varphi}_{\gamma'_1,n_2}(2^{-{k_1}}\zeta_1){\widehat \varphi}_{\beta'_2,n'_1}(2^{-{k_2}}\eta_2){\widehat \psi}_{\gamma'_2,n'_2}(2^{-{k_2}}\zeta_2),
\end{eqnarray*}

\noindent where the coefficients, for all four pairs $(i_1,i_2)$ above, satisfy  $\sup_{k_1,k_2 \in \z}|c^{i_1,i_2}_{k_1,k_2,n_1,n'_1,n_2,n'_2}|\lesssim (1+|n_1|+|n_2|)^{-M}(1+|n'_1|+|n'_2|)^{-M}$.\\
\vskip0.5cm

In similar fashion, for $m_1^{1,1}(\xi,\eta,\zeta)$ in $\eqref{dcpm3}$, we denote
$$M_{\beta'_1,\beta'_2,\gamma'_1,\gamma'_2}(\xi):={1\over \beta'_1 ! \beta'_2 ! \gamma'_1 ! \gamma'_2!}\partial_{\eta_1}^{\beta'_1}\partial_{\eta_2}^{\beta'_2}
\partial_{\zeta_1}^{\gamma'_1}\partial_{\zeta_2}^{\gamma'_2}m_1(\xi,0,0,0,0).$$
Note that for any indices $\alpha_1,\alpha_2$ there holds $$|\partial^{\alpha_1}_{\xi_1}\partial^{\alpha_2}_{\xi_2}M_{\beta'_1,\beta'_2,\gamma'_1,\gamma'_2}(\xi)|\lesssim |\xi_1|^{-(\alpha_1+\beta'_1+\gamma'_1)} |\xi_2|^{-(\alpha_2+\beta'_2+\gamma'_2)},$$
which means we can expand in Fourier series to write as:
\begin{gather*}
M_{\beta'_1,\beta'_2,\gamma'_1,\gamma'_2}(\xi){\widehat \psi}(2^{-j_1}\xi_1){\widehat \psi}(2^{-j_2}\xi_2)\\
=\sum_{l_1,l_2\in \z} c^{\beta'_1,\beta'_2,\gamma'_1,\gamma'_2}_{j_1,j_2,l_1,l_2} 2^{-j_1(\beta'_1+\gamma'_1)} 2^{-j_2(\beta'_2+\gamma'_2)}{\widehat \psi}_{0,l_1}(2^{-j_1}\xi_1){\widehat \psi}_{0,l_2}(2^{-j_2}\xi_2),
\end{gather*}
where ${\widehat \psi}_{0,l_1}(u)=e^{inu}{\widehat \psi}(u), {\widehat \psi}_{0,l_2}(v)=e^{inu}{\widehat \psi}(v)$ for $u,v\in \mathbb{R}$ are defined as before, and
$$\sup_{j_1,j_2}|c^{\beta_1,\beta_2,\gamma_1,\gamma_2}_{j_1,j_2,l_1,l_2}|\lesssim (1+|l_1|)^{-M}(1+|l_2|)^{-M}.$$

Then if we denote by $d_1=\beta'_1+\gamma'_1$ and $d_2=\beta'_2+\gamma'_2$, and put everything together, we have that
\begin{eqnarray}
 & & m_1^{1,1}(\xi,\eta,\zeta)m_2(\eta,\zeta)\chi(\xi_1,\eta_1,\zeta_1)\chi(\xi_2,\eta_2,\zeta_2) \nonumber\\
&=&\sum c^{1,1} 2^{-(j_1-k_1)d_1}2^{-(j_2-k_2)d_2} {\widehat \psi}_{0,l_1}(\frac{\xi_1}{2^{j_1}}){\widehat \psi}_{0,l_2}(\frac{\xi_2}{2^{j_2}}){\widehat \varphi}(\frac{\eta_1}{2^{j_1-10}}) {\widehat \varphi}(\frac{\eta_2}{2^{j_2-10}}) \nonumber \\
& &  \qquad \cdot {\widehat \varphi}(\frac{\zeta_1}{2^{j_1-10}}) {\widehat \varphi}(\frac{\zeta_2}{2^{j_2-10}}) {\widehat \psi}_{\beta'_1,n_1}(\frac{\eta_1}{2^{{k_1}}}){\widehat \psi}'_{\gamma'_1,n_2}(\frac{\zeta_1}{2^{{k_1}}})
{\widehat \psi}_{\beta'_2,n'_1}(\frac{\eta_2}{2^{{k_2}}}){\widehat \psi}'_{\gamma'_2,n'_2}(\frac{\zeta_2}{2^{{k_2}}}) \nonumber \\
& &+ \sum c^{1,2} 2^{-(j_1-k_1)d_1}2^{-(j_2-k_2)d_2} {\widehat \psi}_{0,l_1}(\frac{\xi_1}{2^{j_1}}){\widehat \psi}_{0,l_2}(\frac{\xi_2}{2^{j_2}}){\widehat \varphi}(\frac{\eta_1}{2^{j_1-10}}) {\widehat \varphi}(\frac{\eta_2}{2^{j_2-10}}) \nonumber\\
& &  \qquad \cdot {\widehat \varphi}(\frac{\zeta_1}{2^{j_1-10}}) {\widehat \varphi}(\frac{\zeta_2}{2^{j_2-10}}) {\widehat \psi}_{\beta'_1,n_1}(\frac{\eta_1}{2^{{k_1}}}){\widehat \psi}'_{\gamma'_1,n_2}(\frac{\zeta_1}{2^{{k_1}}})
{\widehat \psi}_{\beta'_2,n'_1}(\frac{\eta_2}{2^{{k_2}}}){\widehat \varphi}_{\gamma'_2,n'_2}(\frac{\zeta_2}{2^{{k_2}}}) \nonumber\\
& &+\sum c^{2,2} 2^{-(j_1-k_1)d_1}2^{-(j_2-k_2)d_2} {\widehat \psi}_{0,l_1}(\frac{\xi_1}{2^{j_1}}){\widehat \psi}_{0,l_2}(\frac{\xi_2}{2^{j_2}}){\widehat \varphi}(\frac{\eta_1}{2^{j_1-10}}) {\widehat \varphi}(\frac{\eta_2}{2^{j_2-10}}) \nonumber\\
& &  \qquad \cdot {\widehat \varphi}(\frac{\zeta_1}{2^{j_1-10}}) {\widehat \varphi}(\frac{\zeta_2}{2^{j_2-10}}) {\widehat \psi}_{\beta'_1,n_1}(\frac{\eta_1}{2^{{k_1}}}){\widehat \varphi}_{\gamma'_1,n_2}(\frac{\zeta_1}{2^{{k_1}}})
{\widehat \psi}_{\beta'_2,n'_1}(\frac{\eta_2}{2^{{k_2}}}){\widehat \varphi}_{\gamma'_2,n'_2}(\frac{\zeta_2}{2^{{k_2}}}) \nonumber\\
& &+\sum c^{2,3} 2^{-(j_1-k_1)d_1}2^{-(j_2-k_2)d_2} {\widehat \psi}_{0,l_1}(\frac{\xi_1}{2^{j_1}}){\widehat \psi}_{0,l_2}(\frac{\xi_2}{2^{j_2}}){\widehat \varphi}(\frac{\eta_1}{2^{j_1-10}}) {\widehat \varphi}(\frac{\eta_2}{2^{j_2-10}}) \nonumber \\
\label{dcpm11}
& &  \qquad \cdot {\widehat \varphi}(\frac{\zeta_1}{2^{j_1-10}}) {\widehat \varphi}(\frac{\zeta_2}{2^{j_2-10}}) {\widehat \psi}_{\beta'_1,n_1}(\frac{\eta_1}{2^{{k_1}}}){\widehat \varphi}_{\gamma'_1,n_2}(\frac{\zeta_1}{2^{{k_1}}})
{\widehat \varphi}_{\beta'_2,n'_1}(\frac{\eta_2}{2^{{k_2}}}){\widehat \psi}'_{\gamma'_2,n'_2}(\frac{\zeta_2}{2^{{k_2}}}),
\end{eqnarray}
where the summation is taken over $j_1,j_2,k_1,k_2,l_1,l_2,n_1,n_2,n'_1,n'_2$ and
$$c^{i_1,i_2}=c^{\beta'_1,\beta'_2,\gamma'_1,\gamma'_2}_{j_1,j_2,l_1,l_2}c^{i_1,i_2}_{k_1,k_2,n_1,n'_1,n_2,n'_2}$$
for all the four pairs $(i_1,i_2)$ as above.

Note that when $j_1-k_1< 10$ or $j_2-k_2< 10$,  each of the four parts in $\eqref{dcpm11}$ must be zero.
{ Actually, if we take a look at the expressions in \eqref{dcpm11}, for example,
\begin{eqnarray*}
& &\sum c^{2,3} 2^{-(j_1-k_1)d_1}2^{-(j_2-k_2)d_2} {\widehat \psi}_{0,l_1}(\frac{\xi_1}{2^{j_1}}){\widehat \psi}_{0,l_2}(\frac{\xi_2}{2^{j_2}}){\widehat \varphi}(\frac{\eta_1}{2^{j_1-10}}) {\widehat \varphi}(\frac{\eta_2}{2^{j_2-10}}) \nonumber \\
& &  \qquad \cdot {\widehat \varphi}(\frac{\zeta_1}{2^{j_1-10}}) {\widehat \varphi}(\frac{\zeta_2}{2^{j_2-10}}) {\widehat \psi}_{\beta'_1,n_1}(\frac{\eta_1}{2^{{k_1}}}){\widehat \varphi}_{\gamma'_1,n_2}(\frac{\zeta_1}{2^{{k_1}}})
{\widehat \varphi}_{\beta'_2,n'_1}(\frac{\eta_2}{2^{{k_2}}}){\widehat \psi}'_{\gamma'_2,n'_2}(\frac{\zeta_2}{2^{{k_2}}}),
\end{eqnarray*}
we can see actually ${\widehat \varphi}(\frac{\eta_1}{2^{j_1-10}}){\widehat \psi}_{\beta'_1,n_1}(\frac{\eta_1}{2^{{k_1}}})=0$, since they have disjoint supports. More precisely,
\begin{gather*}
  \spt {\widehat \varphi}(\frac{\eta_1}{2^{j_1-10}}) \subset \{\eta_1: |\eta_1|\leq 2^{j_1-12}\},\\
   {\widehat \psi}_{\beta'_1,n_1}(\frac{\eta_1}{2^{{k_1}}})\subset  \{\eta_1:  2^{k_1-1} \leq |\eta_1|\leq 2^{k_1+1}\},\\
   \spt {\widehat \varphi}(\frac{\eta_1}{2^{j_1-10}}) \cap  \spt {\widehat \psi}_{\beta'_1,n_1}(\frac{\eta_1}{2^{{k_1}}})=\emptyset\quad \textit{if}\quad  j_1-k_1< 10.
\end{gather*}
Other terms are handled similarly.
}

Therefore,
we just need to consider the case $j_1-k_1\geq 10$ and $j_2-k_2\geq 10$. Moreover, when $j_1-k_1\geq 20$, ${\widehat \varphi}(\frac{\eta_1}{2^{j_1-10}})=1$ on $\spt {\widehat \psi}_{\beta'_1,n_1}(\frac{\eta_1}{2^{k_1}})\cup \spt {\widehat \varphi}_{\beta'_1,n_1}(\frac{\eta_1}{2^{k_1}})$ and  ${\widehat \varphi}(\frac{\zeta_1}{2^{j_1-10}})=1$ on $\spt {\widehat \psi}_{\gamma'_1,n_2}(\frac{\zeta_1}{2^{k_1}})\cup \spt {\widehat \varphi}_{\gamma'_1,n_2}(\frac{\zeta_1}{2^{k_1}})$. Further, when
$10<j_1-k_1<20$, one can see that the summation of the terms involving $\xi_1,\eta_1,\zeta_1$ gives a multiplier in $\mathcal{M}(\nnn)$. The same argument works for the other half of variables $\xi_2,\eta_2,\zeta_2$ as well based on the similar choice of $k_2,j_2$.

Due to the decay  in the coefficients $a_{j_1,j_2}:=c^{\beta'_1,\beta'_2,\gamma'_1,\gamma'_2}_{j_1,j_2,l_1,l_2}, b_{k_1,k_2}:=c^{i_1,i_2}_{k_1,k_2,n_1,n'_1,n_2,n'_2}$, we can fix $l_1,l_2,n_1,n_2,n'_1,n'_2$ and only take the summation over $j_1,j_2,k_1,k_2$. And without loss of generality we can assume $|a_{j_1,j_2}|\leq 1, |b_{k_1,k_2}|\leq 1$.

~\\

Thus, when $10\leq j_1-k_1\leq 20$ and $10\leq j_2-k_2\leq 20$, the above multipliers belong to $B\mathcal{M}(\n^6)$ and Theorem $\ref{bipcoifman}$ gives us the desired estimate. So we only need to consider the case
$j_1-k_1\geq20$, $j_2-k_2\geq 20$, and the cases $j_1-k_1\geq20,10\leq j_2-k_2\leq 20$ and $10\leq j_1-k_1\leq20, j_2-k_2\geq 20$ correspond to the estimates like \eqref{mid1} and \eqref{mid2}.
~\\

For the reduction in \eqref{dcpm11}, one will see later what really matters there is the ``type" of those ${\widehat \psi}$ and ${\widehat \varphi}$ functions, i.e, whether $0$ is contained in the supports of functions. We call these functions $\Psi$ type and $\Phi$ type functions respectively. Because of that, we can simplify the notations for operators in \eqref{dcpm11} as below, where we also denote by $d_i=\beta'_i+\gamma'_i<N$ for $i=1,2$.

\begin{gather*}
 m_{d_1,d_2}^{1}(\xi,\eta,\zeta)=\\
\sum_{\substack{j_1-k_1\geq 20\\j_2-k_2\geq 20}} a_{j_1,j_2}b_{k_1,k_2} 2^{-(j_1-k_1)d_1}2^{-(j_2-k_2)d_2} {\widehat \psi^{0}}(\frac{\xi_1}{2^{j_1}}) {\widehat {\widetilde  \psi^{0}}}(\frac{\xi_2}{2^{j_2}}) {\widehat \psi^{1}}(\frac{\eta_1}{2^{{k_1}}}){\widehat \psi^{2}}(\frac{\zeta_1}{2^{{k_1}}})
{\widehat {\widetilde \psi^{1}}}(\frac{\eta_2}{2^{{k_2}}}) {\widehat {\widetilde\psi^{2}}}(\frac{\zeta_2}{2^{{k_2}}}),
\end{gather*}
\vskip0.3cm
\begin{gather*}
 m_{d_1,d_2}^{2}(\xi,\eta,\zeta)=\\
\sum_{\substack{j_1-k_1\geq 20\\j_2-k_2\geq 20}} a_{j_1,j_2}b_{k_1,k_2} 2^{-(j_1-k_1)d_1}2^{-(j_2-k_2)d_2} {\widehat \psi^{0}}(\frac{\xi_1}{2^{j_1}}) {\widehat {\widetilde \psi^{0}}}(\frac{\xi_2}{2^{j_2}}) {\widehat \psi^{1}}(\frac{\eta_1}{2^{{k_1}}}){\widehat \psi^{2}}(\frac{\zeta_1}{2^{{k_1}}})
 {\widehat {\widetilde \psi^{1}}}(\frac{\eta_2}{2^{{k_2}}}) {\widehat {\widetilde \varphi^{0}}}(\frac{\zeta_2}{2^{{k_2}}}),
\end{gather*}
\vskip0.3cm
\begin{gather*}
  m_{d_1,d_2}^{3}(\xi,\eta,\zeta)=\\
\sum_{\substack{j_1-k_1\geq 20\\j_2-k_2\geq 20}} a_{j_1,j_2}b_{k_1,k_2} 2^{-(j_1-k_1)d_1}2^{-(j_2-k_2)d_2} {\widehat \psi^{0}}(\frac{\xi_1}{2^{j_1}}) {\widehat {\widetilde \psi^{0}}}(\frac{\xi_2}{2^{j_2}}) {\widehat \psi^{1}}(\frac{\eta_1}{2^{{k_1}}}){\widehat \varphi^{0}}(\frac{\zeta_1}{2^{{k_1}}})
 {\widehat {\widetilde \psi^{1}}}(\frac{\eta_2}{2^{{k_2}}}) {\widehat {\widetilde \varphi^{0}}}(\frac{\zeta_2}{2^{{k_2}}}) \nonumber ,
\end{gather*}

\vskip0.3cm
\begin{gather*}
  m_{d_1,d_2}^{4}(\xi,\eta,\zeta)=\\
  \sum_{\substack{j_1-k_1\geq 20\\j_2-k_2\geq 20}} a_{j_1,j_2}b_{k_1,k_2} 2^{-(j_1-k_1)d_1}2^{-(j_2-k_2)d_2} {\widehat \psi^{0}}(\frac{\xi_1}{2^{j_1}}) {\widehat {\widetilde \psi^{0}}}(\frac{\xi_2}{2^{j_2}}) {\widehat \psi^{1}}(\frac{\eta_1}{2^{{k_1}}}){\widehat \varphi}^{0}(\frac{\zeta_1}{2^{{k_1}}})
 {\widehat {\widetilde \varphi^{0}}}(\frac{\eta_2}{2^{{k_2}}}) {\widehat {\widetilde\psi^{1}}}(\frac{\zeta_2}{2^{{k_2}}}) \nonumber .
\end{gather*}

Here ${\widehat \psi^i},{\widehat \varphi^0},  {\widehat {\widetilde \psi^i}}, {\widehat {\widetilde \varphi^0}}$ satisfy
\begin{itemize}
\item[] $\spt {\widehat \psi}^i, \spt  {\widehat {\widetilde \psi^i}} \subset \{u\,|\,2^{-1}\leq |u|\leq 2\},\qquad i=0,1,$\\
\item[] $\spt {\widehat \psi}^2, \spt  {\widehat {\widetilde \psi^2}} \subset \{u\,|\,2^{-3}\leq |u|\leq 2^3\},$\\
\item[] $\spt {\widehat \varphi}^0, \spt  {\widehat {\widetilde \varphi^0}} \subset \{u\,| \,|u|\leq 2^{-2}\}.$
\end{itemize}

As previously mentioned, because of the ``type" of the functions, we do not distinguish between ${\widehat \psi^i}$ and $ {\widehat {\widetilde \psi^i}}$ ($i=0,1,2$), between ${\widehat \varphi^0}$ and ${\widehat {\widetilde  \varphi^0}}$, and we denote them to be ${\widehat \psi}$ and ${\widehat \varphi}$ respectively. But note that ${\widehat \psi},{\widehat \varphi}$ are different from the ones in $\eqref{psidef}$ and $\eqref{varphidef}$. And we use the notations
$$ \widehat{\Delta_j f} (\xi)={\widehat \psi}({\xi \over 2^j})\hat f(\xi),\qquad \widehat{S_k f}(\xi)={\widehat \varphi}({\xi \over 2^k})\hat f(\xi).$$
Then finally we reduce our original problem to the study of the following cases
\begin{eqnarray*}
  T_{d_1,d_2}^1&=&\sum_{\substack{j_1-k_1\geq 20\\j_2-k_2\geq 20}} a_{j_1,j_2}b_{k_1,k_2} 2^{-(j_1-k_1)d_1}2^{-(j_2-k_2)d_2} \Delta_{j_1}\Delta_{j_2}f \ \Delta_{k_1}\Delta_{k_2}g \ \Delta_{k_1}\Delta_{k_2}h \\
  T_{d_1,d_2}^2&=&\sum_{\substack{j_1-k_1\geq 20\\j_2-k_2\geq 20}} a_{j_1,j_2}b_{k_1,k_2} 2^{-(j_1-k_1)d_1}2^{-(j_2-k_2)d_2} \Delta_{j_1}\Delta_{j_2}f \ \Delta_{k_1}\Delta_{k_2}g \ \Delta_{k_1}S_{k_2}h \\
   T_{d_1,d_2}^3&=&\sum_{\substack{j_1-k_1\geq 20\\j_2-k_2\geq 20}} a_{j_1,j_2}b_{k_1,k_2} 2^{-(j_1-k_1)d_1}2^{-(j_2-k_2)d_2} \Delta_{j_1}\Delta_{j_2}f \ \Delta_{k_1}\Delta_{k_2}g \ S_{k_1}S_{k_2}h\\
   T_{d_1,d_2}^4&=& \sum_{\substack{j_1-k_1\geq 20\\j_2-k_2\geq 20}} a_{j_1,j_2}b_{k_1,k_2} 2^{-(j_1-k_1)d_1}2^{-(j_2-k_2)d_2} \Delta_{j_1}\Delta_{j_2}f \ \Delta_{k_1}S_{k_2}g \ S_{k_1}\Delta_{k_2}h
\end{eqnarray*}
~\\

Note the fact that in $T_{d_1,d_2}^i$ $(i=1,2,3,4)$, the support for each of the Fourier transforms of $(\Delta_{k_1}\Delta_{k_2}g\Delta_{k_1}S_{k_2}h)(x_1,x_2),\ (\Delta_{k_1}\Delta_{k_2}gS_{k_1}S_{k_2}h)(x_1,x_2),\ (\Delta_{k_1}S_{k_2}g S_{k_1}\Delta_{k_2}h)(x_1,x_2)$  is contained in $\{|u_1|\lesssim 2^{k_1}, |u_2|\lesssim 2^{k_2}\}$.  Thus the Fourier transform of
$T_{d_1,d_2}^i$ $(i=1,2,3,4)$ is supported in $\{ |u_1|\approx 2^{j_1}, |u_2|\approx 2^{j_2}\}$ for fixed  $j_1,j_2$. Moreover, from the argument below, one can see  it suffices to consider the case for $d_1=d_2=0$ since $j_1-k_1\geq 20, j_2-k_2\geq 20.$
~\\

\subsection{The $L^r$ boundedness of H\"older type}
~\\

In this subsection we study the $L^{p_1} \times L^{p_2}\times L^{p_3}\to L^r$ estimate for $1<p_1,p_2,p_3 < \infty$ for the operators $T_{d_1,d_2}^i$ $(i=1,2,3,4)$.
\medskip

\subsubsection{$d_1,d_2>0$}
 ~\\

The approach for the case $d_1,d_2>0$ works for all $T^i_{d_1,d_2},(1\leq i\leq 4)$. Consider $T^2_{d_1,d_2}$ for example, with other cases treated similarly. Since the support of the Fourier transform of $T^{2}_{d_1,d_2}$ for fixed $j_1,j_2$ is included in $\{|u_1|\approx 2^{j_1},|u_2|\approx 2^{j_2}\}$, there holds

\begin{eqnarray*}
 \|T^{2}_{d_1,d_2}(f,g,h)\|_{L^{r}}
   &\lesssim &\big\|\{ \sum_{j_1,j_2\in \z}\big|a_{j_1,j_2}\Delta_{j_1}\Delta_{j_2}f\big(\sum_{k_1=-\infty}^{j_1-20}\sum_{k_2=-\infty}^{j_2-20}b_{k_1,k_2}
  2^{-(j_1-k_1)d_1} 2^{-(j_2-k_2)d_2} \\
  & & \qquad \qquad \qquad \qquad  \Delta_{k_1}\Delta_{k_2}g\, \Delta_{k_1}S_{k_2}h\big)\big|^2\}^{1\over 2}\big\|_{L^r}\\
  &\lesssim& \big\| \{\sum_{j_1,j_2\in \z}\big|a_{j_1,j_2}\Delta_{j_1}\Delta_{j_2}f \big|^2\}^{1\over 2} \big(\sup_{k_1,k_2}|\Delta_{k_1}\Delta_{k_2}g|\big)\big(\sup_{k_1,k_2}|\Delta_{k_1}S_{k_2}h|\big)\big\|_{L^r}\\
  &\lesssim& \|\{\sum_{j_1,j_2\in \z}\big|a_{j_1,j_2}\Delta_{j_1}\Delta_{j_2}f \big|^2\}^{1\over 2}\|_{L^{p_1}}
   \|\sup_{k_1,k_2}|\Delta_{k_1}\Delta_{k_2}g|\|_{L^{p_2}} \|\sup_{k_1,k_2}|\Delta_{k_1}S_{k_2}h|\|_{L^{p_3}}\\
   &\lesssim& \|\{\sum_{j_1,j_2\in \z}\big|\Delta_{j_1}\Delta_{j_2}f \big|^2\}^{1\over 2}\|_{L^{p_1}}
   \|M_s g\|_{L^{p_2}} \|M_sh\|_{L^{p_3}}\\
   &\lesssim& \|f\|_{L^{p_1}}  \|g\|_{L^{p_2}}  \|h\|_{L^{p_3}},
\end{eqnarray*}

where $0<1/r=1/p_1+1/p_2+1/p_3$ with $1<p_1,p_2,p_3< \infty$, $M_s$ is the strong maximal operator.

~\\

\subsubsection{$d_1=d_2=0$}
~\\
Note that for $d_1>0$ and $d_2=0$,  or $d_1=0$ and $d_2>0$, we can treat the decay factors $2^{-(j_1-k_1)d_1}$ or $2^{-(j_2-k_2)d_2}$ as  uniform constants. That means they are particular cases of $d_1=d_2=0$. For the case $d_1=d_2=0$,  in $\eqref{dcpm3}$ we have  $\beta'_1=\gamma'_1=\beta'_2=\gamma'_2=0$. Here we consider $T^4_{d_1,d_2}$, and one can check the following argument is valid for the other three operators.

We write the the $L^r$ norm of $T^4_{0,0}$ as
\begin{eqnarray}
  & &\left\|\sum_{\substack{j_1-k_1\geq 20\\j_2-k_2\geq 20}} a_{j_1,j_2}b_{k_1,k_2} \Delta_{j_1}\Delta_{j_2}f \ \Delta_{k_1}S_{k_2}g \ S_{k_1}\Delta_{k_2}h\right\|_{L^r} \nonumber \\
  &=&\left\|\sum_{k_1,k_2 \in \z} b_{k_1,k_2}\Delta_{k_1}S_{k_2}g\  S_{k_1} \Delta_{k_2} h \sum_{j_1,j_2=k_1,k_2+20}^{\infty} a_{j_1,j_2} \Delta_{j_1}\Delta_{j_2} f \  \right\|_{L^r} \label{goal}
\end{eqnarray}
We first consider the part $\sum_{j_1,j_2=k_1,k_2+20}^{\infty} \Delta_{j_1}\Delta_{j_2} f$.

In one-parameter case, since $\spt \widehat \psi_j \subset \{2^{j-1}\leq |\xi| \leq 2^{j+1}\}$ ($j\in\z, \xi\in\n$), we have
$$\spt \sum_{j=k+20}^\infty \widehat \psi_j \subset \{2^{k+19}\leq |\xi|\} .$$
 Now we choose a function  $\phi$ such that $\spt \widehat  \phi_k  \subset \{|\xi| \leq 2^{k-2}\}$ and $\widehat  \phi_k=1$ on $\{|\xi| \leq 2^{k-4}\}$.
Note that $$1-\widehat \phi_k=1 \quad \textit{on} \quad \spt \sum_{j=k+20}^\infty \widehat \psi_j,$$
      $$1-\widehat \phi_k=0 \quad \textit{on} \quad \{|\xi|\leq 2^{k-4}\}.$$
Thus,
\begin{eqnarray*}
  \sum_{j=k+20}^\infty \widehat \psi_j&=&(1-\widehat \phi_k)(\sum_{j=k+20}^\infty \widehat \psi_j)=(1-\widehat \phi_k) (\sum_{j \in \z} \widehat \psi_j- \sum_{j=-\infty}^{k+19} \widehat \psi_j )\\
  &=&(1-\widehat \phi_k) (\sum_{j \in \z} \widehat \psi_j- \sum_{j=k-5}^{k+19} \widehat \psi_j ) =(1-\widehat \phi_k) (\sum_{j \in \z} \widehat \psi_j- \sum_{j\approx k} \widehat \psi_j ) \\
 &=&  (1-\widehat \phi_k) (\sum_{j \in \z} \widehat \psi_j-  \widehat {\widetilde \psi_k} ).
\end{eqnarray*}
Here we use the fact $(1-\widehat \phi_k)(\sum_{j=-\infty}^{k-6} \widehat \psi_j)$=0 since they have disjoint supports.

That means
\begin{eqnarray*}
  \sum_{j=k+20}^\infty \Delta_j f =  (1-S_k)  \left( \sum_{j \in \z} \Delta_j -\widetilde \Delta_k\right) f.
\end{eqnarray*}

\vskip0.3cm

Now we come back to the bi-parameter case. In the following arguments, we will omit the uniformly bounded constants $a_{j_1,j_2}$ and $b_{k_1,k_2}$ in \eqref{goal}for simplicity, since one can see they do not play an essential role in our argument

\begin{eqnarray}
 & & \sum_{j_1=k_1+20}^\infty   \sum_{j_2=k_2+20}^\infty \Delta_{j_1} \Delta_{j_2} f  \nonumber\\
  &=&  (1-S_{k_1})(1-S_{k_2})\left(\sum_{{j_1}\in \z}\Delta_{j_1}-\widetilde \Delta_{k_1}\right)\left(\sum_{{j_2}\in \z}\Delta_{j_2}-\widetilde \Delta_{k_2}\right) f \nonumber\\
  &=&  (1-S_{k_1}-S_{k_2}+ S_{k_1} S_{k_2})\left(\sum_{{j_1,j_2}\in \z}\Delta_{j_1}\Delta_{j_2}-\widetilde \Delta_{k_1}\sum_{j_2\in \z} \Delta_{j_2}  -\sum_{j_1\in \z} \Delta_{j_1} \widetilde \Delta_{k_2}+\widetilde \Delta_{k_1}\widetilde \Delta_{k_2}\right) f \nonumber \\
  &:=& (1-S_{k_1}-S_{k_2}+ S_{k_1} S_{k_2})\left(O_1(f)+O_2(f)+O_3(f)+O_4(f)\right) \label{opdcp}
\end{eqnarray}

\vskip0.2cm
We consider $(1-S_{k_1}-S_{k_2}+ S_{k_1} S_{k_2})O_1(f)$ first.  Using $O_1(f)$,  our operator becomes
\begin{eqnarray*}
 I_1:&=&\left\|\sum_{k_1,k_2 \in \z} \Delta_{k_1}S_{k_2}g\  S_{k_1} \Delta_{k_2}h \  O_1(f) \  \right\|_{L^r}\\
 &\lesssim & \left\|\sum_{k_1,k_2 \in \z} \Delta_{k_1}S_{k_2}g\  S_{k_1} \Delta_{k_2}h \right\|_{L^s} \left\| \  O_1(f) \  \right\|_{L^{p_1}} \\
 &\lesssim& \left\| g \right\|_{L^{p_2}} \left\| h \right\|_{L^{p_3}} \left\| \  O_1(f) \  \right\|_{L^{p_1}},
\end{eqnarray*}
where the first inequality follows from the H\"older's inequality with $1/p_2+1/p_3=1/s$, and the second inequality follows from the classical bilinear bi-parameter multiplier boundedness.

Using $S_{k_1} S_{k_2}O_1(f)$, we have

\begin{eqnarray*}
   II_1:&=&\left\|\sum_{k_1,k_2 \in \z} \Delta_{k_1}S_{k_2}g\  S_{k_1} \Delta_{k_2}h \ S_{k_1} S_{k_2} O_1(f) \  \right\|_{L^r}\\
 &\lesssim& \left\| g \right\|_{L^{p_2}} \left\| h \right\|_{L^{p_3}} \left\| \  O_1(f) \  \right\|_{L^{p_1}},
\end{eqnarray*}
where the inequality follows from the classical trilinear bi-parameter boundedness. Then our desired H\"older type estimate would hold since the classical linear theory gives
$$\left\| O_1(f)\right\|_{L^{p_1}}=\left\| \sum_{{j_1,j_2}\in \z}\Delta_{j_1}\Delta_{j_2}f\right\|_{L^{p_1}}\lesssim \left\| f \right\|_{L^{p_1}}, \quad 1<p_1<\infty.$$

Note that actually $I_1$ corresponds to the boundedness of the following trilinear Fourier multiplier.
\begin{eqnarray*}
T_{I_1}&=&\int \sum_{k_1,k_2}\widehat \psi_{k_1}(\eta_1) \widehat \varphi_{k_2}(\eta_2) \widehat \varphi_{k_1}(\zeta_1) \widehat \psi_{k_2}(\zeta_2) \widehat O_1(f)(\xi)\hat g(\eta)  \hat h(\zeta)  e^{2\pi i x(\xi+\eta+\zeta)} d\xi d\eta d\zeta \\
&=&\int\left( \sum_{k_1}\widehat \psi_{k_1}(\eta_1) \widehat \varphi_{k_1}(\zeta_1) \right)\left( \sum_{k_2}\widehat \varphi_{k_2}(\eta_2) \widehat \psi_{k_2}(\zeta_2) \right) \widehat O_1(f)(\xi)\hat g(\eta)  \hat h(\zeta)  e^{2\pi i x(\xi+\eta+\zeta)} d\xi d\eta d\zeta \\
&=& \int m'(\eta_1,\zeta_1)m''(\eta_2,\zeta_2) \widehat O_1(f)(\xi)\hat g(\eta)  \hat h(\zeta)  e^{2\pi i x(\xi+\eta+\zeta)} d\xi d\eta d\zeta \\
&=&  \int m'(\eta_1,\zeta_1)m''(\eta_2,\zeta_2) \hat g(\eta)  \hat h(\zeta)  e^{2\pi i x(\eta+\zeta)}  d\eta d\zeta \cdot O_1 (f)(x)
\end{eqnarray*}
This  trilinear operator has a special bilinear symbol $m'(\eta_1,\zeta_1)m''(\eta_2,\zeta_2)$, which is actually the product of $ O_1 (f)(x)$ and a bilinear bi-parameter multiplier.  Then we can take advantage of the H\"older's inequality and a bilinear bi-parameter result. More precisely,

\begin{eqnarray*}
 & & \|T_{I_1}(f,g,h)\|_{L^r} \\
 &\lesssim& \left\|\int m'(\eta_1,\zeta_1)m''(\eta_2,\zeta_2) \hat g(\eta)  \hat h(\zeta)  e^{2\pi i x(\eta+\zeta)}  d\eta d\zeta \right\|_{L^s} \left\|O_1 (f)(x)\right\|_{L^{p_1}}\\
 &\lesssim&  \left\| g \right\|_{L^{p_2}} \left\| h \right\|_{L^{p_3}} \left\| \  f \  \right\|_{L^{p_1}}.
\end{eqnarray*}

Similarly, the boundedness $II_1$ corresponds to the trilinear Fourier multiplier
\begin{eqnarray*}
T_{II_1}&=&\int \sum_{k_1,k_2} \widehat \varphi_{k_1}(\xi_1) \widehat \varphi_{k_2}(\xi_2)  \widehat \psi_{k_1}(\eta_1) \widehat \varphi_{k_2}(\eta_2) \widehat \varphi_{k_1}(\zeta_1) \widehat \psi_{k_2}(\zeta_2) \widehat O_1(f)(\xi)\hat g(\eta)  \hat h(\zeta)  e^{2\pi i x(\xi+\eta+\zeta)} d\xi d\eta d\zeta \\
&=&\int\left( \sum_{k_1} \widehat \varphi_{k_1}(\xi_1)\widehat \psi_{k_1}(\eta_1) \widehat \varphi_{k_1}(\zeta_1) \right)\left( \sum_{k_2}\widehat \varphi_{k_2}(\xi_2) \widehat \varphi_{k_2}(\eta_2) \widehat \psi_{k_2}(\zeta_2) \right) \\
 & & \quad \widehat O_1(f)(\xi)\hat g(\eta)  \hat h(\zeta)  e^{2\pi i x(\xi+\eta+\zeta)} d\xi d\eta d\zeta \\
&=& \int m'(\xi_1,\eta_1,\zeta_1)m''(\xi_2,\eta_2,\zeta_2) \widehat O_1(f)(\xi)\hat g(\eta)  \hat h(\zeta)  e^{2\pi i x(\xi+\eta+\zeta)} d\xi d\eta d\zeta.
\end{eqnarray*}
Note that this is a standard bi-parameter trilinear Fourier multiplier, and that's why its H\"older type estimate holds.


Then we consider the terms like $S_{k_1}O(f)$ in \eqref{opdcp}
\begin{eqnarray}
   III_1:&=&\left\|\sum_{k_1,k_2 \in \z} \Delta_{k_1}S_{k_2}g\  S_{k_1} \Delta_{k_2}h \ S_{k_1} O_1(f) \  \right\|_{L^r} \nonumber
\end{eqnarray}
Similar as before, this corresponds to the following trilinear operator
\begin{eqnarray*}
T_{III_1}&=&\int \sum_{k_1,k_2} \widehat \varphi_{k_1}(\xi_1)  \widehat \psi_{k_1}(\eta_1) \widehat \varphi_{k_2}(\eta_2) \widehat \varphi_{k_1}(\zeta_1) \widehat \psi_{k_2}(\zeta_2) \widehat O_1(f)(\xi)\hat g(\eta)  \hat h(\zeta)  e^{2\pi i x(\xi+\eta+\zeta)} d\xi d\eta d\zeta \\
&=&\int\left( \sum_{k_1} \widehat \varphi_{k_1}(\xi_1)\widehat \psi_{k_1}(\eta_1) \widehat \varphi_{k_1}(\zeta_1) \right)\left( \sum_{k_2} \widehat \varphi_{k_2}(\eta_2) \widehat \psi_{k_2}(\zeta_2) \right) \\
 & & \quad \widehat O_1(f)(\xi)\hat g(\eta)  \hat h(\zeta)  e^{2\pi i x(\xi+\eta+\zeta)} d\xi d\eta d\zeta \\
&=& \int m'(\xi_1,\eta_1,\zeta_1)m''(\eta_2,\zeta_2) \widehat O_1(f)(\xi)\hat g(\eta)  \hat h(\zeta)  e^{2\pi i x(\xi+\eta+\zeta)} d\xi d\eta d\zeta.
\end{eqnarray*}

\begin{remark}
\label{rmk1}
Note that this trilinear  multiplier has a symbol $$m(\xi,\eta,\zeta)=m_1(\xi_1,\eta_1,\zeta_1)m_2(\eta_2,\zeta_2),$$ where one variable is missing for the second parameter, i.e., the function $f$ is not actually transformed in the second  variable. Such a symbol can be interpreted as an intermediate case between the previous $I_1$ and $II_1$,  and that's why we conjecture that H\"older type estimates should hold for operator.
If we make the assumption that this operator is bounded, we have the following.
 \begin{eqnarray*}
  \|T_{III_1}(f,g,h)\|_{L^r}  &\lesssim&  \left\| g \right\|_{L^{p_2}} \left\| h \right\|_{L^{p_3}} \left\| \  O_1(f) \  \right\|_{L^{p_1}} \lesssim \left\| g \right\|_{L^{p_2}} \left\| h \right\|_{L^{p_3}} \left\| \  f \  \right\|_{L^{p_1}}
\end{eqnarray*}
\end{remark}

\medskip

Now we consider $(1-S_{k_1}-S_{k_2}+S_{k_1}S_{k_2})O_4(f)$. We still consider the three parts  $O_4(f)$, $S_{k_1}S_{k_2}O_4(f)$ and $S_{k_1}O_4(f)$. Recall
$$O_4(f)=\left(\widetilde \Delta_{k_1} \widetilde \Delta_{k_2}\right) f .$$
Thus, with $O_4(f)$, we have
\begin{eqnarray*}
 I_4:&=&\left\|\sum_{k_1,k_2 \in \z} \Delta_{k_1}S_{k_2}g\  S_{k_1} \Delta_{k_2}h \  O_4(f) \  \right\|_{L^r}\\
 &=& \left\|\sum_{k_1,k_2 \in \z} \Delta_{k_1}S_{k_2}g\  S_{k_1} \Delta_{k_2}h \  \widetilde  \Delta_{k_1}\widetilde \Delta_{k_2} f \right\|_{L^{r}} \\
 &\lesssim& \left\| g \right\|_{L^{p_2}} \left\| h \right\|_{L^{p_3}} \left\| f \  \right\|_{L^{p_1}},
\end{eqnarray*}
where the estimate follows from the classical trilinear bi-parameter multiplier boundedness, as what we argued for $T_{I_1}$.

With  $S_{k_1}S_{k_2}O_4(f)$,
\begin{eqnarray*}
 II_4:&=&\left\|\sum_{k_1,k_2 \in \z} \Delta_{k_1}S_{k_2}g\  S_{k_1} \Delta_{k_2}h \ S_{k_1} S_{k_2} O_4(f) \  \right\|_{L^r}\\
 &=& \left\|\sum_{k_1,k_2 \in \z} \Delta_{k_1}S_{k_2}g\  S_{k_1} \Delta_{k_2}h \ S_{k_1} S_{k_2} \widetilde \Delta_{k_1} \widetilde \Delta_{k_2} f \right\|_{L^{r}}
\end{eqnarray*}
Recall that
$$\spt \widehat \phi_k\subset \{|\xi|\leq 2^{k-2} \},$$
$$\spt \widehat \psi_k\subset \{ 2^{k-1}\leq |\xi|\leq 2^{k+1} \},$$
$$\spt \widehat {\widetilde \psi_k} \subset \{ 2^{k-6}\leq |\xi|\leq 2^{k+20} \}$$
Thus, we can write $$S_{k_1} S_{k_2} \widetilde \Delta_{k_1} \widetilde \Delta_{k_2} f=   \widetilde \Delta'_{k_1} \widetilde  \Delta'_{k_2}f,$$
where $\widetilde  \Delta'_{k} f =\left(\widehat {\widetilde \psi'_{k} } \hat f\right)^\vee= \left(\widehat \phi_k \widehat {\widetilde  \psi_{k} } \hat f\right)^\vee$. Note that $\spt \widehat {\widetilde \psi'_{k}} \subset \{2^{k-6}\leq |\xi|\leq 2^{k-2}\} $. Then the estimate

\begin{eqnarray*}
 II_4:&=&\left\|\sum_{k_1,k_2 \in \z} \Delta_{k_1}S_{k_2}g\  S_{k_1} \Delta_{k_2}h \ S_{k_1} S_{k_2} O_4(f) \  \right\|_{L^r}\\
 &=& \left\|\sum_{k_1,k_2 \in \z} \Delta_{k_1}S_{k_2}g\  S_{k_1} \Delta_{k_2}h\  \widetilde \Delta'_{k_1} \widetilde \Delta'_{k_2} f \right\|_{L^{r}}\\
 &\lesssim& \left\| g \right\|_{L^{p_2}} \left\| h \right\|_{L^{p_3}} \left\| f \  \right\|_{L^{p_1}}
\end{eqnarray*}
follows from the classical trilinear bi-parameter boundedness.

With $S_{k_1}O_4(f)$, we have

\begin{eqnarray*}
 III_4:&=&\left\|\sum_{k_1,k_2 \in \z} \Delta_{k_1}S_{k_2}g\  S_{k_1} \Delta_{k_2}h \ S_{k_1} O_4(f) \  \right\|_{L^r}\\
  &=& \left\|\sum_{k_1,k_2 \in \z} \Delta_{k_1}S_{k_2}g\  S_{k_1} \Delta_{k_2}h \  S_{k_1} \widetilde \Delta_{k_1} \widetilde \Delta_{k_2} f \right\|_{L^{r}}\\
 &=& \left\|\sum_{k_1,k_2 \in \z} \Delta_{k_1}S_{k_2}g\  S_{k_1} \Delta_{k_2}h\  \widetilde \Delta'_{k_1} \widetilde \Delta_{k_2} f \right\|_{L^{r}}\\
 &\lesssim& \left\| g \right\|_{L^{p_2}} \left\| h \right\|_{L^{p_3}} \left\| f \  \right\|_{L^{p_1}}
\end{eqnarray*}
where we use the same trick
$$S_{k_1} \widetilde \Delta_{k_1} \widetilde \Delta_{k_2} f=   \widetilde \Delta'_{k_1} \widetilde  \Delta_{k_2}f.$$

\vskip0.2cm
Then we consider the last situation left in \eqref{opdcp},  which is $O_2(f)$. Using the same argument as above, we can write
\begin{eqnarray*}
  & &(1-S_{k_1}-S_{k_2}+ S_{k_1} S_{k_2})O_2 (f)\\
  &=& (1-S_{k_1}-S_{k_2}+ S_{k_1} S_{k_2})\widetilde \Delta_{k_1}\sum_{j_2\in \z} \Delta_{j_2} (f)\\
  &=& \widetilde \Delta_{k_1} \left(\sum_{j_2\in \z} \Delta_{j_2} f\right)- \widetilde \Delta'_{k_1} \left(\sum_{j_2\in \z} \Delta_{j_2} f\right)-  \widetilde \Delta_{k_1}S_{k_2}\left(\sum_{j_2\in \z}  \Delta_{j_2} f\right)+ \widetilde \Delta'_{k_1} S_{k_2} \left(\sum_{j_2\in \z}  \Delta_{j_2} f\right) \\
  &:=&  \widetilde \Delta_{k_1} Q(f)- \widetilde \Delta'_{k_1} Q(f)-  \widetilde \Delta_{k_1}S_{k_2}Q(f)+ \widetilde \Delta'_{k_1} S_{k_2} Q(f),
\end{eqnarray*}
where obviously $$\|Q(f)\|_{L^{p_1}}=\|\sum_{j_2\in \z}  \Delta_{j_2} f\|_{L^{p_1}} \lesssim \| f\|_{L^{p_1}} .$$

Now we consider the $L^r$ norm of .

\begin{eqnarray}
\nonumber
 & &\left\|\sum_{k_1,k_2 \in \z} \Delta_{k_1}S_{k_2}g\ S_{k_1} \Delta_{k_2}h \  (1-S_{k_1}-S_{k_2}+ S_{k_1} S_{k_2})O_2 (f)  \right\|_{L^r}\\
\nonumber &=& \left\|\sum_{k_1,k_2 \in \z} \Delta_{k_1}S_{k_2}g\  S_{k_1} \Delta_{k_2}h   \left( \widetilde \Delta_{k_1} Q(f)- \widetilde \Delta'_{k_1} Q(f)-  \widetilde \Delta_{k_1}S_{k_2}Q(f)+ \widetilde \Delta'_{k_1} S_{k_2} Q(f) \right) \right\|_{L^{r}} \\
\nonumber &\lesssim& \left\|\sum_{k_1,k_2 \in \z} \Delta_{k_1}S_{k_2}g\  S_{k_1} \Delta_{k_2}h   \ \widetilde \Delta_{k_1} Q(f)\right\|_{L^{r}}+ \left\|\sum_{k_1,k_2 \in \z} \Delta_{k_1}S_{k_2}g\  S_{k_1} \Delta_{k_2}h  \  \widetilde \Delta'_{k_1} Q(f)  \right\|_{L^{r}} \\
 &+&\left\|\sum_{k_1,k_2 \in \z} \Delta_{k_1}S_{k_2}g\  S_{k_1} \Delta_{k_2}h  \  \widetilde \Delta_{k_1}S_{k_2}Q(f) \right\|_{L^{r}}+ \left\|\sum_{k_1,k_2 \in \z} \Delta_{k_1}S_{k_2}g\  S_{k_1} \Delta_{k_2}h  \    \widetilde \Delta'_{k_1} S_{k_2} Q(f) \right\|_{L^{r}} \nonumber \\
 & & \label{q2dcp} \\
\nonumber &\lesssim& \left\| g \right\|_{L^{p_2}} \left\| h \right\|_{L^{p_3}} \left\| Q(f)\right\|_{L^{p_1}} \lesssim \left\| g \right\|_{L^{p_2}} \left\| h \right\|_{L^{p_3}} \left\|f \right\|_{L^{p_1}}
\end{eqnarray}
where the estimates for the last two terms in \eqref{q2dcp} are given by the classical trilinear bi-parameter boundedness, while the first two terms correspond to situation $III_1$ in Remark \ref{rmk1}, where in the second parameter we just have a bilinear multiplier.

 Moreover, this approach should work for all of the  operators $T^1$, $T^2$, $T^3$ and $T^4$, since more $\psi$-type functions appear in $T^1$, $T^2$ and $T^3$. In fact, these four operators differ from each other in the decomposition for $g$ and $h$ part.  However, what they have in common is that for each parameter, there is at least one $\psi$-function or $\Delta$, and this is the key in our argument. Thus, the boundedness of \eqref{biop34} would follow, as long as the assumption in Remark \ref{rmk1} is true.

\section{Reduction of Theorem $\ref{bithp}$ }
\label{rtl}
In this section we give the idea to prove Theorem \ref{bithp}; the strategy  is to reduce the pseudo-differential operator to a localized version. From now on we will redefine the functions that were used in the previous sections: $\psi$, $\varphi$, and $\phi$.

First pick two sequences of smooth functions $(\varphi_n)_{n\in \z}$ , $(\varphi'_m)_{m\in \z}$ such that $\spt \varphi_n\subseteq [n-1,n+1]$ and $\spt \varphi'_m\subseteq [m-1,m+1]$ satisfying
 $$\sum_{n\in\z} \varphi_n(x_1)=1, \qquad \sum_{m\in\z} \varphi'_m(x_2)=1,\qquad \text{where}\;(x_1,x_2)\in \nn.$$
 Then we can decompose the operator $T_{ab}$ in $\eqref{bipdo}$ as $$T_{ab}=\sum_{n,m\in\z} T_{ab}^{n,m}$$ where
 $$T_{ab}^{n,m}(f,g,h)(x):=T_{ab}(f,g,h)(x)\varphi_n(x_1)\varphi_m'(x_2).$$

 Suppose we can prove the estimate
 \begin{equation}
 \label{lcdomp}
 \|T_{ab}^{n,m}(f,g,h)\|_r\lesssim \|f\tilde \chi_{R_{nm}}\|_{p_1} \|g\tilde \chi_{R_{nm}}\|_{p_2} \|h \tilde \chi_{R_{nm}}\|_{p_3},
 \end{equation}
where $R_{nm}=I_n\times J_m$, $I_n=[n,n+1]$, $J_m=[m,m+1]$ and $\tilde \chi_{R_{nm}}= \tilde \chi_{I_n}(x_1)\times  \tilde \chi_{J_m}(x_2) $ as defined in $\eqref{apc}$.

Then our main Theorem $\ref{bithp}$ can be proved by the following estimate
\begin{eqnarray*}
\|T_{ab}(f,g,h)\|_r &\lesssim & (\sum_{n,m\in\z}\|T_{ab}^{n,m}(f,g,h)\|_r^r)^{1/r} \\
&\lesssim& (\sum_{n,m\in\z}\|f\tilde \chi_{R_{nm}}\|_{p_1}^r \|g\tilde \chi_{R_{nm}}\|_{p_2}^r \|h \tilde \chi_{R_{nm}}\|_{p_3}^r)^{1/r} \\
&\lesssim& (\sum_{n,m\in\z} \|f\tilde \chi_{R_{nm}}\|_{p_1}^{p_1})^{1/p_1}(\sum_{n,m\in\z} \|g\tilde \chi_{R_{n,m}}\|_{p_2}^{p_2})^{1/p_2} (\sum_{n,m\in\z} \|h\tilde \chi_{R_{nm}}\|_{p_3}^{p_3})^{1/p_3} \\
&\lesssim& \|f\|_{p_1}\|g\|_{p_2}\|h\|_{p_3}.
\end{eqnarray*}
Thus, we only need to prove $\eqref{lcdomp}$.

\medskip

Consider that for a fixed $n_0,m_0\in \z$, we have
\begin{eqnarray*}
T_{ab}^{n_0,m_0}(f,g,h)(x)&=&\int_{\n^6}a(x,\xi,\eta,\zeta)\tilde \varphi_{n_0}(x_1)\tilde \varphi_{m_0}'(x_2)b(x,\eta,\zeta)\tilde \varphi_{n_0}(x_1) \tilde \varphi_{m_0}'(x_2) \\
& & \cdot \varphi_{n_0}(x_1) \varphi'_{m_0}(x_2) \hat f(\xi)\hat g(\eta)\hat h(\zeta)e^{2\pi i x(\xi+\eta+\zeta)}d\xi d\eta d\zeta,
\end{eqnarray*}
where $\tilde \varphi_{n_0},\tilde \varphi_{m_0}'$ are smooth functions supported on the intervals $[n_0-2,n_0+2]$, $[m_0-2,m_0+2]$, which equal 1 on the supports of $\varphi_{n_0},\varphi'_{m_0}$ respectively. Then we rewrite the symbols $a(x,\xi,\eta,\zeta)\tilde \varphi_{n_0}(x_1)\tilde \varphi_{m_0}'(x_2)$ and $b(x,\eta,\zeta)\tilde \varphi_{n_0}(x_1)\tilde \varphi_{m_0}'(x_2)$ by using Fourier series with respect to the $x$ variable

\begin{gather*}
  a(x,\xi,\eta,\zeta)\tilde \varphi_{n_0}(x_1)\tilde \varphi_{m_0}'(x_2)= \sum_{l_1,l_2\in\z}a_{l_1,l_2}(\xi,\eta,\zeta,\zeta)e^{2\pi i( x_1 l_1+x_2l_2)}\\
  b(x,\eta,\zeta)\tilde \varphi_{n_0}(x_1)\tilde \varphi_{m_0}'(x_2)= \sum_{l_1',l_2'\in\z}b_{l_1',l_2'}(\eta,\zeta)e^{2\pi i ( x_1 l_1'+x_2l_2')},
\end{gather*}
where
\begin{eqnarray*}
  a_{l_1,l_2}(\xi,\eta,\zeta)=\int_{\n^2} a(x,\xi,\eta,\zeta)\tilde \varphi_{n_0}(x_1)\tilde \varphi_{m_0}'(x_2) e^{-2\pi i (x_1l_1+x_2l_2)}dx,
\end{eqnarray*}
\begin{eqnarray*}
  b_{l'_1,l'_2}(\eta,\zeta)=\int_{\n^2} b(x,\eta,\zeta)\tilde \varphi_{n_0}(x_1)\tilde \varphi_{m_0}'(x_2) e^{-2\pi i(x_1l'_1+x_2l'_2)}dx.
\end{eqnarray*}
By taking advantage of conditions $\eqref{bisymbol}$ we have
\begin{gather*}
|\partial^{\alpha_1,\beta_1,\gamma_1}_{\xi_1,\eta_1,\zeta_1}\partial^{\alpha_2,\beta_2,\gamma_2}_{\xi_2,\eta_2,\zeta_2}a_{l_1,l_2}(\xi,\eta,\zeta)|\\
\lesssim {1\over (1+|(l_1,l_2)|)^M}{1\over (1+|\xi_1|+|\eta_1|+|\zeta_1|)^{\alpha_1+\beta_1+\gamma_1}} {1\over (1+|\xi_2|+|\eta_2|+|\zeta_2|)^{\alpha_2+\beta_2+\gamma_2}}\\
|\partial^{\beta_1,\gamma_1}_{\eta_1,\zeta_1} \partial^{\beta_2,\gamma_2}_{\eta_2,\zeta_2} b_{l_1',l_2'}(\eta,\zeta)|\lesssim {1\over (1+|(l_1',l_2')|)^M}{1\over (1+|\eta_1|+|\zeta_1|)^{\beta_1+\gamma_1}}{1\over (1+|\eta_2|+|\zeta_2|)^{\beta_2+\gamma_2}}
\end{gather*}
for a large number $M$ and all  indices $\alpha_1,\alpha_2,\beta_1,\beta_2,\gamma_1,\gamma_2$. Note that the decay in $l_1,l_2,l_1',l_2'$ allows one to take summation $T_{ab}^{n_0,m_0}=\sum_{l_1,l_2,l_1',l_2'}T_{ab}^{n_0,m_0,l_1,l_2,l_1',l_2'}$, where
\begin{gather*}
T_{ab}^{n_0,m_0,l_1,l_2,l_1',l_2'}(f,g,h)(x)= \\
(\int_{\n^6}a_{l_1,l_2}(\xi,\eta,\zeta)b_{l_1',l_2'}(\eta,\zeta) \hat f(\xi)\hat g(\eta)\hat h(\zeta)e^{2\pi i x(\xi+\eta+\zeta)}d\xi d\eta d\zeta)\varphi_{n_0}(x_1)\varphi'_{m_0}(x_2).
\end{gather*}
That means we only need to consider the case for $l_1,l_2,l_1',l_2'=0$. For simplicity, we denote it by
\begin{gather*}
T_{ab}^{n_0,m_0,0,0}(f,g,h)(x)= \\
(\int_{\n^6}a_{ 0}(\xi,\eta,\zeta)b_{ 0}(\eta,\zeta) \hat f(\xi)\hat g(\eta)\hat h(\zeta)e^{2\pi i x(\xi+\eta+\zeta)}d\xi d\eta d\zeta)\varphi_{n_0}(x_1) \varphi'_{m_0}(x_2),
\end{gather*}
where the symbols $a_0,b_0$ satisfy the following conditions
\begin{gather}
\nonumber
|\partial^{\alpha_1,\beta_1,\gamma_1}_{\xi_1,\eta_1,\zeta_1}\partial^{\alpha_2,\beta_2,\gamma_2}_{\xi_2,\eta_2,\zeta_2}a_0(\xi,\eta,\zeta)|\\
\nonumber \lesssim {1\over (1+|\xi_1|+|\eta_1|+|\zeta_1|)^{\alpha_1+\beta_1+\gamma_1}} {1\over (1+|\xi_2|+|\eta_2|+|\zeta_2|)^{\alpha_2+\beta_2+\gamma_2}}\\
\label{cdab0}
|\partial^{\beta_1,\gamma_1}_{\eta_1,\zeta_1} \partial^{\beta_2,\gamma_2}_{\eta_2,\zeta_2} b_{0}(\eta,\zeta)|\lesssim {1\over (1+|\eta_1|+|\zeta_1|)^{\beta_1+\gamma_1}}{1\over (1+|\eta_2|+|\zeta_2|)^{\beta_2+\gamma_2}}
\end{gather}
for all indices $\alpha_1,\alpha_2,\beta_1,\beta_2,\gamma_1,\gamma_2$.

By translation invariance, we only need to prove the following localized result for $n_0,m_0=0$.
\begin{theorem}
\label{localth}
For $1<p_1,p_2,p_3< \infty$,  and $1/p_1+1/p_2+1/p_3=1/r$ the operator
\begin{gather}
\nonumber T_{ab}^0:=T_{ab}^{0,0,0,0}(f,g,h)(x)=\\
(\int_{\n^6}a_0(\xi,\eta,\zeta)b_0(\eta,\zeta) \hat f(\xi)\hat g(\eta)\hat h(\zeta)e^{2\pi i x(\xi+\eta+\zeta)}d\xi d\eta d\zeta)\varphi_0(x_1)\varphi'_0 (x_2) \label{lcop}
\end{gather}
has the following boundedness property
\begin{equation*}
\|T_{ab}^{0}(f,g,h)\|_r\lesssim \|f\tilde \chi_{R_{00}}\|_{p_1}\|g\tilde \chi_{R_{00}}\|_{p_2}\|h\tilde \chi_{R_{00}}\|_{p_3},
\end{equation*}
 where $\varphi_0,\varphi'_0$ are  smooth functions supported within $I^0=[-1,1]$, $\tilde \chi_{R_{00}}(x)=\tilde \chi_{I^0}(x_1)\cdot \tilde \chi_{I^0}(x_2)$ and $a_0,b_0$ satisfy the conditions $\eqref{cdab0}$.

\medskip

\end{theorem}
In short, the proof of Theorem $\ref{bithp}$ can be reduced to the above theorem, and in the next section we will show how to deal with the operator in $\eqref{lcop}$.

\section{Proof of Theorem $\ref{localth}$}
\label{rtp}

In this section we prove Theorem \ref{localth}. The first step is to use Fourier series as before and rewrite the operator $\eqref{lcop}$. Here we make use of the fact that the conditions $\eqref{cdab0}$ do not involve any singularity. That means that there is no problem when the variables $\xi,\eta,\zeta$ are close to zero. More precisely, we can modify the Littlewood-Paley decomposition as follows.

  Let $\varphi$ be a Schwartz function such that $\spt \hat \varphi \subseteq [-1,1]$ and $\hat\varphi(u)=1$ on $[-1/2,1/2]$, and let $\psi $ be the Schwartz function satisfying
  $$\hat \psi(u):=\hat \varphi(u/2)-\hat \varphi(u),$$ and let $$\widehat {\psi_k}(\cdot)=\widehat \psi (\cdot/2^k)\qquad \text{and} \qquad \widehat {\psi_{-1}}(\cdot)=\hat \varphi (\cdot).$$
  Note that

  $$1=\sum_{k\geq -1}\widehat{\psi_k}, \quad \text{where}\;\spt \hat\psi \subseteq [-2^{k+1},-2^{k-1}]\cup [2^{k-1},2^{k+1}]\;\text{for}\; k\geq 0 .$$
The key thing here is that one does not have to decompose the identity near $0$. Moreover, for any smooth function $\phi$ supported on a closed interval, we write $\tilde \phi$ to denote a smooth function that is supported on a slightly larger interval and equal to $1$ on the support of $\phi$. Actually, we will use $\phi$ to represent either a $\varphi$ function or a $\psi$ function. For simplicity, let us consider the single-parameter case first, i.e. temporarily assume $\xi,\eta,\zeta \in \n$. By expanding in Fourier series as before, it can be seen to be sufficient to replace the symbols $a_0(\xi,\eta,\zeta)$ and $b_0(\xi,\eta,\zeta)$ with
\begin{equation*}
\label{a0decomp}
a_0(\xi,\eta,\zeta)=\sum_{k\geq 0} \widehat{\phi^1_k}(\xi)\widehat{\phi^2_k}(\eta)\widehat{\phi^3_k}(\zeta)+  \hat \varphi(\xi)\hat \varphi(\eta)\hat \varphi(\zeta),
\end{equation*}
where at least one of the families $(\widehat{\phi^1_k}(\xi))_k$, $(\widehat{\phi^2_k}(\eta))_k$, and $(\widehat{\phi^3_k}(\zeta))_k$ is supported away from the origin. Similarly,
\begin{equation*}
\label{b0decomp}
b_0(\eta,\zeta)=\sum_{k\geq 0} \widehat{\phi^2_k}(\eta)\widehat{\phi^3_k}(\zeta)+ \hat \varphi(\eta)\hat \varphi(\zeta),
\end{equation*}
where at least one of the families $(\widehat{\phi^2_k}(\eta))_k$, $(\widehat{\phi^3_k}(\zeta))_k$ is supported away from the origin. Now we can replace the symbol $a_0(\xi,\eta,\zeta)b_0(\eta,\zeta)$ by
\begin{eqnarray}
\nonumber
& &a_0(\xi,\eta,\zeta)b_0(\eta,\zeta)\\
 &=& \nonumber
 \big(\sum_{{k_1}\geq 0} \widehat{\phi^1_{k_1}}(\xi)\widehat{\phi^2_{k_1}}(\eta)\widehat{\phi^3_{k_1}}(\zeta)+ \hat \varphi(\xi)\hat \varphi(\eta)\hat \varphi(\zeta)\big) \big(\sum_{{k_2}\geq 0} \widehat{\phi^2_{k_2}}(\eta)\widehat{\phi^3_{k_2}}(\zeta)+ \hat \varphi(\eta)\hat \varphi(\zeta)\big) \\
\nonumber
&\approx&(\sum_{k_1\geq 0} \widehat{\phi^1_{k_1}}(\xi)\widehat{\phi^2_{k_1}}(\eta)\widehat{\phi^3_{k_1}}(\zeta)\sum_{k_2\geq 0} \widehat{\phi^1_{k_2}}(\eta)\widehat{\phi^2_{k_2}}(\zeta))+ (\sum_{k_1\geq 0} \widehat{\phi^1_{k_1}}(\xi)\widehat{\phi^2_{k_1}}(\eta)\widehat{\phi^3_{k_1}}(\zeta))\hat \varphi(\eta)\hat \varphi(\zeta) \\
\nonumber
& &\quad + (\sum_{k_2\geq 0}\widehat{\phi^1_{k_2}}(\eta)\widehat{\phi^2_{k_2}}(\zeta))\hat \varphi(\xi)\hat \varphi(\eta)\varphi(\zeta)+ \hat \varphi(\xi)\hat \varphi(\eta) \hat \varphi(\zeta) \hat \varphi(\eta)\hat \varphi(\zeta) \nonumber \\
& = &(\sum_{k_1\geq 0} \widehat{\phi^1_{k_1}}(\xi)\widehat{\phi^2_{k_1}}(\eta)\widehat{\phi^3_{k_1}}(\zeta)\sum_{k_2\ll k_1} \widehat{\phi^1_{k_2}}(\eta)\widehat{\phi^2_{k_2}}(\zeta)) \nonumber \\
& & + (\sum_{k_1\geq 0} \widehat{\phi^1_{k_1}}(\xi)\widehat{\phi^2_{k_1}}(\eta)\widehat{\phi^3_{k_1}}(\zeta) \sum_{k_2\gg k_1} \widehat{\phi^1_{k_2}}(\eta)\widehat{\phi^2_{k_2}}(\zeta))\nonumber \\
& &+ (\sum_{k_1\geq 0} \widehat{\phi^1_{k_1}}(\xi)\widehat{\phi^2_{k_1}}(\eta)\widehat{\phi^3_{k_1}}(\zeta)\sum_{k_2\simeq k_1} \widehat{\phi^1_{k_2}}(\eta)\widehat{\phi^2_{k_2}}(\zeta))\nonumber \\
& &+ (\sum_{k_1\geq 0} \widehat{\phi^1_{k_1}}(\xi)\widehat{\phi^2_{k_1}}(\eta)\widehat{\phi^3_{k_1}}(\zeta))\hat \varphi(\eta)\hat \varphi(\zeta) \nonumber\\
& & + (\sum_{k_2\geq 0}\widehat{\phi^1_{k_2}}(\eta)\widehat{\phi^2_{k_2}}(\zeta))\hat \varphi(\xi)\hat \varphi(\eta)\hat \varphi(\zeta)+ \hat \varphi(\xi)\hat \varphi(\eta)\hat \varphi(\eta)\varphi(\zeta)\hat \varphi(\zeta) \nonumber\\
\label{biparafinal}
&:=& (E+F+G+H+K+L)(\xi,\eta,\zeta),
\end{eqnarray}

First note that it is not possible that $k_2\gg k_1$, which implies that $F=0$, since at least one of $(\widehat{\phi^2_k}(\eta))_k$, $(\widehat{\phi^3_k}(\zeta))_k$ is supported
away from the origin.

 To take care of other terms, the essential idea here is to compare the sizes of the supports of $\xi,\eta,\zeta$, as we have done before.  Roughly speaking, one can consider the following two cases:
\begin{itemize}
\item Case I: When $\{|\xi|\leq c (|\eta|+|\zeta|)\}$ for some constant $c$, i.e. the terms $G,K,L$, such terms correspond to the following estimate of the symbol
\begin{equation*}
\label{case1}
  |\partial_\xi^\alpha \partial_\eta^\beta \partial_\zeta^\gamma a_0(\xi,\eta,\zeta) b_0(\eta,\zeta)| \lesssim \frac{1}{(|\xi|+|\eta|+|\gamma|)^{\alpha+\beta+\gamma}}.
\end{equation*}
 \item Case II: When $\{|\xi|\geq {1\over 2}c (|\eta|+|\zeta|)\}$, i.e. the terms  $E,G$ correspond to the operators in Definition $\ref{defop}$ - see \cite{muscaluflag,muscalu2007paraproducts} for more details.
\end{itemize}
  \vskip0.5cm
With the above argument, we can simplify $\eqref{biparafinal}$ to
\begin{equation*}
  E+F+G+H+K+L\approx E+G,
\end{equation*}
where we use $E$ to represent Case II, and $G$ to represent Case I.
\medskip

Now we come back to the bi-parameter case. By doing the decomposition as above in each parameter , i.e. $(\xi_1,\eta_1,\zeta_1)$ and $(\xi_2,\eta_2,\zeta_2)$, one should have four cases to estimate. More precisely, we can replace $a_0(\xi,\eta,\zeta)b_0(\eta,\zeta)$ by
\begin{equation*}
  a_0(\xi,\eta,\zeta)b_0(\eta,\zeta)\approx (E+G)(\xi_1,\eta_1,\zeta_1)(E'+G')(\xi_2,\eta_2,\zeta_2)
\end{equation*}
Correspondingly, the localized operator is changed to
\begin{eqnarray*}
\nonumber
 & &T_{ab}^{0}(f,g,h)(x)\\
 \nonumber &=&\big(\int_{\n^6}a_0(\xi,\eta,\zeta)b_0(\eta,\zeta)\hat f(\xi)\hat g(\eta)\hat h(\zeta)e^{2\pi i x(\xi+\eta+\zeta)}d\xi d\eta d\zeta\big)\varphi_0(x_1) \varphi_0'(x_2) \\ \nonumber
  &=&\big(\int_{\n^6}(E+G)(E'+G')\hat f(\xi)\hat g(\eta)\hat h(\zeta)e^{2\pi i x(\xi+\eta+\zeta)}d\xi d\eta d\zeta\big)\varphi_0(x)\varphi_0'(x_2)\nonumber\\
  \label{efgh}
  &:=&T^{E,E',0}_{ab}+T^{E,G',0}_{ab}+T^{G,E',0}_{ab} +T^{G,G',0}_{ab}.
\end{eqnarray*}
\vskip0.5cm
\subsection{ Estimates for $T^{G,G',0}_{ab}$}
\label{secLL}
~\\
First consider $T^{G,G',0}_{ab}$, recall
\begin{gather*}
T^{G,G',0}_{ab}(f,g,h)(x)\\
=\big(\int_{\n^6} m_{G,G'}^0(\xi,\eta,\zeta) \hat f(\xi)\hat g(\eta)\hat h(\zeta)e^{2\pi i x(\xi+\eta+\zeta)}d\xi d\eta d\zeta\big)\varphi_0(x_1)\varphi_0'(x_2),
\end{gather*}
where  $m_{G,G'}^0(\xi,\eta,\zeta):= G G'$ satisfies
\begin{gather*}
| \partial^{\alpha_1,\alpha_2}_{\xi_1,\xi_2} \partial^{\beta_1,\beta_2}_{\eta_1,\eta_2} \partial^{\gamma_1,\gamma_2}_{\zeta_1,\zeta_2} m_{G,G'}^0(\xi,\eta,\zeta)|\\
\lesssim {1\over (1+|\xi_1|+|\eta_1|+|\zeta_1|)^{\alpha_1+\beta_1+\gamma_1}}{1\over (1+|\xi_2|+|\eta_2|+|\zeta_2|)^{\alpha_2+\beta_2+\gamma_2}}
\end{gather*}
for sufficiently many indices $\alpha_1,\alpha_2,\beta_1,\beta_2,\gamma_1, \gamma_2 .$  Then our desired localized estimate
$$\|T^{G,G',0}_{ab}(f,g,h)\|_r\lesssim \|f\tilde \chi_{R_{00}}\|_{p_1}\|g\tilde \chi_{R_{00}}\|_{p_2} \|h\tilde \chi_{R_{00}}\|_{p_3}$$
follows from  the proof of Theorem \ref{trilocalth}, see \cite{muscalu2013classical,dai2013p}. \\

\subsection{ Estimates for $T^{E,E',0}_{ab}$}\label{secEE}
~\\

Recall \begin{eqnarray*}
  E\cdot E'&=&\bigg(\sum_{k_1} \widehat{\phi^1_{k_1}}(\xi_1)\widehat{\phi^2_{k_1}}(\eta_1)\widehat{\phi^3_{k_1}}(\zeta_1)\sum_{k_2\ll k_1} \widehat{\phi^1_{k_2}}(\eta_1)\widehat{\phi^2_{k_2}}(\zeta_1)\bigg) \\
  & &\quad \cdot \bigg(\sum_{k'_1} \widehat{\phi^1_{k'_1}}(\xi_2)\widehat{\phi^2_{k'_1}}(\eta_2)\widehat{\phi^3_{k'_1}}(\zeta_2)\sum_{k'_2\ll k'_1} \widehat{\phi^1_{k'_2}}(\eta_2)\widehat{\phi^2_{k'_2}}(\zeta_2)\bigg),
\end{eqnarray*}
where for each $l=k_1,k_2,k'_1,k'_2$, at least one of the families $(\widehat{\phi^1_{l}})_{l}$ and $(\widehat{\phi^2_{l}})_{l}$ is $\Psi$ $type$. And
\begin{eqnarray*}
T_{ab}^{E,E',0}(f,g,h)(x)=\bigg(\int_{\n^6} (E\cdot E') e^{2\pi i x(\xi+\eta+\zeta)}\hat f(\xi) \hat g(\eta) \hat h(\zeta) d\xi d\eta d\zeta\bigg) \varphi_0(x_1)\varphi_0'(x_2),
\end{eqnarray*}
where we have removed multipliers $a_0,b_0$ by using Fourier series as before.

\medskip

We now give two lemmas for single parameter operators when $x,\xi,\eta,\zeta \in \n$, which will be used later.
 \begin{gather}
 \nonumber
 T^E(f,g,h)(x)\cdot \varphi_0(x) =:T^{E,0}_{ab}(f,g,h)(x):=\\
\label{dfti}
 (\int_{\nnn} (\sum_{k_1}\widehat{\phi^1_{k_1}}(\xi)\widehat{\phi^2_{k_1}}(\eta)\widehat{\phi^3_{k_1}}(\zeta))( \sum_{k_2\ll k_1} \widehat{\phi^1_{k_2}}(\eta)\widehat{\phi^2_{k_2}}(\zeta)) \hat f(\xi)\hat g(\eta)\hat h(\zeta)e^{2\pi i x(\xi+\eta+\zeta)}d\xi d\eta d\zeta)\varphi_0(x),
\end{gather}
where $x,\xi,\eta,\zeta\in \n$. From \cite {muscalu2013classical,muscalu2007paraproducts}, we can show  $T^E$ can be decomposed into paraproducts.
Before we state this result, we introduce some notations which are needed in the statement of the result.

Then we introduce some notations that will appear in the next lemma.
\begin{itemize}
\item[(a)] We take $T_1(f,g,h)$ and $ B^1_I(g,h)$ from $\eqref{t1}$ and $\eqref{b1}$  in Definition $\ref{defop}$.
\item[(b)] For positive integers $l$ and $k_0\geq 100$, let
 \begin{gather}
\nonumber
T_{l,k_0}^1(f,g,h)=\sum_{I\in \mathcal{I}}{1\over |I|^{1\over 2}}\langle f,\phi_I^{l,1}\rangle \langle B^{1,l}_{I,k_0}(g,h),\phi_I^{l,2}\rangle \phi_I^{l,3} \\
\text{with} \qquad B^{1,l}_{I,k_0}(g,h)=\sum_{\substack{J\in \mathcal{J}\\ 2^{k_0}|\omega_J^3|\simeq |\omega^2_I|}}{1\over |J|^{1\over 2}} \langle g,\phi_J^{l,1}\rangle \langle h,\phi_J^{l,2}\rangle \phi_J^{l,3} \label{tl}.
\end{gather}

The functions $(\phi_I^{l,k})_{k=1,2,3}$ in $T_{l,k_0}^1(f,g,h)$ and the functions $(\phi_I^{k})_{k=1,2,3}$ in $\eqref{tk}$  are of the same ``types" (whether each one is lacunary  or  non-lacunary), and so are the functions $(\phi_J^{l,k})_{k=1,2,3}$ in \eqref{tl} and the functions $(\phi_J^{k})_{k=1,2,3}$ in$\eqref{bk}$. Note the only difference between them is the dependence on $l$. In fact, $\phi_I^{l,2}(x)$ could  be  $(\widehat \phi_I^2(\xi)\xi^l)^\vee$, but fortunately, this does not change the ``types" of those functions.  In this sense, $l$ will not play an important role in our estimates. For simplification, we omit this dependence on $l$ for all the expressions in the rest of the work.

\item[(c)] For a large positive integer $M$, let $T_{M,k_0}^1$ be a Fourier multiplier operator with  $m_{M,k_0}^1(\xi,\eta,\zeta)$ satisfying the condition
      \begin{equation}
      \label{cdmk0}
      |\partial^\alpha_\xi \partial^\beta_\eta \partial^\gamma_\zeta m_{M,k_0}^1(\xi,\eta,\zeta)|\lesssim (2^{k_0})^{\alpha+\beta+\gamma}{1\over (1+|\xi|+|\eta|+|\zeta|)^{\alpha+\beta+\gamma}}
      \end{equation}
for sufficiently many indices $\alpha,\beta,\gamma$.
\item[(d)] For $T_1$ and $T_{l,k_0}^1$ in $(a),(b)$, all the dyadic intervals   have lengths at most $1$ for all $k_0\geq 100, 1\leq l\leq M-1$.
\end{itemize}

With these notations, we are ready to state the following
\begin{lemma}
\label{lm}
Define $T^E$ as in $\eqref{dfti}$, then we can write
\begin{gather*}
T^E(f,g,h)(x)= \\
T_1(f,g,h)(x)+\sum_{l=1}^{M-1}\sum_{k_0=100}^{\infty} (2^{-k_0})^l T_{l,k_0}^1(f,g,h)(x)+\sum_{k_0=100}^{\infty}(2^{-k_0})^M T_{M,k_0}^1(f,g,h)(x).
\end{gather*}

\end{lemma}

\begin{proof}
 One can follow the work \cite{muscalu2007paraproducts} closely, where the Taylor expansions of proper functions are used to get such forms of paraproducts. The only two statements we need to show are that all the dyadic intervals there have lengths at most one and the decay number $1$ in the denominator from $\eqref{cdmk0}$. In fact both of them follow from the fact $k_1,k_2\geq 0$ in \eqref{dfti}.
\end{proof}
Some more remarks for $T^1_{l,k_0}$ and $T^1_{M,k_0}$ are given below.

\medskip
\begin{remark}~\\
\label{rmk}
 \begin{itemize}
 \item[(a)]$T^1_{l,k_0}$:   For each $k_0$, one can see $T^1_{l,k_0}$  and $T_1$ are defined in very similar forms, and that means $T^1_{l,k_0}$ can be treated in the same way as $T_1$, since what really matters in the proof is the forms of paraproducts. More precisely, the bound of $T^1_{l,k_0}$ is actually independent of $k_0$, and then the factor $2^{-k_0l}$ allows us to take the summation over $k_0$. Thus,  we will only deal with $T_1$ here, and one can easily get a similar argument for the $T^1_{l,k_0}$ part.
  \item[(b)]$T^1_{M,k_0}$: the condition $\eqref{cdmk0}$ actually guarantees the estimate $$\|T^1_{M,k_0}(f,g,h)\|_r \lesssim 2^{100k_0}\|f\|_{p_1}  \|g\|_{p_2}  \|h\|_{p_3}, $$
see \cite{muscalu2013classical}. By picking $M$ to be sufficiently large,  we are able to take the summation over $k_0$ for  $(2^{-k_0})^M T_{M,k_0}^1(f,g,h)(x)$.
\end{itemize}
\end{remark}

Now we return to the bi-parameter case. First consider $T^{E,E',0}_{ab}$. Combining the ideas and the proof in those two lemmas (see \cite{ dai2013p,muscalu2013classical,muscalu2007paraproducts} for details), one can check
\begin{gather*}
  T_{ab}^{E,E',0}(f,g,h)(x)\\
  \approx T_{1,1}^{E,E',0}(f,g,h)(x)+ T_{1,m^{\prime,1}_{M,k_0}}^{E,E',0}(f,g,h)(x)\\ +T_{m^1_{M,k_0},1}^{E,E',0}(f,g,h)(x)+T_{m^1_{M,k_0},m^{\prime,1}_{M,k_0}}^{E,E',0}(f,g,h)(x),
\end{gather*}
where the definitions of the four operators will be given below.

First let's see the simplest case, which is
\begin{gather*}
   T_{m^1_{M,k_0},m^{\prime,1}_{M,k_0}}^{E,E',0}(f,g,h)(x)=\\
   (\int_{\n^6} m_{M,k_0}^1(\xi_1,\eta_1,\zeta_1) m_{M,k_0}^{\prime,1}(\xi_2,\eta_2,\zeta_2) e^{2\pi i x(\xi+\eta+\zeta)}\hat f(\xi) \hat g(\eta) \hat h(\zeta) d\xi d\eta d\zeta) \varphi_0(x_1)\varphi_0'(x_2),
\end{gather*}
where the symbols $ m_{M,k_0}^1, m_{M,k_0}^{\prime,1}$ are defined as in Lemma $\ref{lm}$. In this case we can ignore $k_0$ as discussed in Remark $\ref{rmk}$ and clearly the desired estimate follows from Theorem $\ref{trilocalth}$.

Then we consider $T_{1,1}^{E,E',0}(f,g,h)$.
\begin{gather*}
T_{1,1}^{E,E',0}(f,g,h):= \\
(\sum_{\substack{I\in \mathcal{I}\\ I'\in \mathcal{I'}}}{1\over |I|^{1\over 2}} {1\over |I'|^{1\over 2}}\langle f,\phi_I^1\otimes \phi_{I'}^1 \rangle \langle B^1_{I,I'}(g,h),\phi_I^2\otimes \phi_{I'}^2 \rangle \phi_I^3 \otimes \phi_{I'}^3) \varphi_0(x_1)\varphi_0'(x_2)
\end{gather*}
with
\begin{gather*}
 B^1_{I,I'}(g,h)=\sum_{\substack{J\in \mathcal{J}, J\in \mathcal{J'}\\ |\omega_J^3|\leq |\omega^2_I|\\|\omega_{J'}^3|\leq |\omega^2_{I'}|}}{1\over |J|^{1\over 2}} {1\over |J'|^{1\over 2}} \langle g,\phi_J^1 \otimes \phi_{J'}^1\rangle \langle h,\phi_J^2\otimes \phi_{J'}^2\rangle \phi_J^3\otimes \phi_{J'}^3,
\end{gather*}
where the families $(\phi^j_I)_{I\in \mathcal{I}}, (\phi^j_{I'})_{I'\in \mathcal{I'}}\,(j=1,2,3)$ are defined as $(\phi^j_I)_{I\in \mathcal{I}}$ , and  the families $(\phi^j_J)_{I\in \mathcal{J}}, (\phi^j_{J'})_{J'\in \mathcal{J'}}$ are defined as $(\phi^j_J)_{J\in \mathcal{J}}$ in Definition $\ref{defop}$.
 Taking advantage of that $|I|,|I'|\leq 1$, we can split
 \begin{eqnarray}
 \nonumber
 & &T_{1,1}^{E,E',0}(f,g,h)=\{(\sum_{\substack{I\subseteq 5I^0 \\ I'\subseteq 5I^0 }}+ \sum_{\substack{I\subseteq (5I^0)^c \\ I'\subseteq 5I^0 }}+ \sum_{\substack{I\subseteq 5I^0 \\ I'\subseteq (5I^0)^c }} + \sum_{\substack{I\subseteq (5I^0)^c \\ I'\subseteq (5I^0)^c }}) \\
 & & \qquad \qquad {1\over |I|^{1\over 2}} {1\over |I'|^{1\over 2}}\langle f,\phi_I^1\otimes \phi_{I'}^1 \rangle \langle B^1_{I,I'}(g,h),\phi_I^2\otimes \phi_{I'}^2 \rangle \phi_I^3 \otimes \phi_{I'}^3 \}\varphi_0(x_1)\varphi_0'(x_2)\nonumber \\
  \label{mainctr}
 &=& T_{1,1}^{E_1,E_1',0}(f,g,h)+T_{1,1}^{E_2,E_1',0}(f,g,h)+T_{1,1}^{E_1,E_2',0}(f,g,h)+T_{1,1}^{E_2,E_2',0}(f,g,h).
 \end{eqnarray}

We start with $T_{1,1}^{E_1,E_1',0}(f,g,h)$.  Consider the following decompositions
\begin{gather}
\nonumber f(x)=\sum_{n_1,n_1'\in \z}f \chi_{I_{n_1}}(x_1)\chi_{I_{n_1'}}(x_2), \ \ \ g(x)=\sum_{n_2,n_2'\in \z} g \chi_{I_{n_2}}(x_1) \chi_{I_{n_2'}}(x_2), \\
\label{fghdecomp}  h(x)=\sum_{n_3,n_3'\in \z}h \chi_{I_{n_3}}(x_1) \chi_{I_{n_3'}}(x_2).
\end{gather}

 Then we can write
\begin{gather*}
T_{1,1}^{E_1,E_1',0}(f,g,h)(x)=\\
\sum_{n_1,n_1'}\sum_{n_2,n_2'}\sum_{n_3,n_3'}T_{1,1}^{E_1,E_1',0}(f\cdot \chi_{I_{n_1}}\otimes\chi_{I_{n_1'}}, g \cdot \chi_{I_{n_2}}\otimes \chi_{I_{n_2'}}, h \cdot \chi_{I_{n_3}}\otimes \chi_{I_{n_3'}})(x).
\end{gather*}

\medskip

When $|n_1|,|n_1'|> 10$, we write
\begin{eqnarray*}
 & &\|T_{1,1}^{E_1,E_1',0}(f\cdot \chi_{I_{n_1}}\otimes \chi_{I_{n'_1}}, g \cdot \chi_{I_{n_2}}\otimes \chi_{I_{n'_2}}, h \cdot\chi_{I_{n_3}}\otimes \chi_{I_{n'_3}})(x)\|_r \\
 &=&\|\sum_{I\in \mathcal{I},I'\in \mathcal{I'}}\sum_{\substack{J\in \mathcal{J},J'\in \mathcal{J'}\\|\omega^3_J|\leq |\omega^2_I|\\
 |\omega^3_{J'}|\leq |\omega^2_{I'}|}} \frac{1}{|I|^{1\over 2}}\frac{1}{|J|^{1\over 2}}\frac{1}{|I'|^{1\over 2}}\frac{1}{|J'|^{1\over 2}} \\
 & & \cdot \langle f \chi_{I_{n_1}}\otimes \chi_{I_{n'_1}},\phi_I^1\otimes \phi_{I'}^1 \rangle \langle g \chi_{I_{n_2}}\otimes \chi_{I_{n'_2}},\phi_J^1\otimes \phi_{J'}^1\rangle  \langle h  \chi_{I_{n_3}}\otimes\chi_{I_{n'_3}},\phi_J^2\otimes \phi_{J'}^2  \rangle \\
   & & \cdot \langle \phi^2_I\otimes \phi^2_{I'},\phi^3_J\otimes \phi^3_{J'}\rangle\phi^3_I(x_1)\phi^3_{I'}(x_2)\varphi_0(x_1)\varphi'_0(x_2)\|_r.
\end{eqnarray*}
Then we use H\"older's inequality to get
\begin{eqnarray}
\nonumber
 & & \| \frac{1}{|I|^{1\over 2}}\frac{1}{|J|^{1\over 2}}\frac{1}{|I'|^{1\over 2}}\frac{1}{|J'|^{1\over 2}}  \langle f \chi_{I_{n_1}}\otimes \chi_{I_{n'_1}},\phi_I^1\otimes \phi_{I'}^1 \rangle \langle g \chi_{I_{n_2}}\otimes \chi_{I_{n'_2}},\phi_J^1\otimes \phi_{J'}^1\rangle \\ \nonumber
  & &\cdot\langle h  \chi_{I_{n_3}}\otimes\chi_{I_{n'_3}},\phi_J^2\otimes \phi_{J'}^2  \rangle \langle \phi^2_I\otimes \phi^2_{I'},\phi^3_J\otimes \phi^3_{J'}\rangle\phi^3_I(x_1)\phi^3_{I'}(x_2)\varphi_0(x_1)\varphi'_0(x_2)\|_r  \nonumber\\
 &\lesssim& \frac{1}{|I|^2}\frac{1}{|J|^2}\frac{1}{|I'|^2}\frac{1}{|J'|^2} \nonumber \\
 & & \cdot  (1+\frac{\dst(I_{n_1},I)}{|I|})^{-M_1} (1+\frac{\dst(I_{n'_1},I')}{|I'|})^{-M'_1}  (\|f\chi_{I_{n_1}}\otimes\chi_{I_{n'_1}}\|_{p_1} (|I||I'|)^{p_1-1\over p_1}) \nonumber  \\ \nonumber
 & & \cdot(1+\frac{\dst(I_{n_2},J)}{|J|})^{-N_1} (1+\frac{\dst(I_{n'_2},J')}{|J'|})^{-N'_1}
  (\|g\chi_{I_{n_2}}\otimes \chi_{I'_{n_2}}\|_{p_2}(|J||J'|)^{p_2-1\over p_2}) \\ \nonumber
 & &\cdot (1+\frac{\dst(I_{n_3},J)}{|J|})^{-N_2} (1+\frac{\dst(I_{n'_3},J')}{|J'|})^{-N'_2}(\|h\chi_{I_{n_3}}\otimes \chi_{I_{n'_3}}\|_{p_3}(|J||J'|)^{p_3-1\over p_3})\\\nonumber
 & & \cdot (|I||I'|)^{1\over r}  \int_{\nn} (1+\frac{\dst (x_1,I)}{|I|})^{-M_2} (1+\frac{\dst(x_1,J)}{|J|})^{-N_3}  \nonumber \\
 & &\qquad \cdot(1+\frac{\dst (x_2,I')}{|I'|})^{-M'_2} (1+\frac{\dst(x_2,J')}{|J'|})^{-N'_3} dx  \nonumber\\ \nonumber
 &\lesssim & {1\over |I||I'|}({|I||I'|\over |J||J'|})^{{1\over p_2}+{1\over p_3}} (1+\frac{\dst(I_{n_1},I)}{|I|})^{-M_1} (1+\frac{\dst(I_{n'_1},I')}{|I'|})^{-M'_1} \\ \nonumber
 & & \cdot (1+\frac{\dst(I_{n_2},J)}{|J|})^{-N_1} (1+\frac{\dst(I_{n'_2},J')}{|J'|})^{-N'_1}  \\ \nonumber
   & & \cdot (1+\frac{\dst(I_{n_3},J)}{|J|})^{-N_2} (1+\frac{\dst(I_{n'_3},J')}{|J'|})^{-N'_2} \\ \nonumber
  & &\cdot \int_{\nn} (1+\frac{\dst (x_1,I)}{|I|})^{-M_2} (1+\frac{\dst(x_1,J)}{|J|})^{-N_3}  \\
  & & \cdot (1+\frac{\dst (x_2,I')}{|I'|})^{-M'_2} (1+\frac{\dst(x_2,J')}{|J'|})^{-N'_3} dx \nonumber  \\
 & & \cdot    \|f  \chi_{I_{n_1}}\otimes \chi_{I_{n'_1}}\|_{p_1} \|g \chi_{I_{n_2}}\otimes \chi_{I_{n'_2}} \|_{p_2} \|h \chi_{I_{n_3}}\otimes\chi_{I_{n'_3}} \|_{p_3},  \label{dcest}
\end{eqnarray}
where $M_j,M'_j,N_j,N'_j$ are sufficiently large integers and $\phi_I^j,\phi_J^j,\phi_{I'}^j,\phi_{J'}^j$ are $L^2$-normalized bump functions adapted to $I,I',J,J'$ for $j=1,2,3$.

 Then  we use the fact that $|\omega^3_J|\leq |\omega^2_I|, |\omega^3_{J'}|\leq |\omega^2_{I'}|$, which implies $|I|\lesssim |J|, |I'|\lesssim |J'|$ and take advantage of the locations of dyadic intervals $J$ as well.  Using the notation $J_m=[m,m+1],m\in\z$ and $\eqref{dcest}$ we can get for $0<r<1$ ($r>1$ will be similar, and from now on we always assume $0<r<1$)
\begin{eqnarray*}
  & &\|T_{1,1}^{E_1,E_1',0}(f\cdot \chi_{I_{n_1}}\otimes \chi_{I_{n'_1}}, g \cdot \chi_{I_{n_2}}\otimes \chi_{I_{n'_2}}, h \cdot\chi_{I_{n_3}}\otimes \chi_{I_{n'_3}})(x)\|_r^r \\
  &\lesssim& \sum_{\substack{i,i'\geq 0\\j,j'\geq 0}}\sum_{\substack{I,I'\subseteq 5I^0\\ |I|=2^{-i}\\ |I'|=2^{-i'}}}   \sum_{m,m'\in \z} \sum_{\substack{J\subseteq J_m,|J|=2^{-j}\\J'\subseteq J_{m'},|J'|=2^{-j'}}}  \big({1\over |I||I'|}(1+\frac{\dst(I_{n_1},I)}{|I|})^{-M_1} (1+\frac{\dst(I_{n'_1},I)}{|I|})^{-M'_1} \\
  & &(1+\frac{\dst(I_{n_2},J)}{|J|})^{-N_1}(1+\frac{\dst(I_{n'_2},J')}{|J'|})^{-N'_1}\\
  & & \cdot (1+\frac{\dst(I_{n_3},J)}{|J|})^{-N_2}(1+\frac{\dst(I_{n'_3},J')}{|J'|})^{-N'_2} \\
 & & \cdot  \int_{\nn} (1+\frac{\dst(x_1,I)}{|I|})^{-M_2}(1+\frac{\dst(x_2,I')}{|I'|})^{-M'_2} \\
 & & \cdot  (1+\frac{\dst(x_1,J)}{|J|})^{-N_3}(1+\frac{\dst(x_2,J')}{|J'|})^{-N'_3} dx \\
   & & \cdot \|f  \chi_{I_{n_1}}\otimes \chi_{I_{n'_1}}\|_{p_1} \|g \chi_{I_{n_2}}\otimes \chi_{I_{n'_2}} \|_{p_2} \|h \chi_{I_{n_3}}\otimes\chi_{I_{n'_3}} \|_{p_3}\big)^r \\
 &\lesssim &  \sum_{\substack{i,i'\geq 0\\j,j'\geq 0}}\sum_{\substack{I,I'\subseteq 5I_0\\ |I|=2^{-i}\\ |I'|=2^{-i'}}}   \sum_{m,m'\in \z} \sum_{\substack{J\subseteq J_m,|J|=2^{-j}\\J'\subseteq J_{m'},|J'|=2^{-j'}}}    (2^{i+i'} (1+2^i(|n_1|-6))^{-M_1}(1+2^{i'}(|n'_1|-6))^{-M'_1}\\
 & &\cdot  (1+2^j|m-n_2|)^{-N_1}  (1+2^{j'}|m'-n'_2|)^{-N'_1} (1+2^j|m-n_3|)^{-N_2}\\
 & &\cdot   (1+2^{j'}|m'-n'_3|)^{-N'_2}  (1+|m|)^{-N_0} (1+|m'|)^{-N'_0}  \\
 & & \cdot \|f  \chi_{I_{n_1}}\otimes \chi_{I_{n'_1}}\|_{p_1} \|g \chi_{I_{n_2}}\otimes \chi_{I_{n'_2}} \|_{p_2} \|h \chi_{I_{n_3}}\otimes\chi_{I_{n'_3}} \|_{p_3})^r \\
&\lesssim &    \sum_{m,m'\in \z}  ( (|n_1|-6)^{-{M_1\over 2}}(|n'_1|-6)^{-{M'_1\over 2}} (1+|n_2|)^{-L} (1+|n'_2|)^{-L} \\
& & \cdot (1+|n_3|)^{-L} (1+|n'_3|)^{-L}  (1+|m|)^{-{N_0\over 2}}(1+|m'|)^{-{N'_0\over 2}}\\
 & & \cdot \|f  \chi_{I_{n_1}}\otimes \chi_{I_{n'_1}}\|_{p_1} \|g \chi_{I_{n_2}}\otimes \chi_{I_{n'_2}} \|_{p_2} \|h \chi_{I_{n_3}}\otimes\chi_{I_{n'_3}} \|_{p_3})^r
\end{eqnarray*}
where the positive integers $N_0=\min\{M_2,N_3\}$, $N'_0=\min\{M'_2,N'_3\},\, L$ are sufficiently large and the last inequality holds since for any $l,m\in \z$ and any large integer $M$, there exists a large integer $M'$ such that
\begin{eqnarray}
\label{bridge}
 (1+|l-m|)^{-M}(1+|m|)^{-{N_0\over 4}}\lesssim (1+|l|)^{-M'}.
\end{eqnarray}
And also note that  in the above calculation  we can take summation over $i,j,i',j'$ because when $|n_1-6|,|n'_1-6|>0$ the power $M_1,M'_1$ can give a decay for $i,i'$, and consequently a decay for $j,j'$ as well since  $i\gtrsim j\geq0,\, i'\gtrsim j'\geq 0$.
\medskip

Now we take the summation
\begin{eqnarray*}
 & &\|\sum_{\substack{|n_1|>10\\|n'_1|> 10}}\sum_{n_2,n_3,n'_2,n'_3 \in \z }T_{1,1}^{E_1,E_1',0}(f\cdot \chi_{I_{n_1}}\otimes \chi_{I_{n'_1}}, g \cdot \chi_{I_{n_2}}\otimes \chi_{I_{n'_2}}, h \cdot\chi_{I_{n_3}}\otimes \chi_{I_{n'_3}})(x)\|_r^r \\
 &\lesssim& \sum_{|n_1|>10,|n'_1|> 10}\sum_{n_2,n_3,n'_2,n'_3 \in \z } ( (|n_1|-6)^{-{M_1\over 2}}(|n'_1|-6)^{-{M'_1\over 2}} (1+|n_2|)^{-L} (1+|n'_2|)^{-L} \\
 & &\cdot (1+|n_3|)^{-L} (1+|n'_3|)^{-L}  \|f  \chi_{I_{n_1}}\otimes \chi_{I_{n'_1}}\|_{p_1} \|g \chi_{I_{n_2}}\otimes \chi_{I_{n'_2}} \|_{p_2} \|h \chi_{I_{n_3}}\otimes\chi_{I_{n'_3}} \|_{p_3})^r\\
 &\lesssim& \sum_{|n_1|>10,|n'_1|> 10}\sum_{n_2,n_3,n'_2,n'_3 \in \z } ((|n_1|-6)^{-{M_1\over 4}} (|n'_1|-6)^{-{M'_1\over 4}} (1+|n_2|)^{-{L\over 2}} (1+|n'_2|)^{-{L\over 2}} \\
 & & \cdot(1+|n_3|)^{-{L\over 2}} (1+|n'_3|)^{-{L\over 2}}   \|f\tdd\|_{p_1} \|g\tdd\|_{p_2} \|h\tdd\|_{p_3})^r \\
 &\lesssim &  (\|f\tdd\|_{p_1} \|g\tdd\|_{p_2} \|h\tdd\|_{p_3})^r,
\end{eqnarray*}
where we use the fact for any $n\in \z$ and large integer $L$, there holds
$$(1+|n|)^{-{L\over 2}} \cdot \chi_{I_{n}}\lesssim \td .$$

\medskip

 When $|n_1|\leq 10$ or $|n'_1|\leq 10$  things are different.  Say $|n_1|\leq 10$, in this situation, the terms like $(1+\frac{\dst(I_{n_1},I)}{|I|})^{-M_1} $ in $\eqref{dcest}$ won't give us a decay factor on $i$, which means we will have trouble  taking the summation over dyadic intervals $I$. Actually the decay factors from other terms are with respect to $j$ which can't help since $i\gtrsim j$. And the same problem exists for $i',j'$ as well. This is actually where such paraproducts behave differently from the classical ones. In the classical case only one class of dyadic intervals is involved, but here we have $I$ and $J$, $I'$ and $J'$, where the decay factors coming from either class might not be used for the other one.   We will make use of Theorem $\ref{bithm}$ here. Without loss of generality, we assume both $|n_1|\leq 10$ and $|n'_1|\leq 10$. Now the goal is

\begin{eqnarray}
 \nonumber
 & &\|\sum_{\substack{|n_1|,|n'_1|\leq 10\\n_2,n'_2,n_3,n'_3\in \z}}T_{1,1}^{E_1,E_1',0}(f\cdot \chi_{I_{n_1}}\otimes \chi_{I_{n'_1}}, g \cdot \chi_{I_{n_2}}\otimes \chi_{I_{n'_2}}, h \cdot\chi_{I_{n_3}}\otimes \chi_{I_{n'_3}})(x) \|_r  \\ \label{spc}
 & \lesssim  &\|f\tdd\|_{p_1} \|g\tdd\|_{p_2} \|h\tdd\|_{p_3}.
\end{eqnarray}

 Recall that when $I,I'\subseteq 5I^0$ and $J\in J_m, J'\in J_{m'}$,   in $\eqref{dcest}$ we can write
 $$\langle \phi^2_I\otimes \phi^2_{I'},\phi^3_J\otimes \phi^3_{J'}\rangle\approx (1+|m|)^{-L} (1+|m'|)^{-L}\langle \phi^2_I\otimes \phi^2_{I'},\tilde \phi^3_J\otimes \tilde \phi^3_{J'}\rangle,$$

 $$\langle g  \chi_{I_{n_2}}\otimes\chi_{I_{n'_2}},\phi_J^1\otimes \phi_{J'}^1  \rangle \approx (1+|n_2-m|)^{-M} (1+|n'_2-m'|)^{-M}\langle g \chi_{I_{n_2}}\otimes\chi_{I_{n'_2}},\tilde \phi_J^1\otimes \tilde \phi_{J'}^1 \rangle, $$

  $$\langle h  \chi_{I_{n_3}}\otimes\chi_{I_{n'_3}},\phi_J^2\otimes \phi_{J'}^2  \rangle \approx (1+|n_3-m|)^{-M} (1+|n'_3-m'|)^{-M}\langle h  \chi_{I_{n_3}}\otimes\chi_{I_{n'_3}},\tilde \phi_J^2\otimes \tilde \phi_{J'}^2 \rangle, $$
where $\tilde \phi_J^l, \tilde \phi_{J'}^l$ are properly chosen bump functions adapted to $J,J'$ that  have the same type as  $ \phi_J^l,\phi_{J'}^l$ ($l=1,2,3$) respectively, and $L,M$ are sufficiently large integers. Also, by $\eqref{bridge}$
  $$ (1+|n_2-m|)^{-M} (1+|n'_2-m'|)^{-M} (1+|m|)^{-L/2} (1+|m'|)^{-L/2} \approx (1+|n_2|)^{-M'}\cdot (1+|n'_2|)^{-M'},$$
 $$ (1+|n_3-m|)^{-M} (1+|n'_3-m'|)^{-M} (1+|m|)^{-L/2} (1+|m'|)^{-L/2} \approx (1+|n_3|)^{-M'}\cdot (1+|n'_3|)^{-M'},$$
 where $M'$ can be  sufficiently large.

 That means when dealing with the paraproducts, we can go back to the original form of operators in Theorem $\ref{bithm}$, with additional decay factors $(1+|n_2|)^{-M'}(1+|n'_2|)^{-M'} \cdot (1+|n_3|)^{-M'}(1+|n'_3|)^{-M'}$. Thus,

 \begin{eqnarray*}
 \nonumber
 & &\|\sum_{\substack{|n_1|,|n'_1|\leq 10\\n_2,n'_2,n_3,n'_3\in \z}} T_{1,1}^{E_1,E_1',0}(f\cdot \chi_{I_{n_1}}\otimes \chi_{I_{n'_1}}, g \cdot \chi_{I_{n_2}}\otimes \chi_{I_{n'_2}}, h \cdot\chi_{I_{n_3}}\otimes \chi_{I_{n'_3}})(x) \|_r  \\
 &\lesssim & \sum_{\substack{|n_1|,|n'_1|\leq 10\\n_2,n'_2,n_3,n'_3\in \z}} (1+|n_2|)^{-M'}(1+|n'_2|)^{-M'} \cdot (1+|n_3|)^{-M'}(1+|n'_3|)^{-M'} \\
 & & \cdot \| f\cdot \chi_{I_{n_1}}\otimes \chi_{I_{n'_1}}\|_{p_1} \| g \cdot \chi_{I_{n_2}}\otimes \chi_{I_{n'_2}}\|_{p_2} \| h\cdot \chi_{I_{n_3}}\otimes \chi_{I_{n'_3}}\|_{p_3}\\
 & \lesssim  &\|f\tdd\|_{p_1} \|g\tdd\|_{p_2} \|h\tdd\|_{p_3}.
\end{eqnarray*}

For the cases $n_1\leq 10, \, n'_1>10$ or $n_1>10, \, n'_1\leq 10$, one just needs to combine the ideas in the above two situations together and use Theorem $\ref{bithm}$. We omit the details and the case $T^{E_1,E'_1,0}_{ab}$ has been done, where $I,I'\subseteq 5I^0$.

\medskip
Now we turn to the study of the operators $T^{E_2,E_2',0}_{1,1}$.
 \begin{eqnarray*}
& &\|T^{E_2,E_2',0}_{1,1}(f,g,h)(x)\|_r^r \\
& =&\|\sum_{\substack{I\subseteq (5I^0)^c \\ I'\subseteq (5I'^0)^c }} {1\over |I|^{1\over 2}} {1\over |I'|^{1\over 2}}\langle f,\phi_I^1\otimes \phi_{I'}^1 \rangle \langle B^1_{I,I'}(g,h),\phi_I^2\otimes \phi_{I'}^2 \rangle \phi_I^3 \otimes \phi_{I'}^3 \varphi_0(x_1)\varphi_0'(x_2)\|_r^r\\
 &=&\|\sum_{\substack{I\in \mathcal{I}\\I'\in \mathcal{I'}}}\sum_{\substack{J\in \mathcal{J},J'\in \mathcal{J'}\\|\omega^3_J|\leq |\omega^2_I|\\|\omega^3_{J'}|\leq |\omega^2_{I'}|}} \frac{1}{|I|^{1\over 2}}\frac{1}{|J|^{1\over 2}}\frac{1}{|I'|^{1\over 2}}\frac{1}{|J'|^{1\over 2}}  \langle f,\phi_I^1\otimes \phi_{I'}^1 \rangle \langle g ,\phi_J^1\otimes \phi_{J'}^1\rangle \\
  & &\ \cdot\langle h  ,\phi_J^2\otimes \phi_{J'}^2  \rangle \langle \phi^2_I\otimes \phi^2_{I'},\phi^3_J\otimes \phi^3_{J'}\rangle\phi^3_I(x_1)\phi^3_{I'}(x_2)\varphi_0(x_1)\varphi'_0(x_2)\|_r^r\\
 &\lesssim& \sum_{\substack{|n|,|n'|\geq 5\\m,m'\in\z}}  \sum_{\substack{I\subseteq I_n\\I'\subseteq I_{n'}}} \sum_{\substack{J\subseteq J_m\\|\omega^3_J|\leq |\omega^2_I|}}  \sum_{\substack{J'\subseteq J_{m'}\\|\omega^3_{J'}|\leq |\omega^2_{I'}|}} \|\frac{1}{|I|^{1\over 2}}\frac{1}{|J|^{1\over 2}}\frac{1}{|I'|^{1\over 2}}\frac{1}{|J'|^{1\over 2}}  \langle f ,\phi_I^1\otimes \phi_{I'}^1 \rangle \\
  & &\ \langle g ,  \phi_J^1\otimes \phi_{J'}^1\rangle \langle h ,\phi_J^2\otimes \phi_{J'}^2  \rangle \langle \phi^2_I\otimes \phi^2_{I'},\phi^3_J\otimes \phi^3_{J'}\rangle\phi^3_I(x_1)\phi^3_{I'}(x_2)\varphi_0(x_1)\varphi'_0(x_2)\|_r^r .
\end{eqnarray*}
We  use H\"older's inequality and take advantage of the decay factors as before
\begin{eqnarray}
\nonumber
  & &\sum_{\substack{|n|,|n'|\geq 5\\m,m'\in\z}} \sum_{\substack{i,j\geq 0\\ i',j'\geq 0}} \sum_{\substack{I \subseteq I_n, J\subseteq J_m\\|I|=2^{-i},|J|=2^{-j} }}
  \sum_{\substack{I' \subseteq I_{n'}, J'\subseteq J_{m'}\\|I'|=2^{-i'},|J'|=2^{-j'} }} \|\frac{1}{|I|^{1\over 2}}\frac{1}{|J|^{1\over 2}}\frac{1}{|I'|^{1\over 2}}\frac{1}{|J'|^{1\over 2}} \langle f ,\phi_I^1\otimes \phi_{I'}^1 \rangle  \\
  & &  \ \langle g , \phi_J^1\otimes \phi_{J'}^1\rangle \langle h  ,\phi_J^2\otimes \phi_{J'}^2  \rangle \langle \phi^2_I\otimes \phi^2_{I'},\phi^3_J\otimes \phi^3_{J'}\rangle  \phi^3_I(x_1)\phi^3_{I'}(x_2)\varphi_0(x_1)\varphi'_0(x_2)\|_r^r \nonumber  \\
  \nonumber
  &\lesssim&  \sum_{\substack{|n|,|n'|\geq 5\\m,m'\in\z}} \sum_{\substack{i,j\geq 0\\ i',j'\geq 0}} \sum_{\substack{I \subseteq I_n, J\subseteq J_m\\|I|=2^{-i},|J|=2^{-j} }}
  \sum_{\substack{I' \subseteq I_{n'}, J'\subseteq J_{m'}\\|I'|=2^{-i'},|J'|=2^{-j'} }} \big( \frac{1}{|I|^2}\frac{1}{|J|^2}\frac{1}{|I'|^2}\frac{1}{|J'|^2}
   \\ \nonumber
   & &\cdot(\|f\chi_{I_n}\otimes \chi_{I_{n'}}\|_{p_1}(|I||I'|)^{p_1-1\over p_1}) (\|g\chi_{J_m}\otimes \chi_{J_{m'}} \|_{p_2} (|J||J'|)^{p_2-1\over p_2}) \\ \nonumber
 & &  \cdot (\|h \chi_{J_m}\otimes \chi_{J_{m'}}\|_{p_3}(|J||J')|^{p_3-1\over p_3}) (|I||I'|)^{1\over r} (1+\frac{\dst(I,I^0)}{|I|})^{-M_3}\\ \nonumber
 & & \cdot (1+\frac{\dst(I',I^0)}{|I'|})^{-M'_3}   \int_{\nn} (1+\frac{\dst (x_1,I)}{|I|})^{-M_2} (1+\frac{\dst (x_2,I')}{|I'|})^{-M'_2}\\
  & &\qquad  \cdot (1+\frac{\dst(x_1,J)}{|J|})^{-N_3} (1+\frac{\dst(x_2,J')}{|J'|})^{-N'_3} dx\big)^r  \nonumber \\
 \nonumber
 &\lesssim&  \sum_{\substack{|n|,|n'|\geq 5\\m,m'\in\z}} \sum_{\substack{i,j\geq 0\\ i',j'\geq 0}} \sum_{\substack{I \subseteq I_n, J\subseteq J_m\\|I|=2^{-i},|J|=2^{-j} }}
  \sum_{\substack{I' \subseteq I_{n'}, J'\subseteq J_{m'}\\|I'|=2^{-i'},|J'|=2^{-j'} }}  \big(2^{i+i'} (1+2^i(|n|-2))^{-M_3}\\
  & &     \cdot (1+2^{i'}(|n'|-2))^{-M'_3} (1+|n-m|)^{-N_0} (1+|n'-m'|)^{-N'_0} \nonumber \\
  & & \cdot\|f\chi_{I_n}\otimes \chi_{I_{n'}}\|_{p_1}  \|g\chi_{J_m}\otimes \chi_{J_{m'}} \|_{p_2} \|h \chi_{J_m}\otimes \chi_{J_{m'}}\|_{p_3}
  \big)^r, \nonumber \\
  &\lesssim &\sum_{\substack{|n|,|n'|\geq 5\\m,m'\in\z}}    \big( (|n|-2)^{-M_3\over 2}   (|n'|-2)^{-M'_3\over 2} (1+|n-m|)^{-N_0} (1+|n'-m'|)^{-N'_0}\nonumber \\
  & & \cdot\|f\chi_{I_n}\otimes \chi_{I_{n'}}\|_{p_1}  \|g\chi_{J_m}\otimes \chi_{J_{m'}} \|_{p_2} \|h \chi_{J_m}\otimes \chi_{J_{m'}}\|_{p_3}
  \big)^r,
  \label{nmainest}
\end{eqnarray}
where again $M_j,N_j,M'_j,N'_j,\, j=1,2,3$ are sufficiently large integers. The last inequality holds since  $|n|,|n'|\geq 5$ and $i\geq j,\,i'\geq j'$, from which we can get a decay for $i,i'$ and consequently for $j,j'$ as well. Similar to $\eqref{bridge}$, there exist large integers $L,L'$
$$ (|n|-2)^{-{M_3\over 6}} (1+|n-m|)^{-N_0} \lesssim (1+|m|)^{-L}, $$
$$ (|n'|-2)^{-{M'_3\over 6}} (1+|n'-m'|)^{-N'_0} \lesssim (1+|m'|)^{-L'},$$
and also
$$(|n|-2)^{-{M\over 6}}\tilde \chi_{I_n}\lesssim \tilde \chi_{I^0}\qquad \text{and}\qquad (|n'|-2)^{-{M'\over 6}}\tilde \chi_{I_{n'}}\lesssim \tilde \chi_{I^0},$$
$$(1+|m|)^{-{L\over 3}}\tilde \chi_{J_{m}}\lesssim \tilde \chi_{I^0}\qquad \text{and}\qquad (1+|m'|)^{-{L'\over 3}}\tilde \chi_{J_{m'}}\lesssim \tilde \chi_{I^0}.$$
Then  $\eqref{nmainest}$ can be estimated by
\begin{eqnarray*}
&\lesssim & \sum_{|n|,|n'|\geq 5}\sum_{m,m'\in\z}((|n-2|^{-{M_3\over 3}})(|n'-2|^{-{M'_3\over 3}})(1+|m|)^{-L} (1+|m'|)^{-L'} \\
& &\cdot \|f\chi_{I_n}\otimes \chi_{I_{n'}}\|_{p_1}  \|g\chi_{J_m}\otimes \chi_{J_{m'}} \|_{p_2} \|h \chi_{J_m}\otimes \chi_{J_{m'}}\|_{p_3})^r \\
&\lesssim&\sum_{|n|,|n'|\geq 5}\sum_{m,m'\in\z}(|n-2|^{-{M_3\over 6}} |n'-2|^{-{M'_3\over 6}} (1+|m|)^{-{L\over 3}} (1+|m'|)^{-{L'\over 3}} \\
 & & \cdot \|f \tdd\|_{p_1}\|g \tdd\|_{p_2}\|h \tdd \|_{p_3})^r \\
& \lesssim & (\|f\tdd\|_{p_1}\|g\tdd\|_{p_2}\|h\tdd\|_{p_3})^r,
\end{eqnarray*}
Now we have proved the desired estimate for $T_{1,1}^{E_2,E'_2,0}(f,g,h)(x)$.

\medskip

 For  $T_{1,1}^{E_1,E'_2,0}(f,g,h)(x)$, we just need to combine the ideas for $T_{1,1}^{E_1,E'_1,0}(f,g,h)(x)$ and $T_{1,1}^{E_2,E'_2,0}(f,g,h)(x)$ together. More precisely, since we have $I\subseteq 5I^0$ and  $I'\subseteq (5I^0)^c$, we can do the decomposition
 $$f(x)=\sum_{n_1\in \z}f(x)\cdot  \chi_{I_{n_1}}(x_1),\ \  g(x)=\sum_{n_2\in \z} g(x) \cdot \chi_{I_{n_2}}(x_1) ,\ \  h(x)=\sum_{n_3 \in \z}h(x) \cdot \chi_{I_{n_3}}(x_1).$$

 As before, first consider $|n_1|\geq 10$.
 \begin{eqnarray*}
  & &\|T_{1,1}^{E_1,E_2',0}(f\chi_{I_{n_1}},\  g  \chi_{I_{n_2}}, \ h\chi_{I_{n_3}})(x)\|_r^r \\
  &\lesssim&  \sum_{\substack{|n'|\geq 5\\ m,m'\in \z}} \sum_{\substack{i,i'\geq 0\\j,j'\geq 0}}\sum_{\substack{I\subseteq 5I^0\\ |I|=2^{-i}}}  \sum_{\substack{I' \subseteq I_{n'}\\|I'|=2^{-i'}}} \sum_{\substack{J\subseteq J_m,|J|=2^{-j}\\J'\subseteq J_{m'},|J'|=2^{-j'}}}  \big({1\over |I||I'|}(1+\frac{\dst(I_{n_1},I)}{|I|})^{-M_1} \\
  & &  (1+\frac{\dst(I_{n_2},J)}{|J|})^{-N_1}  (1+\frac{\dst(I_{n_3},J)}{|J|})^{-N_2} (1+\frac{\dst(I',I^0)}{|I'|})^{-M'_3} \\
   & & \int_{\nn} (1+\frac{\dst(x_1,I)}{|I|})^{-M_2} (1+\frac{\dst(x_2,I')}{|I'|})^{-M'_2} \\
 & & \cdot   (1+\frac{\dst(x_1,J)}{|J|})^{-N_3}(1+\frac{\dst(x_2,J')}{|J'|})^{-N'_3} dx \\
   & & \cdot \|f  \chi_{I_{n_1}}\otimes \chi_{I_{n'}}\|_{p_1} \|g \chi_{I_{n_2}}\otimes \chi_{I_{n'}} \|_{p_2} \|h \chi_{I_{n_3}}\otimes\chi_{I_{n'}} \|_{p_3}\big)^r \\
 &\lesssim &   \sum_{\substack{|n'|\geq 5\\ m,m'\in \z}} \sum_{\substack{i,i'\geq 0\\j,j'\geq 0}}\sum_{\substack{I\subseteq 5I^0\\ |I|=2^{-i}}}  \sum_{\substack{I' \subseteq I_{n'}\\|I'|=2^{-i'}}} \sum_{\substack{J\subseteq J_m,|J|=2^{-j}\\J'\subseteq J_{m'},|J'|=2^{-j'}}}   (2^{i+i'} (1+2^i(|n_1|-6))^{-M_1}\\
 & &\cdot  (1+2^{i'}(|n'|-2))^{-M'_3} (1+2^j|m-n_2|)^{-N_1}  \\
 & & \cdot  (1+2^j|m-n_3|)^{-N_2} (1+|m|)^{-N_0} (1+|m'-n'|)^{-N'_0}  \\
 & &\cdot  \|f  \chi_{I_{n_1}}\otimes \chi_{I_{n'}}\|_{p_1} \|g \chi_{I_{n_2}}\otimes \chi_{I_{m'}} \|_{p_2} \|h \chi_{I_{n_3}}\otimes\chi_{I_{m'}} \|_{p_3})^r \\
&\lesssim &    \sum_{\substack{|n'|\geq 5\\ m,m'\in \z}}  ( (|n_1|-6)^{-{M_1\over 2}}(|n'|-2)^{-{M'_3\over 2}} (1+|n_2|)^{-L}  (1+|n_3|)^{-L} (1+|n'_3|)^{-L}    \\
 & & \cdot  (1+|m|)^{-{N_0\over 2}}(1+|n'-m'|)^{-{N'_0}} \\
  && \cdot \|f  \chi_{I_{n_1}}\otimes \chi_{I_{n'}}\|_{p_1} \|g \chi_{I_{n_2}}\otimes \chi_{I_{m'}} \|_{p_2} \|h \chi_{I_{n_3}}\otimes\chi_{I_{m'}} \|_{p_3})^r\\
 &\lesssim&  \sum_{|n'|\geq 5}\sum_{m' \in \z }((|n_1|-6)^{-{M_1\over 4}} (|n'|-2)^{-{M'_1\over 6}} (1+|n_2|)^{-{L\over 2}} \\
 & & \cdot(1+|n_3|)^{-{L\over 2}} \cdot(1+|m'|)^{-{L\over 3}}  \|f\tdd\|_{p_1} \|g\tdd\|_{p_2} \|h\tdd\|_{p_3})^r \\
 &\lesssim &  ( (|n_1|-6)^{-{M_1\over 4}} (1+|n_2|)^{-{L\over 2}} (1+|n_3|)^{-{L\over 2}} \|f\tdd\|_{p_1} \|g\tdd\|_{p_2} \|h\tdd\|_{p_3})^r,
\end{eqnarray*}
Thus,
\begin{eqnarray*}
 & & \|\sum_{|n_1|>10}\sum_{n_2,n_3\in \z} T_{1,1}^{E_1,E_2',0}(f\chi_{I_{n_1}}, g  \chi_{I_{n_2}}, h \chi_{I_{n_3}})(x)\|_r^r \\
  &\lesssim & \sum_{|n_1|>10} \sum_{n_2,n_3 \in \z} ( (|n_1|-6)^{-{M_1\over 4}} (1+|n_2|)^{-{L\over 2}} (1+|n_3|)^{-{L\over 2}}\\
   & & \cdot \|f\tdd\|_{p_1} \|g\tdd\|_{p_2} \|h\tdd\|_{p_3})^r\\
  &\lesssim & (\|f\tdd\|_{p_1} \|g\tdd\|_{p_2} \|h\tdd\|_{p_3})^r.
\end{eqnarray*}

When $|n_1|<10$, as before we can get some decay factors by using $\eqref{bridge}$, and with Theorem $\ref{bithm}$ one can get
\begin{eqnarray*}
 & & \|\sum_{\substack{|n_1|<10\\n_2,n_3\in \z}} T_{1,1}^{E_1,E_2',0}(f \chi_{I_{n_1}}, g  \chi_{I_{n_2}}, h \chi_{I_{n_3}})(x) \|_r^r  \\
  &\lesssim& \sum_{|n'|\geq 5}\sum_{m' \in\z} \sum_{\substack{|n_1|<10\\n_2,n_3\in \z}}  \big(  (1+|n_2|)^{-M'} (1+|n_3|)^{-M'}(|n'|-2)^{-M'}(1+|m'|)^{-M'}\\
   & & \ \| T_{1,1}^{E_1,E_2',0}(f\cdot \chi_{I_{n_1}}\otimes \chi_{I_{n'} }, g \cdot \chi_{I_{n_2}}\otimes \chi_{I_{m'} }, h \cdot\chi_{I_{n_3}}\otimes \chi_{I_{m'} })(x) \|_r\big)^r  \\
 &\lesssim & \sum_{|n'|\geq 5} \sum_{\substack{|n_1|<10\\n_2,n_3\in \z}} \big((1+|n_2|)^{-M'} (1+|n_3|)^{-M'}(|n'|-2)^{-M'}\\
 & & \cdot \| f\cdot \chi_{I_{n_1}}\otimes \chi_{I_{n'}}\|_{p_1} \| g \cdot \chi_{I_{n_2}}\otimes \chi_{I_{n'}}\|_{p_2} \| h\cdot \chi_{I_{n_3}}\otimes \chi_{I_{n'}}\|_{p_3}\big)^r \\
 & \lesssim & (\|f\tdd\|_{p_1} \|g\tdd\|_{p_2} \|h\tdd\|_{p_3})^r.
\end{eqnarray*}
for some sufficiently large integer $M'$.

We omit the remaining details, and we are done with $T_{1,1}^{E,E',0}(f,g,h)(x)$. \\
\vskip0.5cm

Then we turn to the study of $T_{1,m_{M,k_0}}^{E,E',0}(f,g,h)(x)$. From the condition that $m_{M,k_0}$ satisfies, we see $T_{1,m_{M,k_0}}^{E,E',0}(f,g,h)(x)$ corresponds to a classical trilinear paraproduct (see \cite{dai2013p,muscalu2013classical}) in the second parameter, while in the first parameter the form is like what happens for $T_{1,1}^{E,E',0}(f,g,h)(x)$.  We have
\begin{eqnarray*}
\|T_{1,m_{M,k_0}}^{E,E',0}(f,g,h)(x)\|_r^r
 & \lesssim &\|\sum_{I\in \mathcal{I},I'\in \mathcal{I'}}\sum_{\substack{J\in \mathcal{J}\\|\omega^3_J|\leq |\omega^2_I|}} \frac{1}{|I|^{1\over 2}}\frac{1}{|J|^{1\over 2}}\frac{1}{|I'|}  \langle f ,\phi_I^1\otimes \phi_{I'}^1 \rangle \langle g ,\phi_J^1\otimes \phi_{I'}^2\rangle \\
  & &\cdot\langle h  ,\phi_J^2\otimes \phi_{I'}^3  \rangle  \langle \phi^2_I,\phi^3_J \rangle \phi^3_I(x_1)\phi^4_{I'}(x_2)\varphi_0(x_1)\varphi'_0(x_2)\|_r,
\end{eqnarray*}
 where the families $(\phi^j_I)_{I\in \mathcal{I}}, (\phi^j_J)_{I\in \mathcal{J}}$ are defined as $(\phi^j_J)_{J\in \mathcal{J}}$ in Definition $\ref{defop}$,
 and two of $ (\phi^l_{I'})_{I'\in \mathcal{I'}}\,(l=1,2,3,4)$ are lacunary functions.  Actually this is an easier case than $T_{1,1}^{E,E',0}(f,g,h)(x)$, since the implicit symbol in the second parameter satisfies a stronger condition than the one in  $T_{1,1}^{E,E',0}(f,g,h)(x)$ and there is only one class of dyadic intervals on the second parameter in the above paraproducts.   That means one can mimic the proof  for $T_{1,1}^{E,E',0}(f,g,h)(x)$  to get the desired estimate, where the following theorem is needed, which plays a similar role as Theorem $\ref{bithm}$ for  $T_{1,1}^{E,E',0}(f,g,h)(x)$.

\begin{theorem}
\label{bithm2}
For  $f,g,h\in \mathcal{S}(\n^{2})$, the bi-parameter operators
\begin{equation*}
T_{m'}(f,g,h)(x):=\int_{\n^{6}}m'(\xi,\eta,\zeta)\hat f(\xi) \hat g(\eta)\hat h(\zeta)e^{2\pi i(\xi+\eta+\zeta)\cdot x}d\xi d\eta d\zeta
\end{equation*}
 map $L^{p_1}\times L^{p_2}\times L^{p_3}\to L^r$ for $1<p_1,p_2,p_3<\infty$ with $1/p_1+1/p_2+1/p_3=1/r$ and $0<r<\infty$, as long as the smooth symbol $m'$ satisfies
 \begin{gather*}
   |\partial^{\alpha_1}_{\xi_1} \partial^{\alpha_2}_{\xi_2} \partial^{\beta_1}_{\eta_1}\partial^{\beta_2}_{\eta_2}\partial^{\gamma_1}_{\zeta_1}\partial^{\gamma_2}_{\zeta_2}m'(\xi,\eta,\zeta)| \\
   \lesssim \sum_{\substack{\beta'_1+\beta''_1=\beta_1\\ \gamma'_1+\gamma''_1=\gamma_1}} (1+|\xi_1|+|\eta_1|+|\zeta_1|)^{-(\alpha_1+\beta'_1+\gamma'_1)}(1+|\eta_1|+|\zeta_1|)^{-(\beta''_1+\gamma''_1)} \\ \cdot (1+|\xi_2|+|\eta_2|+|\zeta_2|)^{-(\alpha_2+\beta_2+\gamma_2)}
 \end{gather*}
 for sufficiently many multi-indices $\alpha_1,\beta_1,\gamma_1$ and $\alpha_2,\beta_2,\gamma_2$.
 \end{theorem}

 \begin{proof}
 This is essentially a corollary of Theorem $\ref{bithm}$, since the symbol $m'$ satisfies a stronger condition in the second parameter than $m_1(\xi,\eta,\zeta) m_2(\eta,\zeta)$. To get the result,  one just needs to keep the argument in \cite{chen2014hormander} on the second parameter, and do the necessary modification as in Theorem $\ref{bithm}$ on the first parameter. We omit the details here.
 \end{proof}
 Now we are ready to prove the estimate for $T_{1,m_{M,k_0}}^{E,E',0}(f,g,h)(x)$.

 Taking advantage of the fact that $|I|,|I'|\leq 1$, we can split
 \begin{eqnarray*}
 \nonumber
 & &T_{1,m_{M,k_0}}^{E,E',0}(f,g,h)=\{(\sum_{\substack{I\subseteq 5I^0 \\ I'\subseteq 5I^0 }}+ \sum_{\substack{I\subseteq (5I^0)^c \\ I'\subseteq 5I^0 }}+ \sum_{\substack{I\subseteq 5I^0 \\ I'\subseteq (5I^0)^c }} + \sum_{\substack{I\subseteq (5I^0)^c \\ I'\subseteq (5I^0)^c }})  \sum_{\substack{J\in \mathcal{J}\\|\omega^3_J|\leq |\omega^2_I|}} \frac{1}{|I|^{1\over 2}}\frac{1}{|J|^{1\over 2}}\frac{1}{|I'|}  \\
 \nonumber
 & & \langle f ,\phi_I^1\otimes \phi_{I'}^1 \rangle \langle g ,\phi_J^1\otimes \phi_{I'}^2\rangle \langle h  ,\phi_J^2\otimes \phi_{I'}^3  \rangle  \langle \phi^2_I,\phi^3_J \rangle \phi^3_I(x_1)\phi^4_{I'}(x_2)\varphi_0(x_1)\varphi'_0(x_2)\\
  \label{mainctr2}
 &=& T_{1,m_{M,k_0}}^{E_1,E'_1,0}(f,g,h)+T_{1,m_{M,k_0}}^{E_2,E'_1,0}(f,g,h)+T_{1,m_{M,k_0}}^{E_1,E'_2,0}(f,g,h)+T_{1,m_{M,k_0}}^{E_2,E'_2,0}(f,g,h).
 \end{eqnarray*}

We start with $T_{1,m_{M,k_0}}^{E_1,E_1',0}(f,g,h)$. We still consider the decomposition $\eqref{fghdecomp}$, and  we can write
\begin{gather*}
  T_{1,m_{M,k_0}}^{E_1,E_1',0}(f,g,h)(x)= \\
  \sum_{n_1,n_1'}\sum_{n_2,n_2'}\sum_{n_3,n_3'}T_{1,m_{M,k_0}}^{E_1,E_1',0}(f\cdot \chi_{I_{n_1}}\otimes\chi_{I_{n_1'}}, g \cdot \chi_{I_{n_2}}\otimes \chi_{I_{n_2'}}, h \cdot \chi_{I_{n_3}}\otimes \chi_{I_{n_3'}})(x).
\end{gather*}

\medskip

When $|n_1|,|n_1'|> 10$, we write
\begin{eqnarray*}
 & &\|T_{1,m_{M,k_0}}^{E_1,E_1',0}(f\cdot \chi_{I_{n_1}}\otimes \chi_{I_{n'_1}}, g \cdot \chi_{I_{n_2}}\otimes \chi_{I_{n'_2}}, h \cdot\chi_{I_{n_3}}\otimes \chi_{I_{n'_3}})(x)\|_r \\
 &=&\|\sum_{I\in \mathcal{I},I'\in \mathcal{I'}}\sum_{\substack{J\in \mathcal{J}\\|\omega^3_J|\leq |\omega^2_I|}} \frac{1}{|I|^{1\over 2}}\frac{1}{|J|^{1\over 2}}\frac{1}{|I'|}  \langle f \chi_{I_{n_1}}\otimes \chi_{I_{n'_1}},\phi_I^1\otimes \phi_{I'}^1 \rangle \langle g \chi_{I_{n_2}}\otimes \chi_{I_{n'_2}},\phi_J^1\otimes \phi_{I'}^2\rangle \\
  & &\cdot\langle h  \chi_{I_{n_3}}\otimes\chi_{I_{n'_3}},\phi_J^2\otimes \phi_{I'}^3  \rangle \langle \phi^2_I,\phi^3_J\rangle \phi^3_I(x_1)\phi^4_{I'}(x_2)\varphi_0(x_1)\varphi'_0(x_2)\|_r.
\end{eqnarray*}
Then we use H\"older's inequality to get
\begin{eqnarray}
\nonumber
 & & \| \frac{1}{|I|^{1\over 2}}\frac{1}{|J|^{1\over 2}}\frac{1}{|I'|} \langle f \chi_{I_{n_1}}\otimes \chi_{I_{n'_1}},\phi_I^1\otimes \phi_{I'}^1 \rangle \langle g \chi_{I_{n_2}}\otimes \chi_{I_{n'_2}},\phi_J^1\otimes \phi_{I'}^2\rangle \\ \nonumber
  & &\quad \cdot\langle h  \chi_{I_{n_3}}\otimes\chi_{I_{n'_3}},\phi_J^2\otimes \phi_{I'}^3  \rangle \langle \phi^2_I,\phi^3_J\rangle \phi^3_I(x_1)\phi^4_{I'}(x_2)\varphi_0(x_1)\varphi'_0(x_2)\|_r  \\ \nonumber
 &\lesssim& \frac{1}{|I|^2}\frac{1}{|J|^2}\frac{1}{|I'|^3}\\ \nonumber
 & & \cdot (1+\frac{\dst(I_{n_1},I)}{|I|})^{-M_1} (1+\frac{\dst(I_{n'_1},I')}{|I'|})^{-M'_1}(\|f\chi_{I_{n_1}}\otimes\chi_{I_{n'_1}}\|_{p_1} (|I||I'|)^{p_1-1\over p_1})\\ \nonumber
 & & \cdot(1+\frac{\dst(I_{n_2},J)}{|J|})^{-N_1} (1+\frac{\dst(I_{n'_2},I')}{|I'|})^{-M'_2}
  (\|g\chi_{I_{n_2}}\otimes \chi_{I'_{n_2}}\|_{p_2}(|J||I|')^{p_2-1\over p_2}) \\ \nonumber
 & &(1+\frac{\dst(I_{n_3},J)}{|J|})^{-N_2} (1+\frac{\dst(I_{n'_3},I')}{|I'|})^{-M'_3}(\|h\chi_{I_{n_3}}\otimes \chi_{I_{n'_3}}\|_{p_3}(|J||I'|)^{p_3-1\over p_3})\\\nonumber
 & & \cdot  (|I||I'|)^{1\over r} \int_{\nn} (1+\frac{\dst (x_1,I)}{|I|})^{-M_2} (1+\frac{\dst(x_1,J)}{|J|})^{-N_3}  dx \\ \nonumber
 &\lesssim & {1\over |I|}({|I|\over |J|})^{{1\over p_2}+{1\over p_3}} (1+\frac{\dst(I_{n_1},I)}{|I|})^{-M_1} (1+\frac{\dst(I_{n'_1},I')}{|I'|})^{-M'_1} (1+\frac{\dst(I_{n_2},J)}{|J|})^{-N_1}\\ \nonumber
 & & \cdot  (1+\frac{\dst(I_{n'_2},I')}{|I'|})^{-M'_2}   (1+\frac{\dst(I_{n_3},J)}{|J|})^{-N_2} (1+\frac{\dst(I_{n'_3},I')}{|I'|})^{-M'_3} \\ \nonumber
  & &\cdot \int_{\n} (1+\frac{\dst (x_1,I)}{|I|})^{-M_2} (1+\frac{\dst(x_1,J)}{|J|})^{-N_3}  dx_1 \\ \label{dcest2}
 & &    \cdot  \|f  \chi_{I_{n_1}}\otimes \chi_{I_{n'_1}}\|_{p_1} \|g \chi_{I_{n_2}}\otimes \chi_{I_{n'_2}} \|_{p_2} \|h \chi_{I_{n_3}}\otimes\chi_{I_{n'_3}} \|_{p_3},
\end{eqnarray}
where $M_j,M'_j,N_j$ are sufficiently large integers and $\phi_I^j,\phi_J^j,\phi^k_{I'}$ are $L^2$-normalized bump functions adapted to $I,J,I'$ for $j=1,2,3$ and $k=1,2,3,4.$ Taking advantage of  $|\omega^3_J|\leq |\omega^2_I|$ and $|I|,|J|,|I'|\leq 1$, one can get
\begin{eqnarray*}
  & &\|T_{1,m_{M,k_0}}^{E_1,E_1',0}(f\cdot \chi_{I_{n_1}}\otimes \chi_{I_{n'_1}}, g \cdot \chi_{I_{n_2}}\otimes \chi_{I_{n'_2}}, h \cdot\chi_{I_{n_3}}\otimes \chi_{I_{n'_3}})(x)\|_r^r \\
  &\lesssim& \sum_{\substack{i,i'\geq 0\\j\geq 0}}\sum_{\substack{I,I'\subseteq 5I^0\\ |I|=2^{-i}\\ |I'|=2^{-i'}}}   \sum_{m\in \z} \sum_{\substack{J\subseteq J_m\\ |J|=2^{-j}}}  \big({1\over |I|}(1+\frac{\dst(I_{n_1},I)}{|I|})^{-M_1} (1+\frac{\dst(I_{n'_1},I')}{|I'|})^{-M'_1}\\ \nonumber
 & & \cdot  (1+\frac{\dst(I_{n_2},J)}{|J|})^{-N_1} (1+\frac{\dst(I_{n'_2},I')}{|I'|})^{-M'_2}   (1+\frac{\dst(I_{n_3},J)}{|J|})^{-N_2} \\ \nonumber
  & &\cdot  (1+\frac{\dst(I_{n'_3},I')}{|I'|})^{-M'_3} \int_{\n} (1+\frac{\dst (x_1,I)}{|I|})^{-M_2} (1+\frac{\dst(x_1,J)}{|J|})^{-N_3}  dx_1 \\
 & &    \cdot  \|f  \chi_{I_{n_1}}\otimes \chi_{I_{n'_1}}\|_{p_1} \|g \chi_{I_{n_2}}\otimes \chi_{I_{n'_2}} \|_{p_2} \|h \chi_{I_{n_3}}\otimes\chi_{I_{n'_3}} \|_{p_3}\big)^r \\
 &\lesssim &  \sum_{\substack{i,i'\geq 0\\j\geq 0}}\sum_{\substack{I,I'\subseteq 5I_0\\ |I|=2^{-i}\\ |I'|=2^{-i'}}}   \sum_{m\in \z} \sum_{\substack{J\subseteq J_m \\ |J|=2^{-j}}}    (2^{i} (1+2^i(|n_1|-6))^{-M_1}(1+2^{i'}(|n'_1|-6))^{-M'_1}\\
 & &\cdot  (1+2^j|m-n_2|)^{-N_1}   (1+2^j|m-n_3|)^{-N_2} (1+2^{i'}|n'_2|)^{-M'_2} (1+2^{i'}|n'_3|)^{-M'_3}  \\
 & &\cdot (1+|m|)^{-N_0}  \|f  \chi_{I_{n_1}}\otimes \chi_{I_{n'_1}}\|_{p_1} \|g \chi_{I_{n_2}}\otimes \chi_{I_{n'_2}} \|_{p_2} \|h \chi_{I_{n_3}}\otimes\chi_{I_{n'_3}} \|_{p_3})^r \\
&\lesssim &    \sum_{m\in \z}  ( (|n_1|-6)^{-{M_1\over 2}}(|n'_1|-6)^{-{M'_1\over 2}} (1+|n_2|)^{-L}  (1+|n_3|)^{-L}    \\
& & \cdot (1+|n'_2|)^{-{M'_2}} (1+|n'_3|)^{-{M'_3}} (1+|m|)^{-{N_0\over 2}}  \\
 & & \cdot   \|f  \chi_{I_{n_1}}\otimes \chi_{I_{n'_1}}\|_{p_1} \|g \chi_{I_{n_2}}\otimes \chi_{I_{n'_2}} \|_{p_2} \|h \chi_{I_{n_3}}\otimes\chi_{I_{n'_3}} \|_{p_3})^r\\
 &\lesssim & ( (1+|n_1|)^{-{M_1\over 2}}|n'_1|^{-{M'_1\over 2}} (1+|n_2|)^{-L}  (1+|n_3|)^{-L} (1+|n'_2|)^{-{M'_2}} (1+|n'_3|)^{-{M'_3}}    \\
 & & \cdot  \|f  \chi_{I_{n_1}}\otimes \chi_{I_{n'_1}}\|_{p_1} \|g \chi_{I_{n_2}}\otimes \chi_{I_{n'_2}} \|_{p_2} \|h \chi_{I_{n_3}}\otimes\chi_{I_{n'_3}} \|_{p_3})^r \\
 &\lesssim & ((|n_1|-6)^{-{M_1\over 4}} (|n'_1|-6)^{-{M'_1\over 4}} (1+|n_2|)^{-{L\over 2}} (1+|n'_2|)^{-{M'_2\over 2}} \\
 & & \cdot(1+|n_3|)^{-{L\over 2}} (1+|n'_3|)^{-{M'_3\over 2}}   \|f\tdd\|_{p_1} \|g\tdd\|_{p_2} \|h\tdd\|_{p_3})^r
\end{eqnarray*}
where the positive integers $N_0=\min\{M_2,N_3\}$, $ L$ are sufficiently large, and the summation over $i'$ is allowed since $|n'_1|\geq 10$.  Thus,
\begin{eqnarray*}
 & &\|\sum_{\substack{|n_1|>10 \\|n'_1|> 10}} \sum_{n_2,n_3,n'_2,n'_3 \in \z }T_{1,1}^{E_1,E_1',0}(f\cdot \chi_{I_{n_1}}\otimes \chi_{I_{n'_1}}, g \cdot \chi_{I_{n_2}}\otimes \chi_{I_{n'_2}}, h \cdot\chi_{I_{n_3}}\otimes \chi_{I_{n'_3}})(x)\|_r^r \\
 &\lesssim& \sum_{|n_1|>10,|n'_1|> 10}\sum_{n_2,n_3,n'_2,n'_3 \in \z } ((|n_1|-6)^{-{M_1\over 4}} (|n'_1|-6)^{-{M'_1\over 4}} (1+|n_2|)^{-{L\over 2}} (1+|n'_2|)^{-{M'_2\over 2}} \\
 & & \cdot(1+|n_3|)^{-{L\over 2}} (1+|n'_3|)^{-{M'_3\over 2}}   \|f\tdd\|_{p_1} \|g\tdd\|_{p_2} \|h\tdd\|_{p_3})^r \\
 &\lesssim &  (\|f\tdd\|_{p_1} \|g\tdd\|_{p_2} \|h\tdd\|_{p_3})^r,
\end{eqnarray*}

\medskip

As before when $|n_1|\leq 10$ or $|n'_1|\leq 10$  things are different, since we cannot take the summation over $i$ or $i'$.  Without loss of generality, we assume both $|n_1|\leq 10$ and $|n'_1|\leq 10$.

\begin{eqnarray*}
 \nonumber
 & &\|\sum_{\substack{|n_1|,|n'_1|\leq 10\\n_2,n'_2,n_3,n'_3\in \z}}T_{1,m_{M,k_0}}^{E_1,E_1',0}(f\cdot \chi_{I_{n_1}}\otimes \chi_{I_{n'_1}}, g \cdot \chi_{I_{n_2}}\otimes \chi_{I_{n'_2}}, h \cdot\chi_{I_{n_3}}\otimes \chi_{I_{n'_3}})(x) \|_r  \\ \label{spc}
 & \lesssim  &\|f\tdd\|_{p_1} \|g\tdd\|_{p_2} \|h\tdd\|_{p_3}.
\end{eqnarray*}

 Recall that when $I,I'\subseteq 5I^0$ and $J\in J_m$,   in $\eqref{dcest2}$ we can write
 $$\langle \phi^2_I ,\phi^3_J\rangle\approx (1+|m|)^{-L} \langle \phi^2_I ,\tilde \phi^3_J \rangle,$$
 $$\langle g  \chi_{I_{n_2}}\otimes\chi_{I_{n'_2}},\phi_J^1\otimes \phi_{I'}^2  \rangle \approx (1+|n_2-m|)^{-M} (1+|n'_2|)^{-M'}\langle g \chi_{I_{n_2}}\otimes\chi_{I_{n'_2}},\tilde \phi_J^1\otimes \tilde \phi_{I'}^2 \rangle, $$
  $$\langle h  \chi_{I_{n_3}}\otimes\chi_{I_{n'_3}},\phi_J^2\otimes \phi_{I'}^3  \rangle \approx (1+|n_3-m|)^{-M} (1+|n'_3|)^{-M'} \langle h  \chi_{I_{n_3}}\otimes\chi_{I_{n'_3}},\tilde \phi_J^2\otimes \tilde \phi_{I'}^3 \rangle, $$
  $$ (1+|n_2-m|)^{-M}  (1+|m|)^{-L/3}  \approx (1+|n_2|)^{-M'},$$
 $$ (1+|n_3-m|)^{-M} (1+|m|)^{-L/3}  \approx (1+|n_3|)^{-M'},$$
 where $M'$ can be  sufficiently large.

 That means we can use Theorem $\ref{bithm2}$ with additional decay factors $(1+|n_2|)^{-M'} (1+|n_3|)^{-M'}(1+|n'_2|)^{-M'} (1+|n'_3|)^{-M'}$. Thus,

 \begin{eqnarray*}
 \nonumber
 & &\|\sum_{\substack{|n_1|,|n'_1|\leq 10\\n_2,n'_2,n_3,n'_3\in \z}} T_{1,1}^{E_1,E_1',0}(f\cdot \chi_{I_{n_1}}\otimes \chi_{I_{n'_1}}, g \cdot \chi_{I_{n_2}}\otimes \chi_{I_{n'_2}}, h \cdot\chi_{I_{n_3}}\otimes \chi_{I_{n'_3}})(x) \|_r  \\
 &\lesssim & \sum_{\substack{|n_1|,|n'_1|\leq 10\\n_2,n'_2,n_3,n'_3\in \z}} (1+|n_2|)^{-M'}(1+|n'_2|)^{-M'} \cdot (1+|n_3|)^{-M'}(1+|n'_3|)^{-M'} \\
 & & \cdot \| f\cdot \chi_{I_{n_1}}\otimes \chi_{I_{n'_1}}\|_{p_1} \| g \cdot \chi_{I_{n_2}}\otimes \chi_{I_{n'_2}}\|_{p_2} \| h\cdot \chi_{I_{n_3}}\otimes \chi_{I_{n'_3}}\|_{p_3}\\
 & \lesssim  &\|f\tdd\|_{p_1} \|g\tdd\|_{p_2} \|h\tdd\|_{p_3}.
\end{eqnarray*}

For the cases $n_1\leq 10, \, n'_1>10$ or $n_1>10, \, n'_1\leq 10$, one just needs to combine the ideas in the above two situations together and use Theorem $\ref{bithm2}$. Then we are done with $T^{E_1,E'_1,0}_{1,m_{M,k_0}}(f,g,h)$.

Now we turn to the study of the operators $T^{E_2,E_2',0}_{1,m_{M,k_0}}$.
 \begin{eqnarray*}
& &\|T^{E_2,E_2',0}_{1,m_{M,k_0}}(f,g,h)(x)\|_r^r \\
& =&\|\sum_{\substack{I\subseteq (5I^0)^c \\ I'\subseteq (5I'^0)^c }} \sum_{\substack{J\in \mathcal{J}\\|\omega^3_J|\leq |\omega^2_I|}}  \frac{1}{|I|^{1\over 2}}\frac{1}{|J|^{1\over 2}}\frac{1}{|I'|}  \langle f ,\phi_I^1\otimes \phi_{I'}^1 \rangle \langle g,\phi_J^1\otimes \phi_{I'}^2\rangle \\
  & &\ \cdot\langle h  ,\phi_J^2\otimes \phi_{I'}^3  \rangle \langle \phi^2_I,\phi^3_J\rangle \phi^3_I(x_1)\phi^4_{I'}(x_2)\varphi_0(x_1)\varphi'_0(x_2)\|_r^r\\
 &\lesssim& \| \sum_{\substack{|n|,|n'|\geq 5\\m\in\z}} \sum_{\substack{i,j\geq 0\\ i'\geq 0}} \sum_{\substack{I \subseteq I_n, J\subseteq J_m\\|I|=2^{-i},|J|=2^{-j} }}
  \sum_{\substack{I' \subseteq I_{n'} \\|I'|=2^{-i'} }}  \frac{1}{|I|^{1\over 2}}\frac{1}{|J|^{1\over 2}}\frac{1}{|I'|}  \langle f ,\phi_I^1\otimes \phi_{I'}^1 \rangle \langle g,\phi_J^1\otimes \phi_{I'}^2\rangle \\
  & &\ \cdot\langle h  ,\phi_J^2\otimes \phi_{I'}^3  \rangle \langle \phi^2_I,\phi^3_J\rangle \phi^3_I(x_1)\phi^4_{I'}(x_2)\varphi_0(x_1)\varphi'_0(x_2)\|_r^r\\
  &\lesssim & \sum_{\substack{|n|,|n'|\geq 5\\m\in\z}} \sum_{\substack{i,j\geq 0\\ i'\geq 0}} \sum_{\substack{I \subseteq I_n, J\subseteq J_m\\|I|=2^{-i},|J|=2^{-j} }} \sum_{\substack{I' \subseteq I_{n'} \\|I'|=2^{-i'} }}   \big({1\over |I|}  (1+\frac{\dst(I,I^0)}{|I|})^{-M_3}\\ \nonumber
  & &\cdot  (1+\frac{\dst(I',I^0)}{|I'|})^{-M'_4}   \int_{\n} (1+\frac{\dst (x_1,I)}{|I|})^{-M_2} (1+\frac{\dst(x_1,J)}{|J|})^{-N_3}  dx_1 \nonumber  \\
 & &     \|f  \chi_{I_{n}}\otimes \chi_{I_{n'}}\|_{p_1} \|g \chi_{J_m}\otimes \chi_{I_{n'}} \|_{p_2} \|h \chi_{J_m}\otimes\chi_{I_{n'}} \|_{p_3}\big)^r\\
 &\lesssim & \sum_{\substack{|n|,|n'|\geq 5\\m\in\z}} \sum_{\substack{i,j\geq 0\\ i'\geq 0}} \sum_{\substack{I \subseteq I_n, J\subseteq J_m\\|I|=2^{-i},|J|=2^{-j} }} \sum_{\substack{I' \subseteq I_{n'} \\|I'|=2^{-i'} }}   \big(2^{i} (1+2^i(|n|-2))^{-M_3}   (1+2^{i'}(|n'|-2))^{-M'_4} \\
  & & \cdot(1+|n-m|)^{-N_0}\|f\chi_{I_n}\otimes \chi_{I_{n'}}\|_{p_1}  \|g\chi_{J_m}\otimes \chi_{I_{n'}} \|_{p_2} \|h \chi_{J_m}\otimes \chi_{I_{n'}}\|_{p_3}
  \big)^r \\
  &\lesssim & \sum_{\substack{|n|,|n'|\geq 5\\m\in\z}}    \big( (|n|-2)^{-M_3\over 2}   (|n'|-2)^{-M'_4\over 2} (1+|n-m|)^{-N_0} \nonumber \\
  & & \cdot   \|f  \chi_{I_{n}}\otimes \chi_{I_{n'}}\|_{p_1} \|g \chi_{J_m}\otimes \chi_{I_{n'}} \|_{p_2} \|h \chi_{J_m}\otimes\chi_{I_{n'}} \|_{p_3}\big)^r \\
    &\lesssim & (\|f\tdd\|_{p_1} \|g\tdd\|_{p_2} \|h\tdd\|_{p_3})^r.
\end{eqnarray*}

\medskip

 For  $T_{1,m_{M,k_0}}^{E_1,E'_2,0}(f,g,h)(x)$, we just need to combine the ideas for $T_{1,m_{M,k_0}}^{E_1,E'_1,0}(f,g,h)(x)$ and $T_{1,m_{M,k_0}}^{E_2,E'_2,0}(f,g,h)(x)$ together. More precisely, since we have $I\subseteq 5I^0$ and  $I'\subseteq (5I^0)^c$, we can do the decomposition
 $$f(x)=\sum_{n_1\in \z}f \cdot\chi_{I_{n_1}}(x_1),\ \  g(x)=\sum_{n_2\in \z} g\cdot \chi_{I_{n_2}}(x_1) ,\ \  h(x)=\sum_{n_3 \in \z}h \cdot \chi_{I_{n_3}}(x_1).$$

 As before first consider when $|n_1|\geq 10$.
 \begin{eqnarray*}
  & &\|T_{1,m_{M,k_0}}^{E_1,E_2',0}(f\chi_{I_{n_1}},\  g  \chi_{I_{n_2}}, \ h\chi_{I_{n_3}})(x)\|_r^r \\
  &\lesssim&  \sum_{\substack{|n'|\geq 5\\ m\in \z}} \sum_{\substack{i,i'\geq 0\\j\geq 0}}\sum_{\substack{I\subseteq 5I^0\\ |I|=2^{-i}}}  \sum_{\substack{I' \subseteq I_{n'}\\|I'|=2^{-i'}}} \sum_{\substack{J\subseteq J_m\\|J|=2^{-j}}}  \big({1\over |I|}(1+\frac{\dst(I_{n_1},I)}{|I|})^{-M_1} \\
  & & (1+\frac{\dst(I_{n_2},J)}{|J|})^{-N_1}  (1+\frac{\dst(I_{n_3},J)}{|J|})^{-N_2} (1+\frac{\dst(I',I^0)}{|I'|})^{-M'_4}  \\
 & & \cdot   \int_{\n} (1+\frac{\dst(x_1,I)}{|I|})^{-M_2} (1+\frac{\dst(x_1,J)}{|J|})^{-N_3} dx_1 \\
   & & \cdot \|f  \chi_{I_{n_1}}\otimes \chi_{I_{n'}}\|_{p_1} \|g \chi_{I_{n_2}}\otimes \chi_{I_{n'}} \|_{p_2} \|h \chi_{I_{n_3}}\otimes\chi_{I_{n'}} \|_{p_3}\big)^r \\
 &\lesssim &  \sum_{\substack{|n'|\geq 5\\ m\in \z}} \sum_{\substack{i,i'\geq 0\\j\geq 0}}\sum_{\substack{I\subseteq 5I^0\\ |I|=2^{-i}}}  \sum_{\substack{I' \subseteq I_{n'}\\|I'|=2^{-i'}}} \sum_{\substack{J\subseteq J_m\\|J|=2^{-j}}}   (2^{i} (1+2^i(|n_1|-6))^{-M_1}(1+2^{i'}(|n'|-2))^{-M'_4}\\
 & &\cdot  (1+2^j|m-n_2|)^{-N_1}   (1+2^j|m-n_3|)^{-N_2} (1+|m|)^{-N_0}   \\
 & &\cdot  \|f  \chi_{I_{n_1}}\otimes \chi_{I_{n'}}\|_{p_1} \|g \chi_{I_{n_2}}\otimes \chi_{I_{n'}} \|_{p_2} \|h \chi_{I_{n_3}}\otimes\chi_{I_{n'}} \|_{p_3})^r \\
&\lesssim &    \sum_{\substack{|n'|\geq 5\\ m\in \z}}  ( (|n_1|-6)^{-{M_1\over 2}}(|n'|-2)^{-{M'_1\over 2}} (1+|n_2|)^{-L}  (1+|n_3|)^{-L}    \\
 & & \cdot  (1+|m|)^{-{N_0\over 2}} \|f  \chi_{I_{n_1}}\otimes \chi_{I_{n'}}\|_{p_1} \|g \chi_{I_{n_2}}\otimes \chi_{I_{n'}} \|_{p_2} \|h \chi_{I_{n_3}}\otimes\chi_{I_{n'}} \|_{p_3})^r\\
 &\lesssim&  \sum_{|n'|\geq 5}((|n_1|-6)^{-{M_1\over 4}} (|n'|-2)^{-{M'_1\over 6}} (1+|n_2|)^{-{L\over 2}} (1+|n_3|)^{-{L\over 2}}  \\
 & & \cdot\|f\tdd\|_{p_1} \|g\tdd\|_{p_2} \|h\tdd\|_{p_3})^r \\
 &\lesssim &  ( (|n_1|-6)^{-{M_1\over 4}} (1+|n_2|)^{-{L\over 2}} (1+|n_3|)^{-{L\over 2}} \|f\tdd\|_{p_1} \|g\tdd\|_{p_2} \|h\tdd\|_{p_3})^r.
\end{eqnarray*}
Thus,
\begin{eqnarray*}
 & & \|\sum_{|n_1|>10}\sum_{n_2,n_3\in \z} T_{1,m_{M,k_0}}^{E_1,E_2',0}(f\cdot \chi_{I_{n_1}}, g \cdot \chi_{I_{n_2}}, h \cdot\chi_{I_{n_3}})(x)\|_r^r \\
  &\lesssim & \sum_{|n_1|>10} \sum_{n_2,n_3 \in \z} ( (|n_1|-6)^{-{M_1\over 4}} (1+|n_2|)^{-{L\over 2}} (1+|n_3|)^{-{L\over 2}} \\
   & & \cdot \|f\tdd\|_{p_1} \|g\tdd\|_{p_2} \|h\tdd\|_{p_3})^r\\
  &\lesssim & (\|f\tdd\|_{p_1} \|g\tdd\|_{p_2} \|h\tdd\|_{p_3})^r.
\end{eqnarray*}

When $|n_1|<10$, as before we need Theorem $\ref{bithm2}$ and some decay factors by  $\eqref{bridge}$, and the following holds
\begin{eqnarray*}
 & & \|\sum_{\substack{|n_1|<10\\n_2,n_3\in \z}} T_{1,m_{M,k_0}}^{E_1,E_2',0}(f\cdot \chi_{I_{n_1}}, g \cdot \chi_{I_{n_2}}, h \cdot\chi_{I_{n_3}})(x) \|_r^r  \\
 &\lesssim& \sum_{|n'|\geq 5} \sum_{\substack{|n_1|<10\\n_2,n_3\in \z}}    (1+|n_2|)^{-M'} (1+|n_3|)^{-M'}(|n'|-2)^{-M'}\\
   & & \cdot \| T_{1,1}^{E_1,E_2',0}(f\cdot \chi_{I_{n_1}}\otimes \chi_{I_{n'} }, g \cdot \chi_{I_{n_2}}\otimes \chi_{I_{n'} }, h \cdot\chi_{I_{n_3}}\otimes \chi_{I_{n'} })(x) \|_r^r  \\
 &\lesssim & \sum_{|n'|\geq 5} \sum_{\substack{|n_1|<10\\n_2,n_3\in \z}} (1+|n_2|)^{-M'}\cdot (1+|n_3|)^{-M'}(|n'|-2)^{-M'}\\
 & & \cdot \| f\cdot \chi_{I_{n_1}}\otimes \chi_{I_{n'}}\|_{p_1} \| g \cdot \chi_{I_{n_2}}\otimes \chi_{I_{n'}}\|_{p_2} \| h\cdot \chi_{I_{n_3}}\otimes \chi_{I_{n'}}\|_{p_3}\\
 & \lesssim & \|f\tdd\|_{p_1} \|g\tdd\|_{p_2} \|h\tdd\|_{p_3}.
\end{eqnarray*}
where $M'$ is sufficiently large. Then we are done with $T_{1,m_{M,k_0}}^{E_1,E_1',0}$.

Moreover, it is obvious that $T_{1,m_{M,k_0}}^{E_2,E_1',0}$ can be treated similarly. We omit the details here. Now we are done with $T_{1,m_{M,k_0}}^{E,E',0}(f,g,h)(x)$. Now we have proved the desired estimate for  the operators $T_{ab}^{E,E',0}(f,g,h)(x)$.
\medskip

\subsection{$T^{E,G',0}_{ab},T^{G,E',0}_{ab} $}
~\\
Consider
\begin{eqnarray*}
 & & T^{E,G',0}_{ab}(f,g,h)(x)\\
 &\approx& \sum_{I\in \mathcal{I}} \sum_{I'\in \mathcal{I'}}\sum_{\substack{J\in \mathcal{J}\\ |\omega_J^3|\leq |\omega^2_I|}}{1\over |J|^{1\over 2}} {1\over |I|^{1\over 2}}{1\over |I'|}  \langle f,\phi_I^1\otimes \phi_{I'}^1 \rangle \langle g,\phi_J^1\otimes \phi_{I'}^2\rangle \langle h,\phi_{I}^2 \otimes \phi_{I'}^3 \rangle \\
 & & \qquad \qquad \cdot \langle \phi_J^3, \phi_I^2\rangle \phi_I^3(x_1) \phi_{I'}^4(x_2)\varphi_0 (x_1) \varphi_0' (x_2)\\
 & &  +(\int_{\n^6} m^1_{M,k_0}(\xi_1,\eta_1,\zeta_1)(\sum_{k\geq0}\phi^1_k(\xi_2)\phi^2_k(\eta_2)\phi^3_k(\zeta_2)) \\
  & & \qquad \qquad \cdot \hat f(\xi) \hat g(\eta) \hat h(\zeta)d\xi d\eta d\zeta)\varphi_0(x_1)\varphi'_0(x_2) \\
 &:= & T^{E,G',0,1}_{ab}+T^{E,G',0,2}_{ab},
\end{eqnarray*}
where $m^1_{M,k_0}$ is as described in Lemma $\ref{lm}$. Also note
$$|\partial^\alpha_{\xi_2}\partial^\beta_{\eta_2} \partial^\gamma_{\zeta_2} (\sum_{k\geq0}\phi^1_k(\xi_2)\phi^2_k(\eta_2)\phi^3_k(\zeta_2))| \lesssim \frac{1}{(1+|\xi_2|+|\eta_2|+|\zeta_2|)^{\alpha+\beta+\gamma}} $$
for sufficiently many indices. Then the desired estimate for $T^{E,G',0,1}_{ab}$ follows from the same argument as for  $T^{E,E',0}_{1,m^1_{M,k_0}}$,
and the estimate for $T^{E,G',0,2}_{ab}$ follows from Theorem $\ref{trilocalth}$.
\medskip

Having treated all the cases in Theorem $\ref{localth}$, the proof of Theorem $\ref{bithp}$ is concluded.

\appendix

\section{}
\label{ap1}
Here we give some details about the reductions used in the Leibniz rule. Here we will still use the notations introduced in Section \ref{sectionlb}. We start from the reduction of $D^{\beta_1}_1 D^{\beta_2}_2 (gh)$, and this part has appeared in \cite{muscalu2004bi,muscalu2013classical,muscalu2004multi}. We recall some arguments here.
Let \begin{eqnarray}
& &1(\eta_1,\eta_2,\zeta_1,\zeta_2) \nonumber \\
&=& \left( \sum _ { k _ { 1 } } \widehat { \psi } _ { k _ { 1 } } \left( \eta _ { 1 } \right) \sum _ { l _ { 1 } } \widehat { \psi } _ { l _ { 1 } } \left( \zeta _ { 1 } \right) \right)\left( \sum _ { k _ { 2 } } \widehat { \psi } _ { k _ { 2 } } \left( \eta _ { 2} \right) \sum _ { l _ { 2 } } \widehat { \psi } _ { l _ { 2 } } \left( \zeta _ { 2 } \right) \right) \nonumber \\
 &=& \left(\sum _ { k_1 } \widehat { \varphi } _ { k_1 } \left( \eta _ { 1 } \right) \widehat { \psi } _ { k_1 } \left( \zeta _ { 1 } \right) + \sum _ { k_ 1} \widehat { \psi } _ { k_1 } \left( \eta _ { 1 } \right) \widehat { \psi } _ { k_1 } \left( \zeta _ { 1 } \right) + \sum _ { k_1 } \widehat { \psi } _ { k_1 } \left( \eta _ { 1 } \right) \widehat { \varphi } _ { k_1 } \left( \zeta _ { 1 } \right)\right),\nonumber \\
& &\cdot \left(\sum _ { k_2 } \widehat { \varphi } _ { k_2 } \left( \eta _ { 2 } \right) \widehat { \psi } _ { k_2 } \left( \zeta _ { 2 } \right) + \sum _ { k_2 } \widehat { \psi } _ { k_2 } \left( \eta _ { 2 } \right) \widehat { \psi } _ { k_2 } \left( \zeta _ { 2 } \right)+ \sum _ { k_2 } \widehat { \psi } _ { k_2 } \left( \eta _ { 2 } \right) \widehat { \varphi } _ { k_2 } \left( \zeta _ { 2 } \right) \right), \label{detazeta}
\end{eqnarray}
where $\widehat \varphi_{k_i} = \sum_{ l_i \leq k_i-100}   \widehat { \psi } _ { l_i } $, and $\widehat \psi_{k_i}$ may actually represent a function  $ \sum_{k_i-100\leq l_i\leq k_i+100}   \widehat { \psi } _ { l_i }$ (we don't distinguish them since their supports are both away from $0$) for $i=1,2$. Then $g\cdot h$ can be written as a summation of the terms like, for example,
\begin{equation}
\label{dgh}
g\cdot h= \sum _ { k_1 , k_2 } \left( g * \left( \psi _ { k_1 } \otimes \psi _ { k_2 } \right) \right) \cdot \left( h * \left( \varphi _ { k_1 } \otimes \varphi _ { k_2 } \right) \right).
\end{equation}
Moreover, it can actually be rewritten as  the following bi-parameter paraproduct
$$\Pi (g,h)=\sum _ { k_1 , k_2 } \left( \big(\left( g * \left( \psi _ { k_1 } \otimes \psi _ { k_2 } \right) \right) \cdot \left( h * \left( \varphi _ { k_1 } \otimes \varphi _ { k_2 } \right) \right)\big) * \widetilde { \psi } _ { k_1 } \otimes \tilde { \psi } _ { k_2 }\right),$$
where $ \psi _ { k_1 } \otimes \psi _ { k_2 }=\psi _ { k_1 } (x_1)\psi _ { k_2 } (x_2)$, and $\widehat {\widetilde { \psi }} _ { k_i }$ is an inserted Schwartz function whose support is away from $0$ and satisfying
$$\widehat {\tilde { \psi }} _ { k_i }=1 \quad \textit{on } \quad \spt \widehat \psi _ { k_i } +\spt \widehat \varphi_{k_i}, \quad i=1,2.$$
Note that in section \ref{sectionlb}, we simply use $\psi_k$ instead of $\widetilde \psi_k$, since they are of the same type, i.e, they are supported on $\{u: c2^{k-1}\leq |u| \leq c 2^{k+1}\}$ for appropriate constants $c$, which are away from $0$.

Now  the differentiation $D^{\beta_1}_1D^{\beta_2}_2 (g,h)$ can be written as a summation of the terms like $D^{\beta_1}_1D^{\beta_2}_2 \Pi(g,h)$, which can be written as

\begin{eqnarray}
  & &D^{\beta_1}_1 D^{\beta_2}_2 \Pi(g,h) \nonumber\\
  &=& \sum _ { {k_1} , {k_2} } \left( \big( \left( g * \left( \psi _ { {k_1} } \otimes \psi _ { {k_2} } \right) \right) \cdot \left( h * \left( \varphi _ { {k_1} } \otimes \varphi _ { {k_2} } \right) \right) \big)* D^{\beta_1}_1 D^{\beta_2}_2 \left(\widetilde { \psi } _ { {k_1} } \otimes \widetilde { \psi } _ { {k_2} }\right)\right)  \nonumber\\
  &=&\sum _ { {k_1} , {k_2} } \left( \big(\left( g * \left( \psi _ { {k_1} } \otimes \psi _ { {k_2} } \right) \right) \cdot \left( h * \left( \varphi _ { {k_1} } \otimes \varphi _ { {k_2} } \right) \right) \big)* 2^{{k_1}\beta_1} 2^{k_2 \beta_2} \left(\widetilde { \psi' } _ { {k_1} } \otimes \widetilde { \psi' } _ { {k_2} }\right)\right)  \nonumber\\
    &=&\sum _ { {k_1} , {k_2} } \left(\big( \left( g * \left( 2^{{k_1} \beta_1} \psi _ { {k_1} }\otimes   2^{k_2 \beta_2} \psi _ { {k_2} } \right) \right) \cdot \left( h * \left(  \varphi _ { {k_1} } \otimes \varphi _ { {k_2} } \right) \right)\big) *\left(\widetilde { \psi' } _ { {k_1} } \otimes \widetilde { \psi' } _ { {k_2} }\right)\right) \nonumber\\
        &=&\sum _ { {k_1} , {k_2} } \left( \left ( \left( g * \left(D^{\beta_1}_1 \psi'' _ { {k_1} } \otimes  D^{\beta_2}_2 \psi'' _ { {k_2} } \right) \right) \cdot \left( h * \left(  \varphi _ { {k_1} }  \otimes \varphi _ { {k_2} } \right) \right) \right)*\left(\widetilde { \psi' } _ { {k_1} } \otimes \widetilde { \psi' } _ { {k_2} }\right)\right)  \nonumber\\
        &=&\sum _ { {k_1} , {k_2} } \left(\left( \left(D^{\beta_1}_1 D^{\beta_2}_2 g * \left( \psi'' _ { {k_1} } \otimes   \psi'' _ { {k_2} } \right) \right) \cdot \left(  h * \left(  \varphi _ { {k_1} } \otimes \varphi _ { {k_2} } \right) \right) \right)*\left(\widetilde { \psi' } _ { {k_1} } \otimes \widetilde { \psi' } _ { {k_2} }\right)\right)  \nonumber\\
        &:=& \Pi(  D^{\beta_1}_1 D^{\beta_2}_2 g,  h), \label{lbgh}
\end{eqnarray}
where $\widehat { \widetilde { \psi' } } _ { k_i } ( u ) : = \widehat {\widetilde { \psi } } _ { k_i } ( u ) \left| \frac { u } { 2 ^ { k_i } } \right| ^ { \beta_i }$, $\widehat { { { \psi'' } } } _ { k_i } ( u ) : = \widehat { \psi } _ { k_i } ( u ) \left( \frac { 2 ^ { k_i } } { | u | } \right) ^ { \beta_i }$ for $i=1,2$.

Note that the idea is to ``move" the differential operator to appropriate functions. In the above expressions we finally apply the differential operators to $g$ because the associated convolution has both $\psi$ type functions involved. That allows us to  multiply or divide them by functions $|u|^{\beta_i}$ as we need, i.e, we can always make $\psi''$ smooth.

Using a similar idea, we decompose  $f$ as
$$f =\sum _ { j_1 } \sum_{ j_2} \left( f * \left( \psi _ { j_1 } \otimes \psi _ { j_2 } \right) \right),$$
then $f\cdot g\cdot h$ can be written as a summation of the terms  like, for example,
\begin{eqnarray}
   & &\sum _ { j_1 } \sum_{ j_2} \left( f * \left( \psi _ { j_1 } \otimes \psi _ { j_2 } \right) \right) \sum _ { l_1 \ll k_1  } \sum_{ l_2 \ll k_2}\left( g * \left( \psi _ { k_1 } \otimes \psi _ { k_2 } \right) \right) \cdot \left( h * \left( \psi _ { l_1} \otimes \psi _ { \ell_2 } \right) \right)  \nonumber \\
  &=& \left(\sum _ {j_1\ll k_1} + \sum _ {j_1\simeq k_1}+ \sum _ {j_1\gg k_1}\right) \left(\sum _ {j_2\ll k_2} + \sum _ {j_2\simeq k_2}+ \sum _ {j_2\gg k_2}\right)  \sum _ { l_1 \ll k_1  } \sum_{ l_2 \ll k_2}  \nonumber \\
   & & \quad \cdot \left( f * \left( \psi _ { j_1 } \otimes \psi _ { j_2 } \right) \right) \left( g * \left( \psi _ { k_1 } \otimes \psi _ { k_2 } \right) \right) \cdot \left( h * \left( \psi _ { l_1} \otimes \psi _ { \ell_2 } \right) \right) \label{lbsum}.
\end{eqnarray}
In the above summation, let's take a look at the following part with $\sum_{k_1\ll j_1}\sum_{k_2\ll j_2}$,
\begin{eqnarray}
& &  \int_{\n^6}  \sum _ {l_1\ll k_1\ll j_1} \sum _ {l_2\ll k_2 \ll j_2} \left(\widehat \psi _ { j_1 }(\xi_1)  \widehat \psi _ { j_2 }(\xi_2)   \widehat \psi _ { k_1 }(\eta_1)  \widehat \psi _ { k_2 }(\eta_2)  \widehat \psi _ { l_1 }(\zeta_1)  \widehat \psi _ { l_2 }(\zeta_2)\right) \nonumber \\
 & & \quad   \cdot \hat f(\xi_1,\xi_2) \hat g(\eta_1,\eta_2) \hat h(\zeta_1,\zeta_2) e^{2\pi i (\xi+\eta+\zeta)x} d\xi d\eta d\zeta \nonumber\\
 &:=& \int_{\n^6} m(\xi,\eta,\zeta) \hat f(\xi_1,\xi_2) \hat g(\eta_1,\eta_2) \hat h(\zeta_1,\zeta_2) e^{2\pi i (\xi+\eta+\zeta)x} d\xi d\eta d\zeta, \label{lbop}
\end{eqnarray}
where the symbol can be rewritten as
\begin{eqnarray}
& &m(\xi,\eta,\zeta) \nonumber \\
 &=&\left( \sum _ { j_1} \sum _ { j_2} \widehat \psi _ { j_1 }(\xi_1)  \widehat \psi _ { j_2 }(\xi_2) \widehat \varphi _ { j_1 }(\eta_1+\zeta_1)  \widehat \varphi _ { j_2 }(\eta_2+\zeta_2) \widehat \psi _ { j_1 }(\xi_1+\eta_1+\zeta_1)  \widehat \psi _ { j_2 }(\xi_2+\eta_2+\zeta_2) \right)  \nonumber\\
  & & \cdot \left( \sum _ { k_1\ll j_1} \sum _ { k_2 \ll j_2}  \widehat \psi _ { k_1 }(\eta_1)  \widehat \psi _ { k_2 }(\eta_2)  \widehat \varphi _ { k_1 }(\zeta_1)  \widehat \varphi _ { k_2 }(\zeta_2) \widehat \psi _ { k_1 }(\eta_1+\zeta_1)  \widehat \psi _ { k_2 }(\eta_2+\zeta_2) \right) , \label{lbsb2}
\end{eqnarray}
where  some appropriate $\varphi$ type functions and $\psi$ type functions are inserted as before, based on the supports of the functions. With the above, \eqref{lbop} becomes
\begin{eqnarray}
 \sum _ { j _ { 1 }, j_2 } \big( \left( f * \left(\psi _ { j _ { 1 } } \otimes  \psi _ { j _ { 2 } }\right)\right) \left( \Pi_1(g,h) * \left(\varphi _ { j _ { 1 } }\otimes \varphi _ { j _ { 2 } } \right) \right)\big ) *  \left(\psi _ { j _ { 1 } }\otimes \psi _ { j _ { 2 } }\right),
 \label{lbflag}
\end{eqnarray}
where
$$ \Pi_1(g,h)=
  \sum _ { k _ { 1 } \ll j _ { 1 } } \sum _ { k _ { 2 } \ll j _ { 2 } } \left(\left(   g * \left(\psi _ { k _ { 1 } } \otimes \psi _ { k _ { 2} }\right) \right)  \left( h * \left( \varphi _ { k _ { 1 } }\otimes \varphi _ { k _ { 2 }  }\right) \right)\right) * \left(\psi _ { k _ { 1 }} \otimes  \psi _ { k _ { 2 }} \right).$$
Recall by \eqref{lbgh}, $ D_1^{\beta_1}D_2^{\beta_2}\Pi_1(g,h)$ can be written as a summation of the terms like, for example,
$$ \sum _ { k _ { 1 } \ll j _ { 1 } } \sum _ { k _ { 2 } \ll j _ { 2 } }  \left( \left(D^{\beta_1}_1 D^{\beta_2}_2 g * \left( \psi'' _ { k_1 } \otimes   \psi'' _ { k_2 } \right) \right) \cdot \left(  h * \left(  \varphi _ { k_1 } \otimes \varphi _ { k_2 }\right) \right) \right)*\left( { \psi' } _ { k_1 } \otimes { \psi' } _ { k_2 }\right)$$

Then when we apply the differential operator $D^{\alpha_1}_1 D^{\alpha_2}_2$ to $f \cdot D_1^{\beta_1} D_2^{\beta_2}\Pi_1(g,h)$,
\begin{eqnarray*}
& &D^{\alpha_1}_1 D^{\alpha_2}_2 \left( \sum _ { j _ { 1 }, j_2 } \big( \left( f * \left(\psi _ { j _ { 1 } } \otimes  \psi _ { j _ { 2 } }\right)\right) \left( D_1^{\beta_1} D_2^{\beta_2} \Pi_1(g,h) * \left(\varphi _ { j _ { 1 } }\otimes \varphi _ { j _ { 2 } } \right) \right)\big ) *  \left(\psi _ { j _ { 1 } }\otimes \psi _ { j _ { 2 } }\right)\right) \\
&=&   \sum _ { j _ { 1 }, j_2 } \big( \left( f * \left(\psi _ { j _ { 1 } } \otimes  \psi _ { j _ { 2 } }\right)\right) \left( D_1^{\beta_1} D_2^{\beta_2} \Pi_1(g,h) * \left(\varphi _ { j _ { 1 } }\otimes \varphi _ { j _ { 2 } } \right) \right)\big ) *  D^{\alpha_1}_1 D^{\alpha_2}_2 \left(\psi _ { j _ { 1 } }\otimes \psi _ { j _ { 2 } }\right)\\
&=&   \sum _ { j _ { 1 }, j_2 } \big( \left( f * \left(\psi _ { j _ { 1 } } \otimes  \psi _ { j _ { 2 } }\right)\right) \left( D_1^{\beta_1} D_2^{\beta_2} \Pi_1(g,h) * \left(\varphi _ { j _ { 1 } }\otimes \varphi _ { j _ { 2 } } \right) \right)\big ) * 2^{j_1\alpha_1}  2^{j_2\alpha_2}\left(\psi' _ { j _ { 1 } }\otimes \psi' _ { j _ { 2 } }\right)\\
&=&   \sum _ { j _ { 1 }, j_2 } \big( \left( f * \left(2^{j_1\alpha_1}  \psi _ { j _ { 1 } } \otimes 2^{j_2\alpha_2} \psi _ { j _ { 2 } }\right)\right) \left( D_1^{\beta_1} D_2^{\beta_2} \Pi_1(g,h) * \left(\varphi _ { j _ { 1 } }\otimes \varphi _ { j _ { 2 } } \right) \right)\big ) * \left(\psi' _ { j _ { 1 } }\otimes \psi' _ { j _ { 2 } }\right)\\
&=&   \sum _ { j _ { 1 }, j_2 } \big( \left( f * \left(D^{\alpha_1}_1  \psi'' _ { j _ { 1 } } \otimes D^{\alpha_2}_2 \psi'' _ { j _ { 2 } }\right)\right) \left( D_1^{\beta_1} D_2^{\beta_2} \Pi_1(g,h) * \left(\varphi _ { j _ { 1 } }\otimes \varphi _ { j _ { 2 } } \right) \right)\big ) * \left(\psi' _ { j _ { 1 } }\otimes \psi' _ { j _ { 2 } }\right)\\
&=&   \sum _ { j _ { 1 }, j_2 } \big( \left( D^{\alpha_1}_1 D^{\alpha_2}_2  f *\left(    \psi'' _ { j _ { 1 } } \otimes \psi'' _ { j _ { 2 } }\right)\right) \left( D_1^{\beta_1} D_2^{\beta_2} \Pi_1(g,h) * \left(\varphi _ { j _ { 1 } }\otimes \varphi _ { j _ { 2 } } \right) \right)\big ) * \left(\psi' _ { j _ { 1 } }\otimes \psi' _ { j _ { 2 } }\right)\\
&:=& \sum _ { j _ { 1 }, j_2 } \big( \left(  D^{\alpha_1}_1 D^{\alpha_2}_2f * \left(\psi'' _ { j _ { 1 } } \otimes  \psi'' _ { j _ { 2 } }\right)\right) \left( D_1^{\beta_1} D_2^{\beta_2} \Pi_1(g,h) * \left(\varphi _ { j _ { 1 } }\otimes \varphi _ { j _ { 2 } } \right) \right)\big ) *  \left(\psi' _ { j _ { 1 } }\otimes \psi' _ { j _ { 2 } }\right).
\end{eqnarray*}
Based on the above form, and by removing some inserted and reinserting appropriate functions as before, we can write the associated symbol as
\begin{eqnarray*}
  & &\left( \sum _ { j_1} \sum _ { j_2} \widehat \psi'' _ { j_1 }(\xi_1)  \widehat \psi'' _ { j_2 }(\xi_2)  \widehat \psi' _ { j_1 }(\xi_1+\eta_1+\zeta_1)  \widehat \psi' _ { j_2 }(\xi_2+\eta_2+\zeta_2) \right)  \nonumber\\
  & &\quad  \cdot \left( \sum _ { k_1\ll j_1} \sum _ { k_2 \ll j_2}  \widehat \psi'' _ { k_1 }(\eta_1)  \widehat \psi'' _ { k_2 }(\eta_2)  \widehat \varphi _ { k_1 }(\zeta_1)  \widehat \varphi _ { k_2 }(\zeta_2) \widehat \psi' _ { k_1 }(\eta_1+\zeta_1)  \widehat \psi' _ { k_2 }(\eta_2+\zeta_2) \right)\\
  &=&  \left( \sum _ { j_1} \sum _ { j_2} \widehat \psi'' _ { j_1 }(\xi_1)  \widehat \psi'' _ { j_2 }(\xi_2) \widehat \varphi _ { j_1 }(\eta_1)  \widehat \varphi _ { j_2 }(\eta_2) \widehat \psi' _ { j_1 }(\xi_1+\eta_1+\zeta_1)  \widehat \psi' _ { j_2 }(\xi_2+\eta_2+\zeta_2) \right)  \nonumber\\
  & &\quad  \cdot \left( \sum _ { k_1\ll j_1} \sum _ { k_2 \ll j_2}  \widehat \psi'' _ { k_1 }(\eta_1)  \widehat \psi'' _ { k_2 }(\eta_2)  \widehat \varphi _ { k_1 }(\zeta_1)  \widehat \varphi _ { k_2 }(\zeta_2) \right) \\
  &=&  \left( \sum _ { j_1} \sum _ { j_2} \widehat \psi'' _ { j_1 }(\xi_1)  \widehat \psi'' _ { j_2 }(\xi_2) \widehat \varphi _ { j_1 }(\eta_1)  \widehat \varphi _ { j_2 }(\eta_2) \widehat \psi' _ { j_1 }(\xi_1+\eta_1+\zeta_1)  \widehat \psi' _ { j_2 }(\xi_2+\eta_2+\zeta_2) \right)  \nonumber\\
  & &\quad  \cdot \left( \sum _ { k_1} \sum _ { k_2 }  \widehat \psi'' _ { k_1 }(\eta_1)  \widehat \psi'' _ { k_2 }(\eta_2)  \widehat \varphi _ { k_1 }(\zeta_1)  \widehat \varphi _ { k_2 }(\zeta_2) \right)\\
  &=&m_1(\xi,\eta,\zeta) m_2(\eta,\zeta),
\end{eqnarray*}
where in the last estimate we are able to ignore the restriction $k_i\ll j_i$ when taking the summation, because $\varphi_{j_i}$ is a properly chosen function such that $\widehat \varphi_{j_i}(\eta_i)\widehat \psi''_{k_i}(\eta_i)=0$ if the restriction is not satisfied ($i=1,2$).

Note that $m_1$ and $m_2$ belong to the classical symbols $\mathcal{BM}(\n^6)$ and $\mathcal{BM}(\n^4)$ respectively. Then the bound $$\|D^{\alpha_1}_1 D^{\alpha_2}_2 f\|_{L^{p_1}} \cdot \|D^{\beta_1}_1 D^{\beta_2}_2 g\|_{L^{q_1}}\cdot  \| h \|_{L^{s_1}}$$  follows from the boundedness in Theorem \ref{bithm}. In fact,  the other terms in the Leibniz estimate \eqref{lb16} can be obtained in the similar way. Taking the decomposition $g\cdot h$ for example, recall that the above argument is based on one of the terms in the decomposition of $g\cdot h$, i.e. \eqref{dgh}. However, the are actually $9$ terms in the decomposition, as indicated by \eqref{detazeta}. The similar thing happens after $f$ is introduced. In short, among the rest of these terms, parts of them are covered by Theorem \ref{bithm}. The other parts can take different forms, based on the ``positions" of the $\psi$ functions $\varphi$ functions,  and the $L^r$ estimate of them gives the other $15$ pieces in \eqref{lb16}. An example has been given earlier.

However, the other part of those remaining terms cannot be treated as the operator $T_{m_1,m_2}$. As we mentioned in Section \ref{sectionlb}, these terms actually correspond to the multiplier $T_{m_3,m_4}$. More precisely, if we check following term in the decomposition  of $f\cdot g\cdot h$,

$$\sum _ { l_1\ll k_1 \ll j_1 } \sum_{ j_2 \ll k_2 \ll l_2} \left( f * \left( \psi _ { j_1 } \otimes \psi _ { j_2 } \right) \right)\left( g * \left( \psi _ { k_1 } \otimes \psi _ { k_2 } \right) \right)  \left( h * \left( \psi _ { l_1} \otimes \psi _ { \ell_2 } \right) \right),$$
we can see it is actually a Fourier multiplier:

\begin{eqnarray}
& &  \int_{\n^6}  \sum _ {l_1\ll k_1\ll j_1} \sum _ {l_2\gg k_2 \gg j_2} \left(\widehat \psi _ { j_1 }(\xi_1)  \widehat \psi _ { j_2 }(\xi_2)   \widehat \psi _ { k_1 }(\eta_1)  \widehat \psi _ { k_2 }(\eta_2)  \widehat \psi _ { l_1 }(\zeta_1)  \widehat \psi _ { l_2 }(\zeta_2)\right) \nonumber \\
 & & \quad   \cdot \hat f(\xi_1,\xi_2) \hat g(\eta_1,\eta_2) \hat h(\zeta_1,\zeta_2) e^{2\pi i (\xi+\eta+\zeta)x} d\xi d\eta d\zeta \nonumber\\
 &:=& \int_{\n^6} m'(\xi,\eta,\zeta) \hat f(\xi_1,\xi_2) \hat g(\eta_1,\eta_2) \hat h(\zeta_1,\zeta_2) e^{2\pi i (\xi+\eta+\zeta)x} d\xi d\eta d\zeta. \label{lbop2}
\end{eqnarray}
Using the trick of inserting terms as before, we can write

\begin{eqnarray*}
& &m'(\xi,\eta,\zeta) \nonumber \\
 &=&\left( \sum _ { j_1} \sum _ { k_2} \widehat \psi _ { j_1 }(\xi_1)  \widehat \varphi _ { k_2 }(\xi_2) \widehat \varphi _ { j_1 }(\eta_1+\zeta_1)  \widehat \psi _ { k_2 }(\eta_2) \widehat \psi _ { j_1 }(\xi_1+\eta_1+\zeta_1)  \widehat \psi _ { k_2 }(\xi_2+\eta_2) \right)  \nonumber\\
  & & \cdot \left( \sum _ { k_1\ll j_1} \sum _ { l_2 \gg k_2}  \widehat \psi _ { k_1 }(\eta_1)  \widehat \varphi _ { l_2 }(\xi_2+\eta_2)  \widehat \varphi _ { k_1 }(\zeta_1)  \widehat \psi _ { l_2 }(\zeta_2) \widehat \psi _ { k_1 }(\eta_1+\zeta_1)  \widehat \psi _ { l_2 }(\xi_2+\eta_2+\zeta_2) \right) \\
 &=&\left( \sum _ { j_1} \sum _ { l_2} \widehat \psi _ { j_1 }(\xi_1)  \widehat \varphi _ { j_1 }(\eta_1+\zeta_1)  \widehat \psi _ { j_1 }(\xi_1+\eta_1+\zeta_1) \widehat \varphi _ { l_2 }(\xi_2+\eta_2) \widehat \psi _ { l_2 }(\zeta_2)  \widehat \psi _ { l_2 }(\xi_2+\eta_2+\zeta_2) \right)  \nonumber\\
  & & \cdot \left( \sum _ { k_1\ll j_1} \sum _ { k_2 \ll l_2}  \widehat \psi _ { k_1 }(\eta_1)   \widehat \varphi _ { k_1 }(\zeta_1)  \widehat \psi _ { k_1 }(\eta_1+\zeta_1) \widehat \varphi _ { k_2 }(\xi_2) \widehat \psi _ { k_2 }(\eta_2) \widehat \psi _ { k_2 }(\xi_2+\eta_2) \right).
\end{eqnarray*}

Note that in this case $m'$ is essentially a symbol
\begin{eqnarray*}
m^\prime (\xi,\eta,\zeta)=\left(m^{\prime,1}(\xi_1,\eta_1,\zeta_1)m^{\prime,2}(\eta_1,\zeta_1)\right)\left( m^{\prime,3}(\xi_2,\eta_2,\zeta_2) m^{\prime,4}(\xi_2,\eta_2)\right),
\end{eqnarray*}
with $m^{\prime,1}, m^{\prime,3} \in \mathcal{M}(\n^3)$ and $m^{\prime,2}, m^{\prime,4}\in \mathcal{M}(\n^2)$.  Strictly speaking, $m'$ is not exactly a symbol having the form $m_4(\eta,\zeta)m_3(\xi,\eta)$, with $m_3,m_4$ in $\mathcal{BM}(\n^4)$.  However, it is a fact that these two symbols share the same difficulty in obtaining
their H\"older-type estimates. Thus, without loss of generality, we treat them as if they were the same.  As before, after we apply those differential operators, since the ``types" of those $\psi$ and $\varphi$ functions do not change, the form of the associated symbol won't change, either. This means that the Leibniz rule for these terms will follow from the H\"older type $L^r$ estimate of $T_{m_3,m_4}$. Another difficulty in this case is that an analogue of \eqref{lbflag} will have a more complicated form, another challenge in obtaining the desired estimate.

\section{}
\label{ap2}

  In this section, we obtain the mixed norm estimates when $1<p,p_2,p_3,q_2,q_3<\infty$ in \cite{Zhai} via looking at the bi-parameter multilinear operators \eqref{bicofop} and the reduced operator \eqref{assumption}.  We consider \eqref{assumption} under the tensor product setting first.
  \begin{proposition}
  \label{tensor32}
    Let $g(x)=g_1(x_1)\otimes g_2(x_2) $, $h(x)=h_1(x_1)\otimes h_2(x_2) $, and ${1\over p}+ {1\over p_2}+{1\over p_3}={1\over p}+{1\over q_2}+{1\over q_3}={1\over r}$. Then \eqref{assumption} maps $L^p \times L^{p_2}_{x_1}(L^{q_2}_{x_2})\times L^{p_3}_{x_1}(L^{q_3}_{x_2}) \to L^r$, with $0<r<\infty$ and $1<p,p_2,p_3,q_2,q_3<\infty $.
  \end{proposition}

  \begin{proof}

    Using the tensor products, \eqref{assumption} becomes

    \begin{eqnarray*}
     & & \int_{\n^{6}}m'(\xi_1,\eta_1,\zeta_1)m''(\eta_2,\zeta_2)\hat f(\xi) \hat g_1(\eta_1) \hat g_2(\eta_2) \hat h_1(\zeta_1)\hat h_2(\zeta_2)e^{2\pi i(\xi+\eta+\zeta)\cdot x}d\xi d\eta d\zeta \\
     &=& \left(\int _{\n^{4}} m'(\xi_1,\eta_1,\zeta_1)\hat f(\xi) \hat g_1(\eta_1) \hat h_1(\zeta_1)e^{2\pi i(\xi_1+\eta_1+\zeta_1)\cdot x_1} e^{2\pi i\xi_2\cdot x_2}  d\xi d\eta_1 d\zeta_1 \right) \\
      & & \quad \cdot \left(\int _{\n^{2}}m''(\eta_2,\zeta_2) \hat g_2(\eta_2) \hat h_2(\zeta_2)  e^{2\pi i(\eta_2+\zeta_2)\cdot x_2} d\eta_2 d\zeta_2 \right) \\
      &=& \left(\int _{\n^{3}} m'(\xi_1,\eta_1,\zeta_1) \mathscr{F}_1 f(\xi_1,x_2) \hat g_1(\eta_1) \hat h_1(\zeta_1)e^{2\pi i(\xi_1+\eta_1+\zeta_1)\cdot x_1} e^{2\pi i\xi_2\cdot x_2}  d\xi d\eta_1 d\zeta_1 \right) \\
      & & \quad \cdot \left(\int _{\n^{2}}m''(\eta_2,\zeta_2) \hat g_2(\eta_2) \hat h_2(\zeta_2)  e^{2\pi i(\eta_2+\zeta_2)\cdot x_2} d\eta_2 d\zeta_2 \right) \\
      &=:&   T_1\left(f(\cdot,x_2),g_1(\cdot),h_1(\cdot)\right)(x_1) \cdot  T_2\left(g_2,h_2\right)(x_2),
    \end{eqnarray*}
    where $\mathscr{F}_1$ represents the Fourier transform with respect to the first variable. Then its $L^r$ norm can be estimated as

    \begin{eqnarray*}
     & &\| T_1\left(f(\cdot,x_2),g_1(\cdot),h_1(\cdot)\right)(x_1) \cdot  T_2\left(g_2,h_2\right)(x_2)\|_{L^r}^r\\
     &=& \int \left| T_1\left(f(\cdot,x_2),g_1(\cdot),h_1(\cdot)\right)(x_1) \right|^r  \left|T_2\left(g_2,h_2\right)(x_2)\right|^r dx_1 dx_2 \\
      &=& \int \left(\int \left| T_1\left(f(\cdot,x_2),g_1(\cdot),h_1(\cdot)\right)(x_1) \right|^r  dx_1 \right)  \left|T_2\left(g_2,h_2\right)(x_2)\right|^r dx_2\\
      &\lesssim& \int \|f(\cdot,x_2)\|_{L^p_{x_1}}^r \|g_1\|_{L^{p_2}_{x_1}}^r \|h_1\|_{L^{p_3}_{x_1}}^r  \left|T_2\left(g_2,h_2\right)(x_2)\right|^r dx_2 \\
      &=&  \int \|f(\cdot,x_2)\|_{L^p_{x_1}}^r  \left|T_2\left(g_2,h_2\right)(x_2)\right|^r dx_2 \cdot \|g_1\|_{L^{p_2}_{x_1}}^r\|h_1\|_{L^{p_3}_{x_1}}^r \\
      &\lesssim& \left(\int \|f(\cdot,x_2)\|_{L^p_{x_1}}^p dx_2\right)^{r\over p}  \left( \int \left|T_2\left(g_2,h_2\right)(x_2)\right|^s dx_2 \right)^{r\over s}\cdot \|g_1\|_{L^{p_2}_{x_1}}^r\|h_1\|_{L^{p_3}_{x_1}}^r \\
      &\lesssim& \|f\|_{L^p}^r \|g_2\|_{L^{q_2}_{x_2}}^r\|h_2\|_{L^{q_3}_{x_2}}^r \|g_1\|_{L^{p_2}_{x_1}}^r\|h_1\|_{L^{p_3}_{x_1}}^r,
    \end{eqnarray*}
    where $1<p,p_2,p_3,q_2,q_3<\infty$, ${1\over p_2}+{1\over p_3}={1\over q_2}+{1\over q_3}={1\over s}$, and we just use the H\"older's inequality, the $L^r$ boundedness of the classical one-parameter trilinear Fourier multiplier $T_1$ and bilinear multiplier $T_2$, i.e.,
    $$\|T_1(f_1,f_2,f_3)\|_{L^r}\lesssim \|f_1\|_{L^p}\|g_1\|_{L^{p_2}}\|h_1\|_{L^{p_3}}, \quad \|T_2(g_2,h_2)\|_{L^s}\lesssim \|g_2\|_{L^{q_2}}\|h_2\|_{L^{q_3}}. $$

  \end{proof}

 Then one can get the same mixed  $L^r$ estimates for \eqref{biop34} under the same tensor product assumption.
 \begin{proposition}
 Let $g(x)=g_1(x_1)\otimes g_2(x_2) $, $h(x)=h_1(x_1)\otimes h_2(x_2) $, and ${1\over p}+ {1\over p_2}+{1\over p_3}={1\over p}+{1\over q_2}+{1\over q_3}={1\over r}$. Then \eqref{biop34} maps $L^p \times L^{p_2}_{x_1}(L^{q_2}_{x_2})\times L^{p_3}_{x_1}(L^{q_3}_{x_2}) \to L^r$, with $1<p,p_2,p_3,q_2,q_3< \infty$ and $0<r<\infty$.
 \end{proposition}

 \begin{proof}

   Recall in our earlier reduction in Section \ref{bithmpf},  the study of the multipliers \eqref{biop34} can be reduced to essentially two types of multipliers, i.e., the classical bi-parameter trilinear multipliers (under the tensor product setting)
   \begin{equation}
   \label{opT3}
   T_3(f,g_1\otimes g_2,h_1\otimes h_2)= \int_{\n^{6}}m(\xi,\eta,\zeta)\hat f(\xi) \hat g_1(\eta_1) \hat g_2(\eta_2) \hat h_1(\zeta_1)\hat h_2(\zeta_2)e^{2\pi i(\xi+\eta+\zeta)\cdot x}d\xi d\eta d\zeta,
   \end{equation}
    and the ones like \eqref{assumption}, i.e., \eqref{assumption} or
    \begin{equation}
    \label{assumption2}\int_{\n^{6}}m'(\xi_2,\eta_2,\zeta_2)m''(\eta_1,\zeta_1)\hat f(\xi) \hat g_1(\eta_1) \hat g_2(\eta_2) \hat h_1(\zeta_1)\hat h_2(\zeta_2)e^{2\pi i(\xi+\eta+\zeta)\cdot x}d\xi d\eta d\zeta.
    \end{equation}
   Note that the same argument for \eqref{assumption} shows \eqref{assumption2} maps  $L^p \times L^{p_2}_{x_1}(L^{q_2}_{x_2})\times L^{p_3}_{x_1}(L^{q_3}_{x_2}) \to L^r$, with $1<p,p_2,p_3,q_2,q_3<\infty$ and $0<r<\infty$.

   Then the proof can be completed as long as we can show the mixed norm estimate for $T_3$.
   \end{proof}

\begin{proposition}
\label{ppt3}
  The bi-parameter trilinear multiplier $T_3$ maps $L^{p} \times L^{p_2}_{x_1}(L^{q_2}_{x_2}) \times L^{p_3}_{x_1}(L^{q_3}_{x_2})  \to L^r$, for $1<p,q_1,p_2,q_2,p_3,q_3<\infty$ and $0<r<\infty$, with $1/r=1/p+1/p_2+1/p_3=1/p+ 1/q_2+1/q_3$, if one assumes $g(x_1,x_2)=g(x_1)\otimes g(x_2)$ and $h(x_1,x_2)=h(x_1)\otimes h(x_2)$.
\end{proposition}

For a quick and direct way to see the proof, one can refer to the work \cite{chen2014hormander}, where the classical H\"ormander type $L^r$ estimate was obtained for the multilinear and multi-parameter  multipliers with limited smoothness. The main idea was to control $m$ by using the Sobolev norm, and the main technique was to take care of the functions, i.e., $f,g,h,\dots$, by strong maximal functions and H\"older's inequalities. One can see, such methods would go through if one assumes the tensor product form for these functions.

\begin{definition}
  For $f\in \mathcal{S}'(\mathbb{R}^6)$, define the bi-parameter Sobolev space
  $$\|f\|_{H^{s_1,s_2}}:= \|(I-\Delta)^{s_1/2,s_2/2}f\|_{L^2}<\infty,$$
  where
  \begin{gather*}
  (I-\Delta)^{s_1/2,s_2/2}f= \\
  \mathcal{F}^{-1}[(1+|\xi_1|^2+|\eta_1|^2+|\zeta_1|^2)^{s_1/2}(1+|\xi_2|^2+|\eta_2|^2+|\zeta_2|^2)^{s_2/2} \hat f (\xi_1,\xi_2,\eta_1,\eta_2,\zeta_1,\zeta_2)]
  \end{gather*} for $\xi_1,\xi_2,\eta_1,\eta_2,\zeta_1,\zeta_2 \in \n$.
\end{definition}

Let $m_{j,k}(\xi_1,\xi_2,\eta_1,\eta_2,\zeta_1,\zeta_2)=m(2^j\xi_1,2^k\xi_2,2^j\eta_1,2^k\eta_2,2^j \zeta_1, 2^k \zeta_2)\psi_1(\xi_1,\eta_1,\zeta_1) \psi_2(\xi_2,\eta_2,\zeta_2)$,
where smooth cutoff functions $\psi_1,\psi_2$ satisfy
$$\spt \psi_1, \, \spt \psi_2 \subset \{u\in\n^3: 1/2\leq |u|\leq 2\}.$$

Note that if $m\in \mathcal{BM}(\n^6)$ satisfying
$$ |\partial_{\xi_1,\xi_2}^{\alpha_1,\alpha_2}\partial_{\eta_1,\eta_2}^{\beta_1,\beta_2}\partial_{\zeta_1,\zeta_2}^{\gamma_1,\gamma_2}  m(\xi,\eta,\zeta)|\lesssim \prod_{i=1}^2{1\over (|\xi_i|+|\eta_i|+|\zeta_i|)^{\alpha_i+\beta_i+\gamma_i}},$$
for $\alpha_i+\beta_i+\zeta_i \leq N_i $  ($i=1,2$), then there holds
 $$\sup_{j,k} \|m_{j,k}\|_{H^{N_1,N_2}}<\infty.$$

We also give a lemma which will be useful later.
\begin{lemma}[\cite{chen2014hormander}]
\label{lmstrongmax}
  For any $\epsilon_1,\epsilon_2>0$, there exists a constant $C>0$ such that
  $$\sup_{r_1,r_2>0}\left(r_1^n r_2^n \int_{\mathbb{R}^2} \frac{f(u,v)}{(1+r_1|x_1-u|)^{1+\epsilon_1}(1+r_2|x_2-v|)^{1+\epsilon_2}} dudv\right)\leq C M_sf(x_1,x_2),$$
  where $M_s$ is the strong maximal operator.
\end{lemma}
\medskip
Now we give a proof for Proposition \ref{ppt3}.
\begin{proof}
We first repeat some arguments in \cite{chen2014hormander}.  Consider a decomposition of the symbol, according to the size of the support of each variable. More precisely, take $\phi_1 \in C^\infty$ on $[0,\infty)$ satisfying
$$\phi_1(t)=1 \ \text{on}\ [0,1/12], \quad \spt \phi_1 \subset [0,1/6],$$
and also let $\phi_2(t)=1-\phi_1(t)$. Note that we can write
\begin{eqnarray*}
1(\xi_1,\eta_1,\zeta_1)&=& \left[\phi_1(|\xi_1|/|(\xi_1,\eta_1,\zeta_1)|)+\phi_2(|\xi_1|/|(\xi_1,\eta_1,\zeta_1)|)\right] \\
& & \cdot \left[\phi_1(|\eta_1|/|(\xi_1,\eta_1,\zeta_1)|)+\phi_2(|\eta_1|/|(\xi_1,\eta_1,\zeta_1)|)\right] \\
 & & \cdot \left[\phi_1(|\zeta_1|/|(\xi_1,\eta_1,\zeta_1)|)+\phi_2(|\zeta_1|/|(\xi_1,\eta_1,\zeta_1)|)\right].
\end{eqnarray*}
Obviously a similar expression holds for $1(\xi_2,\eta_2,\zeta_2)$. Correspondingly the symbol $m$ can be decomposed as
\begin{eqnarray*}
 & & m(\xi_1,\xi_2,\eta_1,\eta_2,\zeta_1,\zeta_2) \\
  &=&m(\xi_1,\xi_2,\eta_1,\eta_2,\zeta_1,\zeta_2)\cdot 1(\xi_1,\eta_1,\zeta_1)\cdot 1(\xi_2,\eta_2,\zeta_2)
\end{eqnarray*}
We group the situations after the decomposition as follows.

 For $1(\xi_1,\eta_1,\zeta_1)$, we consider two groups.

  Group $I_1$: The largest component is much larger than the other two, i.e.,
\begin{itemize}
  \item[(a)] $|\xi_1|\gg |\eta_1|+|\zeta_1|.$
  \item[(b)] $|\eta_1|\gg |\xi_1|+|\zeta_1|.$
  \item[(c)] $|\zeta_1|\gg |\xi_1|+|\eta_1|.$
\end{itemize}

 Group $II_1$: The largest component is comparable to the second largest, i.e.,
 \begin{itemize}
    \item[(a)] $|\xi_1|\approx |\eta_1| \approx |\xi_1|+|\eta_1|+|\zeta_1|$.
     \item[(b)] $|\eta_1|\approx|\zeta_1|\approx |\xi_1|+|\eta_1|+|\zeta_1|$.
      \item[(c)] $|\zeta_1|\approx  |\xi_1|\approx |\xi_1|+|\eta_1|+|\zeta_1|$.
\end{itemize}

Similarly, for $1(\xi_2,\eta_2,\zeta_2)$, we consider two groups.

Group $I_2$: The largest component is much larger than the other two, i.e.,
\begin{itemize}
  \item[(a)] $|\xi_2|\gg |\eta_2|+|\zeta_2|.$
  \item[(b)] $|\eta_2|\gg |\xi_2|+|\zeta_2|.$
  \item[(c)] $|\zeta_2|\gg |\xi_2|+|\eta_2|.$
\end{itemize}

 Group $II_2$: The largest component is comparable to the second largest, i.e.,
 \begin{itemize}
    \item[(a)] $|\xi_2|\approx |\eta_2| \approx |\xi_2|+|\eta_2|+|\zeta_2|$.
     \item[(b)] $|\eta_2|\approx|\zeta_2|\approx |\xi_2|+|\eta_2|+|\zeta_2|$.
      \item[(c)] $|\zeta_2|\approx  |\xi_2|\approx |\xi_2|+|\eta_2|+|\zeta_2|$.
\end{itemize}

According to the symmetry, it suffices to consider
\begin{eqnarray*}
I_1\otimes I_2&:=& \{I_1(a)\otimes I_2(a),\ I_1(a)\otimes I_2(b),\ I_1(b)\otimes I_2(b),\ I_1(b)\otimes I_2(c),\dots \},\\
II_1\otimes II_2&:=&\{II_1(a)\otimes II_2(a),\ II_1(a)\otimes II_2(b), \ II_1(b)\otimes II_2(b),\dots \},\\
I_1\otimes II_2&:=& \{I_1(a)\otimes II_2(a),\ I_1(a)\otimes II_2(b),\ I_1(b)\otimes II_2(a),\ I_1(b)\otimes II_2(b), \ I_1(b)\otimes II_2(c),\dots\}.
\end{eqnarray*}
In the rest of the proof, even though $m$ in different groups would be multiplied by  corresponding cutoffs which were constructed at the beginning of the proof, for our convenience we will still use $m$ as the notation, since those cutoff functions do not actually play an important role in our calculation.

\medskip

We first consider the cases that belong to $I_1\otimes I_2$.  We take $I_1(a)\otimes I_2(b)$ for example. Consider  the Littlewood-Paley characterization,
\begin{eqnarray*}
  \|T_3(f,g,h)\|_{L^r}\lesssim \left\| \left\{\sum_{j,k} |\psi_1(D/2^j)\psi_2(D/2^k)T_3(f,g,h) |^2\right\}^{1\over 2}\right\|_{L^r},
\end{eqnarray*}
where $\psi_1(D/2^j)\psi_2(D/2^k)T_3:= \mathcal{F}^{-1}[\psi_1(u/2^j)\psi_2(v/2^k) \mathcal{F}(T_3)(u,v)]$. For simplicity, we will just denote it by $\psi_j(D)\psi_k(D)T_3$ in the following work. Note that a subscript $j$ ($k$) always means that the operation is on the first (second) variable.
\begin{eqnarray*}
  A_{j,k}&=&\psi_1(D/2^j)\psi_2(D/2^k)T_3(f,g,h)\\
  &=& \int_{\n^6} m(\xi,\eta,\zeta) e^{ix_1(\xi_1+\eta_1+\zeta_1)}e^{ix_2(\xi_2+\eta_2+\zeta_2)} \psi_j(\xi_1+\eta_1+\zeta_1)\psi_k(\xi_2+\eta_2+\zeta_2) \\
   & &\quad \cdot \hat f(\xi_1,\xi_2)\hat g(\eta_1,\eta_2) \hat h(\zeta_1,\zeta_2) d\xi_1 d\xi_2 d\eta_1 d\eta_2 d\zeta_1 d\zeta_2 \\
  &=& \int_{\n^6} m(\xi,\eta,\zeta)e^{ix_1(\xi_1+\eta_1+\zeta_1)}e^{ix_2(\xi_2+\eta_2+\zeta_2)} \psi_j(\xi_1+\eta_1+\zeta_1) \psi_k(\xi_2+\eta_2+\zeta_2) \\
  & & \quad \tilde \psi_j(\xi_1) \tilde \psi_k(\eta_2)  \hat f(\xi_1,\xi_2)\hat g(\eta_1,\eta_2) \hat h(\zeta_1,\zeta_2)d\xi_1 d\xi_2 d\eta_1 d\eta_2 d\zeta_1 d\zeta_2 \\
  &=& 2^{3j+3k}\int_{\n^6} (\mathcal{F}^{-1}m_{j,k})(2^j(x_1-y_1),2^k(x_2-y_2),2^j(x_1-z_1), 2^k(x_2-z_2),\\
  & & 2^j(x_1-w_1),2^k(x_2-w_2))(\tilde \psi_j(D) f)(y_1,y_2) (\tilde \psi_k(D)) \\
  & & g(z_1,z_2)h(w_1,w_2) dy_1 dy_2 dz_1 dz_2 dw_1 dw_2,
\end{eqnarray*}
where we use the properties that $|\xi_1|+|\eta_1|+|\zeta_1|\approx |\xi_1|$, $|\xi_2|+|\eta_2|+|\zeta_2|\approx |\eta_2|$ and $\tilde \psi$ is a properly chosen ``completed" function based on the size of support, i.e., $\psi_j(\xi_1+\eta_1+\zeta_1)=\tilde \psi_j(\xi_1)\psi_j(\xi_1+\eta_1+\zeta_1)$  and similar for $\tilde \psi_k$.
\medskip

Then we can estimate $A_{j,k}$ as
\begin{eqnarray*}
  & &A_{j,k}\\
  &\leq &\int_{\n^6}2^{3j+3k}(1+ |2^j(x_1-y_1)|+|2^j(x_1-z_1)|+|2^j(x_1-w_1)|)^{s_1} \\
  & &\quad(1+ |2^k(x_2-y_2)|+|2^k(x_2-z_2)|+|2^k(x_2-w_2)|)^{s_2} \\
  & & \quad (\mathcal{F}^{-1}m_{j,k})(2^j(x_1-y_1),2^k(x_2-y_2),2^j(x_1-z_1),2^k(x_2-z_2), \\
  & & \qquad \qquad \qquad \qquad 2^j(x_1-w_1),2^k(x_2-w_2)) \\
  & & \quad  (1+ |2^j(x_1-y_1)|+|2^j(x_1-z_1)|+|2^j(x_1-w_1)|)^{-s_1} \\
  & & \quad  (1+ |2^k(x_2-y_2)|+|2^k(x_2-z_2)|+|2^k(x_2-w_2)|)^{-s_2} \\
   & & \quad (\tilde \psi_j(D) f)(y_1,y_2) (\tilde \psi_k(D)g)(z_1,z_2) h(w_1,w_2) dy_1 dy_2 dz_1 dz_2 dw_1 dw_2\\
   &\lesssim& \Big(\int_{\n^6}(1+ |y_1|+|z_1|+|w_1|)^{t's_1} (1+ |y_2|+|z_2|+|w_2|)^{t's_2} \\
   & & \quad |(\mathcal{F}^{-1}m_{j,k})(y_1,y_2,z_1,z_2,w_1,w_2)|^{t'} dy dz dw\Big)^{1/t'} \\
    & & \quad \Big(\int_{\mathbb{R}^6}2^{3j+3k}(1+ |2^j(x_1-y_1)|+|2^j(x_1-z_1)|+|2^j(x_1-w_1)|)^{-ts_1}  \\
    & & \quad (1+ |2^k(x_2-y_2)|+|2^k(x_2-z_2)|+|2^k(x_2-w_2)|)^{-ts_2} \\
   & & \quad |(\tilde \psi_j(D) f)(y_1,y_2) (\tilde \psi_k(D)g)(z_1,z_2) h(w_1,w_2)|^t dy_1 dy_2 dz_1 dz_2dw_1 dw_2\Big)^{1\over t} \\
      &\lesssim& \|m_{j,k}\|_{H^{s_1,s_2}} \Big(\int_{\mathbb{R}^6} \frac{2^{j+k}|(\tilde \psi_j(D) f)(y_1,y_2)|^t}{(1+ |2^j(x_1-y_1)|)^{ts_1/3}(1+ |2^k(x_2-y_2)|)^{ts_2/3}} \\
    & & \qquad \qquad \qquad \frac{2^{j+k}|(\tilde \psi_k(D) g)(z_1,z_2)|^t}{(1+ |2^j(x_1-z_1)|)^{ts_1/3}(1+ |2^k(x_2-z_2)|)^{ts_2/3}} \\
     & & \qquad \qquad \qquad\frac{2^{j+k}| h(w_1,w_2)|^t}{(1+ |2^j(x_1-w_1)|)^{ts_1/3}(1+ |2^k(x_2-w_2)|)^{ts_2/3}} dy dz dw\Big)^{1\over t} \\
   &\lesssim& \|m_{j,k}\|_{H^{s_1,s_2}}\left(M_s(|\tilde \psi_j(D)f|^t)(x_1,x_2)\right)^{1\over t} \left(M_s(|\tilde \psi_k(D) g|^t)(x_1,x_2)\right)^{1\over t} \\
   & & \qquad \quad \left(M_s(|h|^t)(x_1,x_2)\right)^{1\over t},
\end{eqnarray*}
where $M_s$ appears  because Lemma \ref{lmstrongmax} is used, with $ts_1/3>1, ts_2/3>1$. Here we need $t>1$ so that  H\"older's inequality can be applied. Also, we take $t<2$ so that the term $\|m_{j,k}\|_{H^{s_1,s_2}}$ can be obtained (see \cite{chen2014hormander} for details), and this restriction is also necessary in the arguments later. In short, we need $\max{(1,3/s_1,3/s_2)}<t<2$.

Then
\begin{eqnarray*}
  & &\|T_3(f,g,h)\|_{L^r} \\
  & \lesssim& \left\| \left\{\sum_{j,k} |\psi_1(D/2^j)\psi_2(D/2^k)T_3(f,g,h) |^2\right\}^{1\over 2}\right\|_{L^r} \\
  &\lesssim& \sup_{j,k} \|m_{j,k}\|_{H^{s_1,s_2}} \Bigg\| \Big[ \sum_{j,k}\left( M_s(|\tilde \psi_j(D)f|^t)(x_1,x_2)\right)^{2\over t} \\
 & & \quad \left(M_s(|\tilde \psi_k(D) g|^t)(x_1,x_2)\right)^{2\over t} \cdot  \left(M_s(|h|^t)(x_1,x_2)\right)^{2\over t} \Big]^{1\over 2}\Bigg\|_{L^r} \\
 &=&\sup_{j,k} \|m_{j,k}\|_{H^{s_1,s_2}} \Bigg\| \Big[ \sum_{j}\left( M_s(|\tilde \psi_j(D)f|^t)(x_1,x_2)\right)^{2\over t}\Big]^{1\over 2} \\
 & & \quad\Big[ \sum_{k}\left(M_s(|\tilde \psi_k(D) g|^t)(x_1,x_2)\right)^{2\over t}\Big]^{1\over 2} \cdot  \left(M_s(|h|^t)(x_1,x_2)\right)^{1\over t}\Bigg\|_{L^r}.
\end{eqnarray*}

If we introduce the tensor product assumption $g(x_1,x_2)=g_1(x_1)\otimes g_2(x_2)$ and $h(x_1,x_2)=h_1(x_1)\otimes h_2(x_2)$, note that
 \begin{gather*}
   M_s(g)(x_1,x_2)=M(g_1)(x_1) M(g_2)(x_2), \  M_s(h)(x_1,x_2)= M(h_1)(x_1) M(h_2)(x_2),
 \end{gather*}
 where $M$ is the Hardy-Littlewood maximal operator.

 Thus, we can write
\begin{eqnarray*}
  & &\|T_3(f,g,h)\|_{L^r} \\
  & \lesssim& \left\| \left\{\sum_{j,k} |\psi_1(D/2^j)\psi_2(D/2^k)T_3(f,g,h) |^2\right\}^{1\over 2}\right\|_{L^r} \\
  &\lesssim& \sup_{j,k} \|m_{j,k}\|_{H^{s_1,s_2}} \Bigg\| \Big[ \sum_{j} \left( M_s(|\tilde \psi_j(D)f|^t)(x_1,x_2)\right)^{2\over t} \Big]^{1\over 2}  \\
 & & \quad \left( M(|g_1|^t)(x_1)\right)^{1\over t} \Big[ \sum_k  \left(M(| \tilde \psi_k(D) g_2|^t)(x_2)\right)^{2\over t} \Big]^{1\over 2}  \cdot \left( M(|  h_1|^t)(x_1)\right)^{1\over t} \left( M(|h_2|^t)(x_2) \right)^{1\over t} \Bigg\|_{L^r} \\
  &\lesssim& \sup_{j,k} \|m_{j,k}\|_{H^{s_1,s_2}} \Bigg\| \Big[ \sum_{j} \left( M_s(|\tilde \psi_j(D)f|^t)(x_1,x_2)\right)^{2\over t} \Big]^{1\over 2} \Bigg\|_{L^p} \\
 & & \quad \Bigg\|\left( M(|g_1|^t)(x_1)\right)^{1\over t} \Big[ \sum_k  \left(M(| \tilde \psi_k(D) g_2|^t)(x_2)\right)^{2\over t} \Big]^{1\over 2}  \cdot \left( M(|  h_1|^t)(x_1)\right)^{1\over t} \left( M(|h_2|^t)(x_2) \right)^{1\over t}  \Bigg\|_{L^s},
 \end{eqnarray*}
 where we just apply the H\"older's inequality and $1/p+1/s=1/r$, i.e., $1/p_2+1/q_2=1/p_3+1/q_3=1/s$. Then using the tensor product setting, the above can be estimated by

 \begin{eqnarray*}
   &=& \sup_{j,k} \|m_{j,k}\|_{H^{s_1,s_2}} \Bigg\| \Big[ \sum_{j} \left( M_s(|\tilde \psi_j(D)f|^t)(x)\right)^{2\over t} \Big]^{1\over 2} \Bigg\|_{L^p} \\
 & & \quad \Big\|\left( M(|g_1|^t)(x_1)\right)^{1\over t} \left( M(|  h_1|^t)(x_1)\right)^{1\over t} \Big\|_{L^s(x_1)}   \\
 & & \quad \Bigg\| \Big[ \sum_k  \left(M(| \tilde \psi_k(D) g_2|^t)(x_2)\right)^{2\over t} \Big]^{1\over 2}  \left( M(|h_2|^t)(x_2) \right)^{1\over t} \Bigg\|_{L^s(x_2)} 
 \end{eqnarray*}

Again by H\"older's inequality, we have
 \begin{eqnarray*}
 &\lesssim& \sup_{j,k} \|m_{j,k}\|_{H^{s_1,s_2}} \Bigg\| \Big[ \sum_{j} \left( M_s(|\tilde \psi_j(D)f|^t)(x)\right)^{2\over t} \Big]^{1\over 2} \Bigg\|_{L^p}   \\
 & &  \Big\|\big( M(|g_1|^t)(x_1)\big)^{1\over t}  \Big\|_{L^{p_2}(x_1)} \Bigg\|\Big[ \sum_k  \left(M(| \tilde \psi_k(D) g_2|^t)(x_2)\right)^{2\over t} \Big]^{1\over 2}  \Bigg\|_{L^{q_2}(x_2)}  \\
 & & \quad   \Big\|\left( M(|  h_1|^t)(x_1)\right)^{1\over t} \Big\|_{L^{p_3}(x_1)} \Big\|\left( M(|h_2|^t)(x_2) \right)^{1\over t} \Big\|_{L^{q_3}(x_2)} \\
  &=&\sup_{j,k} \|m_{j,k}\|_{H^{s_1,s_2}} \Bigg\| \Big[ \sum_{j} \left( M_s(|\tilde \psi_j(D)f|^t)(x)\right)^{2\over t} \Big]^{t\over 2} \Bigg\|_{L^{p/t}}^{1/t}   \\
 & & \Big\| M(|g_1|^t)(x_1) \Big\|_{L^{p_2/t}(x_1)}^{1/t}    \Bigg\|\Big[ \sum_k  \left(M(| \tilde \psi_k(D) g_2|^t)(x_2)\right)^{2\over t} \Big]^{t\over 2}  \Bigg\|_{L^{q_2/t}(x_2)}^{1/t}  \\
 & & \quad   \Big\| M(|  h_1|^t)(x_1) \Big\|_{L^{p_3/t}(x_1)}^{1/t}  \Big\|M(|h_2|^t)(x_2) \Big\|_{L^{q_3/t}(x_2)}^{1/t}  \\
 &\lesssim & \sup_{j,k} \|m_{j,k}\|_{H^{s_1,s_2}} \|f\|_{L^{p}}\|g_1\|_{L^{p_2}(x_1)}\|g_2\|_{L^{q_2}(x_2)}\|h_1\|_{L^{p_3}(x_1)}\|h_2\|_{L^{q_3}(x_2)},
\end{eqnarray*}
where we use the vector-valued Fefferman-Stein inequality, $2,p, p_2,q_2,p_3,q_3>t$.

That means for $I_1(a)\otimes I_2(b)$, we have proved the boundedness $L^{p} \times L^{p_2}_{x_1}(L^{q_2}_{x_2}) \times L^{p_3}_{x_1}(L^{q_3}_{x_2})  \to L^r$ for $p, p_2,q_2,p_3,q_3>t$ and $\max{(1,3/s_1,3/s_2)}<t<2$. Thus, if one takes $s_1,s_2$ to be properly large, one can take $t$ to be arbitrarily close to $1$, which means the mixed $L^r$ estimate can hold for $1<p, p_2,q_2,p_3,q_3<\infty$.

The above arguments, with some modifications, will be used to treat all the cases: $I_1\otimes I_2$, $I_1\otimes II_2$ and $II_1\otimes II_2$.
Here we briefly describe the modifications necessary for just a few of these cases.
\medskip

We start by considering some of the other subcases in $I_1\times I_2$. For instance, if one considers $I_1(b)\times I_2(c)$, then
\begin{eqnarray*}
  A_{j,k}&=&\psi_1(D/2^j)\psi_2(D/2^k)T_3(f,g,h) \\
  &\lesssim&  \|m_{j,k}\|_{H^{s_1,s_2}}\left(M_s(|f|^t)(x_1,x_2)\right)^{1\over t} \left(M_s(|\tilde \psi_j(D) g|^t)(x_1,x_2)\right)^{1\over t} \\
   & & \qquad \quad \left(M_s(|\tilde \psi_k(D)h|^t)(x_1,x_2)\right)^{1\over t}.
\end{eqnarray*}
Then
\begin{eqnarray*}
  & &\|T_3(f,g,h)\|_{L^r} \\
  & \lesssim& \left\| \left\{\sum_{j,k} |\psi_1(D/2^j)\psi_2(D/2^k)T_3(f,g,h) |^2\right\}^{1\over 2}\right\|_{L^r} \\
  &\lesssim& \sup_{j,k} \|m_{j,k}\|_{H^{s_1,s_2}} \Big\| \left( M_s(| f|^t)\right)^{1\over t} \Big\|_{L^p} \Bigg\| \Big[\sum_j \left(M(|\tilde \psi_j(D)g_1|^t)(x_1)\right)^{2\over t} \Big]^{1\over 2}\\
 & & \quad  \left( M(|g_2|^t)(x_2)\right)^{1\over t} \cdot \left( M(|  h_1|^t)(x_1)\right)^{1\over t} \Big[ \sum_k \left( M(|\tilde \psi_k(D)h_2|^t)(x_2) \right)^{2\over t} \Big]^{1\over 2}\Bigg\|_{L^s} \\
 &\lesssim & \sup_{j,k} \|m_{j,k}\|_{H^{s_1,s_2}} \|f\|_{L^{p}}\|g_1\|_{L^{p_2}(x_1)}\|g_2\|_{L^{q_2}(x_2)}\|h_1\|_{L^{p_3}(x_1)}\|h_2\|_{L^{q_3}(x_2)}.
 \end{eqnarray*}

 \medskip

 For the cases in $II_1\times II_2$, we first consider $II_1{(a)}\times II_2{(b)}$ for example.  One can write

\begin{eqnarray*}
 & & T_3(f,g,h)(x_1,x_2)\\
  &=& \sum_{j,k} \int_{\n^6} m(\xi,\eta,\zeta) e^{ix_1(\xi_1+\eta_1+\zeta_1)}e^{ix_2(\xi_2+\eta_2+\zeta_2)} \psi_j(\xi_1)  \tilde \psi_j(\eta_1) \psi_k(\eta_2) \tilde \psi_k(\zeta_2)  \\
   & &\quad \cdot \hat f(\xi_1,\xi_2)\hat g(\eta_1,\eta_2) \hat h(\zeta_1,\zeta_2) d\xi_1 d\xi_2 d\eta_1 d\eta_2 d\zeta_1 d\zeta_2 \\
  &=& \sum_{j,k} 2^{3j+3k}\int_{\n^6} (\mathcal{F}^{-1}m_{j,k})(2^j(x_1-y_1),2^k(x_2-y_2),2^j(x_1-z_1), 2^k(x_2-z_2),\\
  & & 2^j(x_1-w_1),2^k(x_2-w_2))( \psi_j(D) f)(y_1,y_2) (\tilde \psi_j(D) \psi_k(D)  g)(z_1,z_2) \\
  & & (\tilde \psi_k(D) h)(w_1,w_2) dy_1 dy_2 dz_1 dz_2 dw_1 dw_2,
\end{eqnarray*}
where we take $\sum_j \psi_j(u)=1$ for $u\neq 0$, and take $\tilde \psi$ a properly chosen ``completed" function by using properties that $|\xi_1|+|\eta_1|+|\zeta_1|\approx |\xi_1|\approx |\eta_1|$, $|\xi_2|+|\eta_2|+|\zeta_2|\approx |\eta_2|\approx |\zeta_2|$. Using the control of Sobolev norm as before,

\begin{eqnarray*}
 & & |T_3(f,g,h)(x_1,x_2)|\\
  &\lesssim&   \sup_{j,k}\|m_{j,k}\|_{H^{s_1,s_2}} \sum_{j,k} \left[M_s \left(|\tilde \psi_j(D) f|^t\right)\right]^{1\over t} \left[M_s \left(| \tilde \psi_j(D) \psi_k(D)  g|^t\right)\right]^{1\over t}  \left[M_s \left(|\tilde \psi_k(D) h|^t\right)\right]^{1\over t} \\
    &\lesssim&   \sup_{j,k}\|m_{j,k}\|_{H^{s_1,s_2}}  \Bigg\{\sum_{j} \left[M_s \left(|\tilde \psi_j(D) f|^t\right)\right]^{2\over t} \Bigg\}^{1\over 2} \Bigg\{\sum_{j,k}\left[M_s \left(| \tilde \psi_j(D) \psi_k(D)  g|^t\right)\right]^{2\over t} \Bigg\}^{1\over 2}  \\
    & & \quad \cdot \Bigg\{ \sum_k \left[M_s \left(|\tilde \psi_k(D) h|^t\right)\right]^{2\over t}\Bigg\}^{1\over 2},
\end{eqnarray*}
where the Cauchy-Schwartz inequality is used.  Then the H\"older's inequality gives
\begin{eqnarray*}
& &  \| T_3(f,g,h)\|_{L^r} \\
&\lesssim&   \sup_{j,k}\|m_{j,k}\|_{H^{s_1,s_2}}  \Bigg\| \Bigg\{\sum_{j} \left[M_s \left(|\tilde \psi_j(D) f|^t\right)\right]^{2\over t} \Bigg\}^{1\over 2}  \Bigg\|_{L^{p}} \\
& & \Bigg\|\Bigg\{\sum_{j,k}\left[M_s \left(| \tilde \psi_j(D) \psi_k(D)  g|^t\right)\right]^{2\over t} \Bigg\}^{1\over 2}  \Bigg\{ \sum_k \left[M_s \left(|\tilde \psi_k(D) h|^t\right)\right]^{2\over t}\Bigg\}^{1\over 2} \Bigg\|_{L^s} \\
&\lesssim &\sup_{j,k}\|m_{j,k}\|_{H^{s_1,s_2}}  \Bigg\| \Bigg\{\sum_{j} \left[M_s \left(|\tilde \psi_j(D) f|^t\right)\right]^{2\over t} \Bigg\}^{1\over 2}  \Bigg\|_{L^{p}} \\
& &\Bigg\|\Bigg\{\sum_{j}\left[M \left(| \tilde \psi_j(D) g_1|^t\right)\right]^{2\over t} \Bigg\}^{1\over 2}  \left[M \left(|h_1|^t\right)\right]^{1\over t} \Bigg\|_{L^s(x_1)} \\
& &\Bigg\|\Bigg\{\sum_{k}\left[M \left(| \tilde \psi_k(D)  g_2|^t\right)\right]^{2\over t} \Bigg\}^{1\over 2}  \Bigg\{ \sum_k \left[M \left(|\tilde \psi_k(D) h_2|^t\right)\right]^{2\over t}\Bigg\}^{1\over 2} \Bigg\|_{L^s(x_2)} \\
&\lesssim &\sup_{j,k}\|m_{j,k}\|_{H^{s_1,s_2}}  \Bigg\| \Bigg\{\sum_{j} \left[M \left(|\tilde \psi_j(D) f|^t\right)\right]^{2\over t} \Bigg\}^{1\over 2}  \Bigg\|_{L^{p}} \\
& &\Bigg\|\Bigg\{\sum_{j}\left[M \left(| \tilde \psi_j(D) g_1|^t\right)\right]^{2\over t} \Bigg\}^{1\over 2}  \Bigg\|_{L^{p_2}(x_1)} \Bigg\|\left[M \left(|h_1|^t\right)\right]^{1\over t} \Bigg\|_{L^{p_3}(x_1)} \\
& &\Bigg\|\Bigg\{\sum_{k}\left[M \left(| \tilde \psi_k(D)  g_2|^t\right)\right]^{2\over t} \Bigg\}^{1\over 2} \Bigg\|_{L^{q_2}(x_2)} \Bigg\| \Bigg\{ \sum_k \left[M \left(|\tilde \psi_k(D) h_2|^t\right)\right]^{2\over t}\Bigg\}^{1\over 2} \Bigg\|_{L^{q_3}(x_2)}. \\
&\lesssim & \sup_{j,k} \|m_{j,k}\|_{H^{s_1,s_2}} \|f\|_{L^{p}}\|g_1\|_{L^{p_2}(x_1)}\|g_2\|_{L^{q_2}(x_2)}\|h_1\|_{L^{p_3}(x_1)}\|h_2\|_{L^{q_3}(x_2)}.
\end{eqnarray*}

For other cases that belong to $II_1\times II_2$, we consider $II_1(a)\times II_2(c)$. One can obtain
\begin{eqnarray*}
 & & T_3(f,g,h)(x_1,x_2)\\
  &=& \sum_{j,k} \int_{\n^6} m(\xi,\eta,\zeta) e^{ix_1(\xi_1+\eta_1+\zeta_1)}e^{ix_2(\xi_2+\eta_2+\zeta_2)} \psi_j(\xi_1)  \tilde \psi_j(\eta_1) \psi_k(\xi_2) \tilde \psi_k(\zeta_2)  \\
   & &\quad \cdot \hat f(\xi_1,\xi_2)\hat g(\eta_1,\eta_2) \hat h(\zeta_1,\zeta_2) d\xi_1 d\xi_2 d\eta_1 d\eta_2 d\zeta_1 d\zeta_2 \\
  &\lesssim&   \sup_{j,k}\|m_{j,k}\|_{H^{s_1,s_2}} \sum_{j,k} \left[M_s \left(| \psi_j(D) \psi_k(D) f|^t\right)\right]^{1\over t} \left[M_s \left(| \tilde \psi_j(D) g|^t\right)\right]^{1\over t}  \left[M_s \left(|\tilde \psi_k(D) h|^t\right)\right]^{1\over t} \\
    &\lesssim&   \sup_{j,k}\|m_{j,k}\|_{H^{s_1,s_2}}  \Bigg\{\sum_{j,k} \left[M_s \left(|\psi_j(D)  \psi_k(D) f|^t\right)\right]^{2\over t} \Bigg\}^{1\over 2} \Bigg\{\sum_{j}\left[M_s \left(| \tilde \psi_j(D)  g|^t\right)\right]^{2\over t} \Bigg\}^{1\over 2}  \\
    & & \quad \cdot \Bigg\{ \sum_k \left[M_s \left(|\tilde \psi_k(D) h|^t\right)\right]^{2\over t}\Bigg\}^{1\over 2}.
\end{eqnarray*}
where we use the properties that $|\xi_1|+|\eta_1|+|\zeta_1|\approx |\xi_1|\approx |\eta_1|$, $|\xi_2|+|\eta_2|+|\zeta_2|\approx |\xi_2|\approx|\zeta_2|$ and $\tilde \psi$ is a proper ``completed" function as before. Then as before H\"older's inequality gives
\begin{eqnarray*}
& &  \| T_3(f,g,h)\|_{L^r} \\
&\lesssim&   \sup_{j,k}\|m_{j,k}\|_{H^{s_1,s_2}}  \Bigg\| \Bigg\{\sum_{j,k} \left[M_s \left(|\tilde \psi_j(D) \psi_k(D) f|^t\right)\right]^{2\over t} \Bigg\}^{1\over 2}  \Bigg\|_{L^{p}} \\
& & \Bigg\|\Bigg\{\sum_{j}\left[M_s \left(| \tilde \psi_j(D)   g|^t\right)\right]^{2\over t} \Bigg\}^{1\over 2}  \Bigg\{ \sum_k \left[M_s \left(|\tilde \psi_k(D) h|^t\right)\right]^{2\over t}\Bigg\}^{1\over 2} \Bigg\|_{L^s} \\
&\lesssim &\sup_{j,k}\|m_{j,k}\|_{H^{s_1,s_2}}  \Bigg\| \Bigg\{\sum_{j,k} \left[M \left(|\tilde \psi_j(D) \psi_k(D)f|^t\right)\right]^{2\over t} \Bigg\}^{1\over 2}  \Bigg\|_{L^{p}} \\
& &\Bigg\|\Bigg\{\sum_{j}\left[M \left(| \tilde \psi_j(D) g_1|^t\right)\right]^{2\over t} \Bigg\}^{1\over 2}  \Bigg\|_{L^{p_2}(x_1)} \Bigg\|\left[M \left(|h_1|^t\right)\right]^{1\over t} \Bigg\|_{L^{p_3}(x_1)} \\
& &\Bigg\| \left[M \left(| g_2|^t\right)\right]^{1\over t}  \Bigg\|_{L^{q_2}(x_1)} \Bigg\| \Bigg\{ \sum_k \left[M \left(|\tilde \psi_k(D) h_2|^t\right)\right]^{2\over t}\Bigg\}^{1\over 2} \Bigg\|_{L^{q_3}(x_2)}. \\
&\lesssim & \sup_{j,k} \|m_{j,k}\|_{H^{s_1,s_2}} \|f\|_{L^{p}}\|g_1\|_{L^{p_2}(x_1)}\|g_2\|_{L^{q_2}(x_2)}\|h_1\|_{L^{p_3}(x_1)}\|h_2\|_{L^{q_3}(x_2)}.
\end{eqnarray*}
Other cases in $II_1 \times II_2$ can be treated similarly.
\medskip

Now consider $I_1 \times II_2$, specifically $I_1(a) \times II_2(c)$.
\begin{eqnarray*}
  \|T_3(f,g,h)\|_{L^r}\lesssim \left\| \left\{\sum_{j} |\psi_1(D/2^j)T_3(f,g,h) |^2\right\}^{1\over 2}\right\|_{L^r}.
\end{eqnarray*}
Let
\begin{eqnarray*}
  B_{j}&=&\psi_1(D/2^j)T_3(f,g,h)(x_1,x_2)\\
  &=& \int_{\n^6} m(\xi,\eta,\zeta) e^{ix_1(\xi_1+\eta_1+\zeta_1)}e^{ix_2(\xi_2+\eta_2+\zeta_2)} \psi_j(\xi_1+\eta_1+\zeta_1) \\
   & &\quad \cdot \hat f(\xi_1,\xi_2)\hat g(\eta_1,\eta_2) \hat h(\zeta_1,\zeta_2) d\xi_1 d\xi_2 d\eta_1 d\eta_2 d\zeta_1 d\zeta_2 \\
  &=& \int_{\n^6}\sum_{k} m(\xi,\eta,\zeta)e^{ix_1(\xi_1+\eta_1+\zeta_1)}e^{ix_2(\xi_2+\eta_2+\zeta_2)} \psi_j(\xi_1+\eta_1+\zeta_1) \tilde  \psi_j(\xi_1)  \\
  & & \quad  \psi_k(\xi_2)  \tilde \psi_k(\zeta_2)  \hat f(\xi_1,\xi_2)\hat g(\eta_1,\eta_2) \hat h(\zeta_1,\zeta_2)d\xi_1 d\xi_2 d\eta_1 d\eta_2 d\zeta_1 d\zeta_2 \\
  &=& \sum_k 2^{3j+3k}\int_{\n^6} (\mathcal{F}^{-1}m_{j,k})(2^j(x_1-y_1),2^k(x_2-y_2),2^j(x_1-z_1), 2^k(x_2-z_2),\\
  & & 2^j(x_1-w_1),2^k(x_2-w_2))(\tilde \psi_j(D)  \psi_k(D) f)(y_1,y_2)  g(z_1,z_2) (\tilde \psi_k(D)h)(w_1,w_2) dy dz dw,
\end{eqnarray*}
where we use the properties that $|\xi_1|+|\eta_1|+|\zeta_1|\approx |\xi_1|$, $|\xi_2|+|\eta_2|+|\zeta_2|\approx |\xi_2| \approx |\zeta_2|$ and $\tilde \psi$ is a properly chosen ``completed" function.

Then as before one can get
\begin{eqnarray*}
  B_j &\lesssim& \sup_{k}\|m_{j,k}\|_{H^{s_1,s_2}} \sum_k \left(M_s(|\tilde \psi_j(D)  \psi_k(D)f|^t)(x_1,x_2)\right)^{1\over t} \left(M_s(| g|^t)(x_1,x_2)\right)^{1\over t} \\
   & & \qquad \quad \left(M_s(|\tilde \psi_k(D) h|^t)(x_1,x_2)\right)^{1\over t} \\
   &\lesssim& \sup_{k}\|m_{j,k}\|_{H^{s_1,s_2}} \Big\{ \sum_k \left(M_s(|\tilde \psi_j(D)  \psi_k(D)f|^t)(x_1,x_2)\right)^{2\over t} \Big\}^{1\over 2}\left(M_s(| g|^t)(x_1,x_2)\right)^{1\over t} \\
   & & \qquad \quad  \left\{\sum_k \left(M_s(|\tilde \psi_k(D) h|^t)(x_1,x_2)\right)^{2\over t} \right\}^{1\over 2},
\end{eqnarray*}
where the Cauch-Schwartz inequality is used. Then
\begin{eqnarray*}
  & &\|T_3(f,g,h)\|_{L^r} \\
  & \lesssim& \left\| \left\{\sum_{j} |\psi_1(D/2^j) T_3(f,g,h) |^2\right\}^{1\over 2}\right\|_{L^r} \\
  &\lesssim& \sup_{j,k} \|m_{j,k}\|_{H^{s_1,s_2}} \Bigg\| \Big\{ \sum_{j,k}\left( M_s(|\tilde \psi_j(D) \psi_k(D)f|^t)(x_1,x_2)  \right)^{2\over t}\Big\}^{1\over 2} \\
 & & \quad \left(M_s(| g|^t)(x_1,x_2)\right)^{1\over t} \cdot  \left\{\sum_k \left(M_s(|\tilde \psi_k(D) h|^t)(x_1,x_2)\right)^{2\over t} \right\}^{1\over 2} \Bigg\|_{L^r} \\
 &\lesssim& \sup_{j,k} \|m_{j,k}\|_{H^{s_1,s_2}} \Bigg\| \Big\{ \sum_{j,k}\left( M_s(|\tilde \psi_j(D) \psi_k(D)f|^t)(x_1,x_2)  \right)^{2\over t}\Big\}^{1\over 2} \Bigg\|_{L^p} \\
 & & \quad \Bigg\|\left(M_s(| g|^t)(x_1,x_2)\right)^{1\over t} \cdot  \left\{\sum_k \left(M_s(|\tilde \psi_k(D) h|^t)(x_1,x_2)\right)^{2\over t} \right\}^{1\over 2} \Bigg\|_{L^s}\\
 &\lesssim& \sup_{j,k} \|m_{j,k}\|_{H^{s_1,s_2}} \Bigg\| \Big\{ \sum_{j,k}\left( M_s(|\tilde \psi_j(D) \psi_k(D)f|^t)(x_1,x_2)  \right)^{2\over t}\Big\}^{1\over 2} \Bigg\|_{L^p} \\
 & & \quad \Bigg\|\left(M(| g_1|^t)(x_1)\right)^{1\over t} \cdot \left(M(| h_1|^t)(x_1)\right)^{1\over t}  \Bigg\|_{L^s(x_1)}\\
 & & \quad \Bigg\|\left(M(| g_2|^t)(x_2)\right)^{1\over t} \cdot  \left\{\sum_k \left(M(|\tilde \psi_k(D) h_2|^t)(x_2)\right)^{2\over t} \right\}^{1\over 2} \Bigg\|_{L^s(x_2)} \\
 &\lesssim& \sup_{j,k} \|m_{j,k}\|_{H^{s_1,s_2}} \Bigg\| \Big\{ \sum_{j,k}\left( M_s(|\tilde \psi_j(D) \psi_k(D)f|^t)(x_1,x_2)  \right)^{2\over t}\Big\}^{1\over 2} \Bigg\|_{L^p} \\
 & & \quad \Big\|\left(M(| g_1|^t)(x_1)\right)^{1\over t}\Big\|_{L^{p_2}(x_1)} \cdot \Big\|  \left(M(| h_1|^t)(x_1)\right)^{1\over t} \Big\|_{L^{p_3}(x_1)}\\
 & & \quad \Big\|\left(M(| g_2|^t)(x_2)\right)^{1\over t} \big\|_{L^{q_2}(x_2)}\cdot \left\| \left\{\sum_k \left(M(|\tilde \psi_k(D) h_2|^t)(x_2)\right)^{2\over t} \right\}^{1\over 2} \right\|_{L^{q_3}(x_2)} \\
 &\lesssim&\sup_{j,k} \|m_{j,k}\|_{H^{s_1,s_2}} \|f\|_{L^{p}}\|g_1\|_{L^{p_2}(x_1)}\|g_2\|_{L^{q_2}(x_2)}\|h_1\|_{L^{p_3}(x_1)}\|h_2\|_{L^{q_3}(x_2)}.
\end{eqnarray*}

\medskip
Similarly, for other cases in $I_1 \times II_2$, if we take $I_1(a) \times II_2(b)$,
\begin{eqnarray*}
  B_{j}&:=&\psi_1(D/2^j)T_3(f,g,h)\\
   &\lesssim& \sup_{k}\|m_{j,k}\|_{H^{s_1,s_2}}  \left(M_s(|\tilde \psi_j(D)  f|^t)(x_1,x_2)\right)^{1\over t}  \Big\{ \sum_k\left(M_s(| \psi_k(D) g|^t)(x_1,x_2)\right)^{2\over t}\Big\}^{1\over 2} \\
   & & \qquad \quad  \left\{\sum_k \left(M_s(|\tilde \psi_k(D) h|^t)(x_1,x_2)\right)^{2\over t} \right\}^{1\over 2}.
\end{eqnarray*}
Then
\begin{eqnarray*}
  & &\|T_3(f,g,h)\|_{L^r} \\
  & \lesssim& \left\| \left\{\sum_{j} |\psi_1(D/2^j) T_3(f,g,h) |^2\right\}^{1\over 2}\right\|_{L^r} \\
  &\lesssim& \sup_{j,k} \|m_{j,k}\|_{H^{s_1,s_2}} \Bigg\| \Big\{ \sum_{j}\left( M_s(|\tilde \psi_j(D) |^t)(x_1,x_2)  \right)^{2\over t}\Big\}^{1\over 2} \\
 & & \quad \Big\{ \sum_k\left(M_s(| \psi_k(D) g|^t)(x_1,x_2)\right)^{2\over t}\Big\}^{1\over 2} \cdot  \left\{\sum_k \left(M_s(|\tilde \psi_k(D) h|^t)(x_1,x_2)\right)^{2\over t} \right\}^{1\over 2} \Bigg\|_{L^r} \\
 &\lesssim& \sup_{j,k} \|m_{j,k}\|_{H^{s_1,s_2}} \Bigg\| \Big\{ \sum_{j}\left( M_s(|\tilde \psi_j(D) f|^t)(x_1,x_2)  \right)^{2\over t}\Big\}^{1\over 2} \Bigg\|_{L^p} \\
 & & \quad \Bigg\|\left(M(| g_1|^t)(x_1)\right)^{1\over t} \cdot  \left(M(| h_1|^t)(x_1)\right)^{1\over t} \Bigg\|_{L^s(x_1)}\\
 & & \quad \Bigg\| \left\{\sum_k \left(M(|\tilde \psi_k(D) g_2|^t)(x_2)\right)^{2\over t} \right\}^{1\over 2}  \cdot  \left\{\sum_k \left(M(|\tilde \psi_k(D) h_2|^t)(x_2)\right)^{2\over t} \right\}^{1\over 2} \Bigg\|_{L^s(x_2)} \\
 &\lesssim&\sup_{j,k} \|m_{j,k}\|_{H^{s_1,s_2}} \|f\|_{L^{p}}\|g_1\|_{L^{p_2}(x_1)}\|g_2\|_{L^{q_2}(x_2)}\|h_1\|_{L^{p_3}(x_1)}\|h_2\|_{L^{q_3}(x_2)}.
\end{eqnarray*}

In this way such modification gives the desired mixed norm estimate.

\end{proof}

\section{}
\label{ap3}

In this section, we sketch how  our reduction  can be used to establish the weighted mixed norm estimates. Thus, the results of Appendix \ref{ap2}
are the unweighted estimates which we presented first for clarity.

Recall the Muckenhoupts $A_p$ weights.

  \medskip
  \begin{definition}[\cite{Muckenhoupt}] ~ \quad
  \begin{itemize}
   \item[(a)] We say a weight $w\geq 0$ belong to  the Muckenhoupt class $A_p(\n)$ $(1<p<\infty)$ if
    $$\sup_I \left(\frac{1}{|I|}\int_I w(x)dx\right) \left(\frac{1}{|I|}\int_I w(x)^{1\over 1-p}dx\right)^{p-1}<\infty,$$
    where the supremum is taken over all intervals in $\n$. Also, $w\geq 0$ belong to $A_1$ if there exists some $C>0$ such that
    $$\sup_I \frac{1}{|I|}\int_I w(x)dx\leq C w(x).$$
    Then class $A_\infty$ is defined to be $A_\infty=\cup_{1\leq p <\infty}A_p$.
     \item[(b)] The weighted $L^p$ space is defined via the norm  $\|f\|_{L_w^p(\n)}=\left(\int_\n |f(x)|^p w(x) dx \right)^{1\over p}$.
    \end{itemize}
  \end{definition}
   In the product setting, $A_p(\n\times \n)$ is defined the same way, replacing the intervals $I$ by rectangles. Then we introduce several lemmas that will be useful later. Since they will be true for both $A_p(\n\times \n)$ and $A_p(\n)$,  for convenience we will simply use the notation $A_p$.

  \begin{lemma}[\cite{GR}]
  \label{lemmaGR}
    Let $1<p<\infty$ and $w\in A_p$, then
    \begin{itemize}
      \item[(1)] $w^{1-p'}\in  A_{p'}$,
      \item[(2)] $\exists 1<q<p $ such that $w\in A_q$.
    \end{itemize}
  \end{lemma}

    \begin{lemma}[\cite{chen2014hormander}]
    \label{lemmaChen}
   Let $w_j \in A_{p_j}$ for $1\leq j \leq m$ for some $1\leq p_1,\dots,p_m \leq \infty$. Then for any $\theta_1+\cdots \theta_m=1$ with  $0<\theta_1,\dots \theta_m<1$, there holds
   $$w_1^{\theta_1}\cdots w_m^{\theta_m} \in A_{\max{(p_1,\dots,p_m)}}. $$
  \end{lemma}

  \begin{lemma}[\cite{FS}]
  \label{lemmaFS}
    Let $1<p,q<\infty$ and $w\in A_p$, the weighted vector-valued maximal inequality
    $$\left\|\left\{\sum_{j\in \z}(M_s f_{j})^q\right\}^{1\over q}\right\|_{L^p(w)}\lesssim \left\|\left\{\sum_{j\in \z}(f_{j})^q\right\}^{1\over q}\right\|_{L^p(w)},$$
    for all sequences $\{f_{j}\}_{j\in \z}$ of locally integrable functions on $\n \times \n$.
  \end{lemma}

  \begin{lemma}[\cite{RFS}]
  \label{lemmaRFS}
    Let $1<p<\infty$ and $w\in A_p$. Suppose $\psi_i \in \mathcal{S}(\n)$ satisfies $\spt \hat \psi_i \subseteq \{u\in \n: 1/a_i \leq |u| \leq a_i\}$ for $a_i>1$ $(i=1,2)$, then there holds
    $$\left\|\left\{\sum_{j,k\in \z}|\psi_1(D/2^j)\psi_2(D/2^k)f|^2\right\}^{1\over 2}\right\|_{L^p(w)} \lesssim \|f\|_{L^p(w)} \quad \text{for } f\in L^p(w).$$
    Moreover, if $\sum_k \hat \psi_i(u/2^k)=1$ for $u\neq 0$ $(i=1,2)$, then
    $$\left\|\left\{\sum_{j,k\in \z}|\psi_1(D/2^j)\psi_2(D/2^k)f|^2\right\}^{1\over 2}\right\|_{L^p(w)} \approx \|f\|_{L^p(w)}\quad \text{for } f\in L^p(w).$$
  \end{lemma}

  \begin{lemma}[\cite{Ruan}]
  \label{lemmaRuan}
    Let $0<p<\infty$ and $w\in A_\infty$ and  $\psi_1,\psi_2$ be as in the previous lemma. Then for a locally integrable  function $f\in H^p(w)$, there holds
    $$ \|f\|_{L^p(w)} \lesssim \left\|\left\{\sum_{j,k\in \z}|\psi_1(D/2^j)\psi_2(D/2^k)f|^2\right\}^{1\over 2}\right\|_{L^p(w)}  .$$

  \end{lemma}

\medskip

In order to establish the weighted mixed norm estimate for \eqref{biop34} under the tensor product assumption $g_1\otimes g_2$ and $h_1\otimes h_2$, which is stated as Theorem  \ref{goalweighted1}, as before it suffices to prove the same boundedness property for our reduced operator \eqref{assumption} (as well as \eqref{assumption2} by symmetry), and the bi-parameter trilinear multiplier $T_3$ defined in \eqref{opT3}.

 We first study the weighted estimate for the reduced operator \eqref{assumption}.  One will see, for technical purposes, we will first consider the Sobolev regularity for the symbols instead of the H\"ormander type condition.
  \begin{proposition}
  \label{weightedreduced}
    Let $g(x)=g_1(x_1)\otimes g_2(x_2) $, $h(x)=h_1(x_1)\otimes h_2(x_2) $,  and  the multipliers in \eqref{assumption} satisfy the limited smoothness condition in the sense that for $3/2< s_1\leq 3$, $1< s_2\leq 2$,
    $$\sup_{j\in \z} \|m'_j\|_{H^{s_1}}<\infty, \quad \sup_{k \in \z} \|m''_k\|_{H^{s_2}}<\infty. $$
    Assume that
    \begin{gather}
     \label{assumptionreduced1}\min {(p,p_2,p_3)}>3/s_1 \quad \text{and} \quad w_1^1(x_1) \in A_{ps_1/3},\  w_1^2(x_1) \in A_{p_2s_1/3}, \ w_1^3(x_1) \in A_{p_3s_1/3},\\
       \min {(q_2,q_3)}>2/s_2  \quad  \text{and} \quad w_2^1(x_2) \in A_{\infty},\  w_2^2(x_2) \in A_{q_2s_2 /2},\  w_2^3(x_2) \in A_{q_3s_2/2}, \label{assumptionreduced2}
   \end{gather}
      then \eqref{assumption} maps $L^p(w_1^1\otimes w_2^1) \times L^{p_2}_{x_1}(w_1^2)(L^{q_2}_{x_2}(w_2^2))\times L^{p_3}_{x_1}(w_1^3)(L^{q_3}_{x_2}(w_2^3)) \to L^r(w_1\otimes w_2)$ for  ${1\over p}+ {1\over p_2}+{1\over p_3}={1\over p}+{1\over q_2}+{1\over q_3}={1\over r}$ with $0<r<\infty$, $1<p,p_2,p_3,q_2,q_3<\infty $,  where
     \begin{gather*}
        w_1(x_1)=(w_1^1)^{r/p} \cdot (w_1^2)^{r/p_2} \cdot (w_1^3)^{r/p_3} \\
        w_2(x_2)=(w_2^1)^{r/p} \cdot (w_2^2)^{r/q_2} \cdot (w_2^3)^{r/q_3}.
     \end{gather*}

    In particular, by taking $w_1^1=w_1^2=w_1^3=w_1\in A_{\min{(ps_1/3,p_2s_1/3,p_3s_1/3)}}$, $w_2^1=w_2^2=w_2^3=w_2\in A_{\min{(q_2s_2/2,q_3s_2/2)}}$, \eqref{assumption} maps $L^p(w_1\otimes w_2) \times L^{p_2}_{x_1}(w_1)(L^{q_2}_{x_2}(w_2))\times L^{p_3}_{x_1}(w_1)(L^{q_3}_{x_2}(w_2)) \to L^r(w_1\otimes w_2)$.
  \end{proposition}

  \begin{proof}
As in the proof Proposition \ref{tensor32}, we just need to apply an iteration argument and use the single-parameter weighted estimate.
    \begin{eqnarray*}
     & &\| T_1\left(f(\cdot,x_2),g_1(\cdot),h_1(\cdot)\right)(x_1) \cdot  T_2\left(g_2,h_2\right)(x_2)\|_{L^r(w_1\otimes w_2)}^r\\
     &=& \int \left| T_1\left(f(\cdot,x_2),g_1(\cdot),h_1(\cdot)\right)(x_1) \right|^r w_1(x_1)  \left|T_2\left(g_2,h_2\right)(x_2)\right|^r w_2(x_2)dx_1 dx_2 \\
      &=& \int \left(\int \left| T_1\left(f(\cdot,x_2),g_1(\cdot),h_1(\cdot)\right)(x_1)\right|^r w_1(x_1)   dx_1 \right)  \left|T_2\left(g_2,h_2\right)(x_2)\right|^r w_2(x) dx_2\\
      &\lesssim& \int \|f(\cdot,x_2)\|_{L^p_{x_1}(w_1^1)}^r \|g_1\|_{L^{p_2}_{x_1}(w_1^2)}^r \|h_1\|_{L^{p_3}_{x_1}(w_1^3)}^r  \left|T_2\left(g_2,h_2\right)(x_2) \right|^r w_2(x_2) dx_2 \\
      &=&  \int \|f(\cdot,x_2)\|_{L^p_{x_1}(w_1^1)}^r  \left|T_2\left(g_2,h_2\right)(x_2)\right|^r w_2(x_2) dx_2 \cdot \|g_1\|_{L^{p_2}_{x_1}(w_1^2)}^r\|h_1\|_{L^{p_3}_{x_1}(w_1^3)}^r \\
      &\lesssim& \left(\int \|f(\cdot,x_2)\|_{L^p_{x_1}(w_1^1)}^p w_2^1(x_2) dx_2\right)^{r\over p}  \left( \int \left|T_2\left(g_2,h_2\right)(x_2)\right|^{s_0} (w_2^2)^{s_0/q_2} \cdot (w_2^3)^{s_0/q_3} dx_2 \right)^{r\over s_0}\\
      & & \quad \cdot \|g_1\|_{L^{p_2}_{x_1}}^r\|h_1\|_{L^{p_3}_{x_1}}^r \\
      &\lesssim& \|f\|_{L^p(w_1^1\otimes w_2^1)}^r \|g_2\|_{L^{q_2}_{x_2}(w_2^2)}^r\|h_2\|_{L^{q_3}_{x_2}(w_2^3)}^r \|g_1\|_{L^{p_2}_{x_1}(w_1^2)}^r\|h_1\|_{L^{p_3}_{x_1}(w_1^3)}^r,
    \end{eqnarray*}
    where $1<p,p_2,p_3,q_2,q_3<\infty$, ${1\over p_2}+{1\over p_3}={1\over q_2}+{1\over q_3}={1\over s_0}$, and we just use the H\"older's inequality, the  weighted norm estimate of the classical one-parameter trilinear Fourier multiplier $T_1$ and bilinear multiplier $T_2$, i.e., under condition  \eqref{assumptionreduced1} there holds
    $$\|T_1 (f_1,f_2,f_3)\|_{L^r(w_1)}\lesssim \|f_1\|_{L^p(w_1^1)}\|g_1\|_{L^{p_2}(w_1^2)}\|h_1\|_{L^{p_3}(w_1^3)},$$
    and under condition \eqref{assumptionreduced2} there holds $$\|T_2(g_2,h_2)\|_{L^{s_0}((w_2^2)^{s_0/q_2} (w_2^3)^{s_0/q_3})}\lesssim \|g_2\|_{L^{q_2}(w_2^2)}\|h_2\|_{L^{q_3}(w_2^3)}. $$
    Such single-parameter estimates can be found, e.g., in \cite{chen2014hormander, FT}.

  \end{proof}

We note that weighted estimates for the classical bi-parameter and trilinear Fourier multipliers  were established in \cite{chen2014hormander}.
For the weighted mixed norm estimate for the bi-parameter and trilinear  operator $T_3$, we have the following
 \begin{proposition}
 \label{weightedT3}
 Let $g(x)=g_1(x_1)\otimes g_2(x_2) $, $h(x)=h_1(x_1)\otimes h_2(x_2) $ and  the symbol in $T_3$ satisfy the limited smoothness condition in the sense that for $3/2< s\leq 3$
    $$\sup_{j,k\in \z} \|m_{j,k}\|_{H^{s,s}}<\infty.$$
    Assume that
     \begin{gather}
    \min {(p,p_2,p_3)}>3/s\quad \text{and} \quad w_1^1(x_1) \in A_{ps/3},\  w_1^2(x_1) \in A_{p_2s/3},\  w_1^3(x_1) \in A_{p_3s/3}, \nonumber \\
   \min {(p,q_2,q_3)}>3/s \quad \text{and}\quad  w_2^1(x_2) \in A_{ps/3}, \ w_2^2(x_2) \in A_{q_2s/3},\  w_2^3(x_2) \in A_{q_3s/3}, \label{assumptionT3}
  \end{gather}
 then $T_3$  maps $L^p(w_1^1 \otimes w_2^1) \times L^{p_2}_{x_1}(w_1^2)(L^{q_2}_{x_2}(w_2^2))\times L^{p_3}_{x_1}(w_1^3)(L^{q_3}_{x_2}(w_2^3)) \to L^r(w_1\otimes w_2)$ for  ${1\over p}+ {1\over p_2}+{1\over p_3}={1\over p}+{1\over q_2}+{1\over q_3}={1\over r}$ with $0<r<\infty$, $1<p,p_2,p_3,q_2,q_3<\infty $,  where
     \begin{gather*}
        w_1(x_1)=(w_1^1)^{r/p} \cdot (w_1^2)^{r/p_2} \cdot (w_1^3)^{r/p_3} \\
        w_2(x_2)=(w_2^1)^{r/p} \cdot (w_2^2)^{r/q_2} \cdot (w_2^3)^{r/q_3}.
     \end{gather*}

    In particular, by taking
    $$w_1^1=w_1^2=w_1^3=w_1\in A_{\min{(ps/3,p_2s/3,p_3s/3)}},$$ $$w_2^1=w_2^2=w_2^3=w_2\in A_{\min{(ps/3,q_2s/3,q_3s/3)}},$$  $T_3$ maps $L^p(w_1\otimes w_2) \times L^{p_2}_{x_1}(w_1)(L^{q_2}_{x_2}(w_2))\times L^{p_3}_{x_1}(w_1)(L^{q_3}_{x_2}(w_2)) \to L^r(w_1\otimes w_2)$,

 \end{proposition}

\begin{proof}
As in the proof of Proposition \ref{ppt3}, we decompose the symbol $m$ by using the appropriate cutoff functions, and  obtain the corresponding different groups.

 We first consider the cases that belong to $I_1\otimes I_2$, and we take $I_1(a)\otimes I_2(b)$ for example.
 Note that by Lemma \ref{lemmaChen}, $w_1(x_1), w_2(x_2)\in   A_\infty$. Then  Lemma \ref{lemmaRuan} implies
\begin{eqnarray*}
  \|T_3(f,g,h)\|_{L^r(w_1\otimes w_2)}\lesssim \left\| \left\{\sum_{j,k} |\psi_1(D/2^j)\psi_2(D/2^k)T_3(f,g,h) |^2\right\}^{1\over 2}\right\|_{L^r(w_1\otimes w_2)}.
\end{eqnarray*}

Recall we set in Proposition \ref{ppt3}
\begin{eqnarray*}
  A_{j,k}&:=& |\psi_1(D/2^j)\psi_2(D/2^k)T_3(f,g,h) | \\
  &\lesssim& \|m_{j,k}\|_{H^{s,s}}\left(M_s(|\tilde \psi_j(D)f|^t)(x_1,x_2)\right)^{1\over t} \left(M_s(|\tilde \psi_k(D) g|^t)(x_1,x_2)\right)^{1\over t} \\
   & & \qquad \quad \left(M_s(|h|^t)(x_1,x_2)\right)^{1\over t},
\end{eqnarray*}
where we need $\max{(1,3/s)}<t<2$, which now is actually $3/s<t<2$, since we have assumed $s\leq 3$. Using H\"older's inequality, we have

\begin{eqnarray*}
  & &\|T_3(f,g,h)\|_{L^r(w_1\otimes w_2)} \\
  & \lesssim& \left\| \left\{\sum_{j,k} |\psi_1(D/2^j)\psi_2(D/2^k)T_3(f,g,h) |^2\right\}^{1\over 2}\right\|_{L^r(w_1\otimes w_2)} \\
  &\lesssim& \sup_{j,k} \|m_{j,k}\|_{H^{s,s}} \Bigg\| \Big[ \sum_{j} \left( M_s(|\tilde \psi_j(D)f|^t)(x_1,x_2)\right)^{2\over t} \Big]^{1\over 2} \left( M(|g_1|^t)(x_1)\right)^{1\over t} \\
 & & \quad  \cdot \Big[ \sum_k  \left(M(| \tilde \psi_k(D) g_2|^t)(x_2)\right)^{2\over t} \Big]^{1\over 2}  \left( M(|  h_1|^t)(x_1)\right)^{1\over t} \left( M(|h_2|^t)(x_2) \right)^{1\over t} \Bigg\|_{L^r(w_1\otimes w_2)} \\
  &\lesssim& \sup_{j,k} \|m_{j,k}\|_{H^{s,s}} \Bigg\| \Big[ \sum_{j} \left( M_s(|\tilde \psi_j(D)f|^t)(x_1,x_2)\right)^{2\over t} \Big]^{1\over 2} \Bigg\|_{L^p(w_1^1\otimes w_2^1)} \\
 & & \qquad \qquad  \qquad \Bigg\|\left( M(|g_1|^t)(x_1)\right)^{1\over t} \Big[ \sum_k  \left(M(| \tilde \psi_k(D) g_2|^t)(x_2)\right)^{2\over t} \Big]^{1\over 2}  \\
 & & \qquad \qquad \qquad   \cdot \left( M(|  h_1|^t)(x_1)\right)^{1\over t} \left( M(|h_2|^t)(x_2) \right)^{1\over t}  \Bigg\|_{L^{s_0}((w_1^2)^{s_0/p_2} (w_1^3)^{s_0/p_3}\otimes (w_2^2)^{s_0/q_2} (w_2^3)^{s_0/q_3})},
 \end{eqnarray*}
 where $1/p+1/s_0=1/r$, i.e., $1/p_2+1/q_2=1/p_3+1/q_3=1/s_0$. Then using the tensor product setting, the above can be estimated by

 \begin{eqnarray*}
   &=& \sup_{j,k} \|m_{j,k}\|_{H^{s,s}} \Bigg\| \Big[ \sum_{j} \left( M_s(|\tilde \psi_j(D)f|^t)(x)\right)^{2\over t} \Big]^{1\over 2} \Bigg\|_{L^p(w_1^1\otimes w_2^1)} \\
 & & \quad \Big\|\left( M(|g_1|^t)(x_1)\right)^{1\over t} \left( M(|  h_1|^t)(x_1)\right)^{1\over t} \Big\|_{L^{s_0}((w_1^2)^{s_0/p_2} (w_1^3)^{s_0/p_3})}   \\
 & & \quad \Bigg\| \Big[ \sum_k  \left(M(| \tilde \psi_k(D) g_2|^t)(x_2)\right)^{2\over t} \Big]^{1\over 2}  \left( M(|h_2|^t)(x_2) \right)^{1\over t} \Bigg\|_{L^{s_0}((w_2^2)^{s_0/q_2} (w_2^3)^{s_0/q_3})}.
 \end{eqnarray*}

Again by H\"older's inequality, we have
 \begin{eqnarray*}
 &\lesssim& \sup_{j,k} \|m_{j,k}\|_{H^{s,s}} \Bigg\| \Big[ \sum_{j} \left( M_s(|\tilde \psi_j(D)f|^t)(x)\right)^{2\over t} \Big]^{1\over 2} \Bigg\|_{L^p(w_1^1\otimes w_2^1)}   \\
 & &  \Big\|\big( M(|g_1|^t)(x_1)\big)^{1\over t}  \Big\|_{L^{p_2}(w_1^2)} \Bigg\|\Big[ \sum_k  \left(M(| \tilde \psi_k(D) g_2|^t)(x_2)\right)^{2\over t} \Big]^{1\over 2}  \Bigg\|_{L^{q_2}(w_2^2)}  \\
 & & \quad   \Big\|\left( M(|  h_1|^t)(x_1)\right)^{1\over t} \Big\|_{L^{p_3}(w_1^3)} \Big\|\left( M(|h_2|^t)(x_2) \right)^{1\over t} \Big\|_{L^{q_3}(w_2^3)} \\
  &=&\sup_{j,k} \|m_{j,k}\|_{H^{s,s}} \Bigg\| \Big[ \sum_{j} \left( M_s(|\tilde \psi_j(D)f|^t)(x)\right)^{2\over t} \Big]^{t\over 2} \Bigg\|_{L^{p/t}(w_1^1\otimes w_2^1)}^{1/t}   \\
 & & \Big\| M(|g_1|^t)(x_1) \Big\|_{L^{p_2/t}(w_1^2)}^{1/t}    \Bigg\|\Big[ \sum_k  \left(M(| \tilde \psi_k(D) g_2|^t)(x_2)\right)^{2\over t} \Big]^{t\over 2}  \Bigg\|_{L^{q_2/t}(w_2^2)}^{1/t}  \\
 & & \quad   \Big\| M(|  h_1|^t)(x_1) \Big\|_{L^{p_3/t}(w_1^3)}^{1/t}  \Big\|M(|h_2|^t)(x_2) \Big\|_{L^{q_3/t}(w_2^3)}^{1/t}  \\
 &\lesssim & \sup_{j,k} \|m_{j,k}\|_{H^{s,s}} \|f\|_{L^{p}(w_1^1\otimes w_2^1)}\|g_1\|_{L^{p_2}(w_1^2)}\|g_2\|_{L^{q_2}(w_2^2)}\|h_1\|_{L^{p_3}(w_1^3)}\|h_2\|_{L^{q_3}(w_2^3)}.
\end{eqnarray*}
In order to get the last estimate, it needs that $2,p,p_2,p_3,q_2,q_3>t$, and
\begin{equation}
\label{weightcd}
w_1^1\otimes w_2^1\in A_{p/t}(\n \times \n),\ \ w_1^2 \in A_{p_2/t}, \ \  w_1^3 \in A_{p_3/t}, \ \ w_2^2 \in A_{q_2/t},\ \ w_2^3 \in A_{q_3/t}.
\end{equation}
To see why the above can be achieved, first consider the  stated assumptions for the weights
\begin{gather}
\min {(p,p_2,p_3,q_2,q_3)}>3/s,\   3/2<s\leq 3, \label{explaincd1}\quad \text{and} \\
 w_1^1,w_2^1 \in A_{ps/3},\ w_1^2 \in A_{p_2s/3},\ w_1^3 \in A_{p_3s/3}, \ w_2^1 \in A_{ps/3}, \ w_2^2 \in A_{q_2s/3}, \ w_2^3 \in A_{q_3s/3}.
 \label{explaincd2}
\end{gather}
The  condition \eqref{explaincd1} means it's possible to choose some $t$ with $3/s<t< \min{(2,p,p_2,p_3,q_2,q_3)}$, i.e., $1/\min{(2,p,p_2,p_3,q_2,q_3)}<1/t<s/3$. For condition \eqref{explaincd2}, Lemma  \ref{lemmaGR} implies there exists $\tau_i^j<s/3$ ($1\leq i,j \leq 3$) such that
$$w_1^1 \in A_{p\tau_1^1}, \ w_2^1 \in A_{p\tau_2^1}, \ w_1^2 \in A_{p_2 \tau_1^2},\ w_1^3 \in A_{p_3\tau_1^3}, \ w_2^1 \in A_{p \tau_2^1}, \ w_2^2 \in A_{q_2 \tau_2^2}, \ w_2^3 \in A_{q_3\tau_2^3}.$$
Now we pick $t$ with $1/t$ sufficiently close to $s/3$, such that  $\tau_i^j<1/t$ for $1\leq i,j \leq 3$. Then it follows that
\begin{gather*}
 w_1^1, w_1^2 \in A_{p/t}, \ \ w_1^2 \in A_{p_2/t}, \ \  w_1^3 \in A_{p_3/t}, \ \ w_2^2 \in A_{q_2/t},\ \ w_2^3 \in A_{q_3/t},
 \end{gather*}
 which implies \eqref{weightcd}.

 \medskip

 Using the same argument, we can deal with the cases in other groups. Note that the above verification for condition \eqref{weightcd} works actually for all the cases, thus  we will not repeat it in the rest of the proof. For the operators in $II_1\times II_2$, we consider $II_1{(a)}\times II_2{(b)}$ for example.  Recall we can write

\begin{eqnarray*}
 & & |T_3(f,g,h)(x_1,x_2)|\\
    &\lesssim&   \sup_{j,k}\|m_{j,k}\|_{H^{s,s}}  \Bigg\{\sum_{j} \left[M_s \left(|\tilde \psi_j(D) f|^t\right)\right]^{2\over t} \Bigg\}^{1\over 2} \Bigg\{\sum_{j,k}\left[M_s \left(| \tilde \psi_j(D) \psi_k(D)  g|^t\right)\right]^{2\over t} \Bigg\}^{1\over 2}  \\
    & & \quad \cdot \Bigg\{ \sum_k \left[M_s \left(|\tilde \psi_k(D) h|^t\right)\right]^{2\over t}\Bigg\}^{1\over 2},
\end{eqnarray*}
where the Cauchy-Schwartz inequality is used.  Then the H\"older's inequality gives
\begin{eqnarray*}
& &  \| T_3(f,g,h)\|_{L^r(w_1\otimes w_2)} \\
&\lesssim&   \sup_{j,k}\|m_{j,k}\|_{H^{s,s}}  \Bigg\| \Bigg\{\sum_{j} \left[M_s \left(|\tilde \psi_j(D) f|^t\right)\right]^{2\over t} \Bigg\}^{1\over 2}  \Bigg\|_{L^{p}(w_1^1\otimes w_2^1)} \Bigg\|\Bigg\{\sum_{j,k}\left[M_s \left(| \tilde \psi_j(D) \psi_k(D)  g|^t\right)\right]^{2\over t} \Bigg\}^{1\over 2}  \\
& & \cdot \Bigg\{ \sum_k \left[M_s \left(|\tilde \psi_k(D) h|^t\right)\right]^{2\over t}\Bigg\}^{1\over 2} \Bigg\|_{L^{s_0}((w_1^2)^{s_0/p_2}(w_1^3)^{s_0/p_3}\otimes (w_2^2)^{s_0/q_2} (w_2^3)^{s_0/q_3})} \\
&\lesssim &\sup_{j,k}\|m_{j,k}\|_{H^{s,s}}  \Bigg\| \Bigg\{\sum_{j} \left[M \left(|\tilde \psi_j(D) f|^t\right)\right]^{2\over t} \Bigg\}^{1\over 2}  \Bigg\|_{L^{p}(w_1^1\otimes w_2^1)} \\
& &\Bigg\|\Bigg\{\sum_{j}\left[M \left(| \tilde \psi_j(D) g_1|^t\right)\right]^{2\over t} \Bigg\}^{1\over 2}  \Bigg\|_{L^{p_2}(w_1^2)} \Bigg\|\left[M \left(|h_1|^t\right)\right]^{1\over t} \Bigg\|_{L^{p_3}(w_1^3)} \\
& &\Bigg\|\Bigg\{\sum_{k}\left[M \left(| \tilde \psi_k(D)  g_2|^t\right)\right]^{2\over t} \Bigg\}^{1\over 2} \Bigg\|_{L^{q_2}(w_2^2)} \Bigg\| \Bigg\{ \sum_k \left[M \left(|\tilde \psi_k(D) h_2|^t\right)\right]^{2\over t}\Bigg\}^{1\over 2} \Bigg\|_{L^{q_3}(w_2^3)}. \\
&\lesssim & \sup_{j,k} \|m_{j,k}\|_{H^{s,s}} \|f\|_{L^{p}(w_1^1\otimes w_2^1)}\|g_1\|_{L^{p_2}(w_1^2)}\|g_2\|_{L^{q_2}(w_2^2)}\|h_1\|_{L^{p_3}(w_1^3)}\|h_2\|_{L^{q_3}(w_2^3)}.
\end{eqnarray*}

\medskip

Then for the situations in $I_1 \times II_2$. We consider $I_1(a) \times II_2(c)$ for an example. Recall
\begin{eqnarray*}
B_{j}&:=&\psi_1(D/2^j)T_3(f,g,h)(x_1,x_2)\\
   &\lesssim& \sup_{k}\|m_{j,k}\|_{H^{s,s}} \Big\{ \sum_k \left(M_s(|\tilde \psi_j(D)  \psi_k(D)f|^t)(x_1,x_2)\right)^{2\over t} \Big\}^{1\over 2}\left(M_s(| g|^t)(x_1,x_2)\right)^{1\over t} \\
   & & \qquad \quad \quad \cdot   \left\{\sum_k \left(M_s(|\tilde \psi_k(D) h|^t)(x_1,x_2)\right)^{2\over t} \right\}^{1\over 2},
\end{eqnarray*}
Then
\begin{eqnarray*}
  & &\|T_3(f,g,h)\|_{L^r(w_1\otimes w_2)} \\
  & \lesssim& \left\| \left\{\sum_{j} |\psi_1(D/2^j) T_3(f,g,h) |^2\right\}^{1\over 2}\right\|_{L^r(w_1\otimes w_2)} \\
 &\lesssim& \sup_{j,k} \|m_{j,k}\|_{H^{s,s}} \Bigg\| \Big\{ \sum_{j,k}\left( M_s(|\tilde \psi_j(D) \psi_k(D)f|^t)(x_1,x_2)  \right)^{2\over t}\Big\}^{1\over 2} \Bigg\|_{L^p(w_1\otimes w_2)} \\
 & & \quad \Big\|\left(M(| g_1|^t)(x_1)\right)^{1\over t}\Big\|_{L^{p_2}(w_1^2)} \cdot \Big\|  \left(M(| h_1|^t)(x_1)\right)^{1\over t} \Big\|_{L^{p_3}(w_1^3)}\\
 & & \quad \Big\|\left(M(| g_2|^t)(x_2)\right)^{1\over t} \big\|_{L^{q_2}(w_2^2)}\cdot \left\| \left\{\sum_k \left(M(|\tilde \psi_k(D) h_2|^t)(x_2)\right)^{2\over t} \right\}^{1\over 2} \right\|_{L^{q_3}(w_2^3)} \\
 &\lesssim&\sup_{j,k} \|m_{j,k}\|_{H^{s,s}} \|f\|_{L^{p}(w_1\otimes w_2)}\|g_1\|_{L^{p_2}(w_1^2)}\|g_2\|_{L^{q_2}(w_2^2)}\|h_1\|_{L^{p_3}(w_1^3)}\|h_2\|_{L^{q_3}(w_2^3)}.
\end{eqnarray*}

\end{proof}

 We are now ready to achieve the main goal of this section, namely, establishing the weighted mixed norm estimate for operator \eqref{biop34}. Since we are less concerned with the limited smoothness of the  H\"ormander condition of the symbols in the above two propositions, one can simply take $s_1=3$ and $s_2=2$ in Proposition \ref{weightedreduced}, and $s=3$ in Proposition \ref{weightedT3}, then  Theorem \ref{goalweighted1} follows immediately.

\end{document}